\documentclass[EJP]{ejpecp} % replace ECP by EJP if needed.  

\usepackage{mathrsfs}

\usepackage{pgfplots}
\usepackage{tikz-cd}
\usetikzlibrary{arrows,automata}

\SHORTTITLE{Semigroups for 1D Schr\"odinger Operators with Gaussian Noise} 

\TITLE{Semigroups for One-Dimensional Schr\"odinger Operators with Multiplicative Gaussian Noise\thanks{This research was partially supported by an
NSERC doctoral fellowship and a Gordon Y.~S.~Wu fellowship.}} % \thanks is optional. Insert line breaks with \\

\AUTHORS{%
  Pierre~Yves~Gaudreau~Lamarre\footnote{University of Chicago, Chicago, IL\space\space60637, USA.
    \EMAIL{pyjgl@uchicago.edu}}}%AUTHORS
 % Separate items with ;

\AMSSUBJ{47H40; 47D08; 60J55} % Edit. Separate items with ;
\SUBMITTED{October 31, 2019} % Edit.
\ACCEPTED{June 7, 2021} % Edit.

\VOLUME{26}
\YEAR{2021}
\PAPERNUM{107}
\DOI{10.1214/21-EJP654}

\ABSTRACT{Let $ H:=-\tfrac12\De+V$ be a one-dimensional continuum Schr\"odinger operator.
Consider ${\hat H}:= H+\xi$, where $\xi$ is a translation invariant Gaussian noise.
Under some assumptions on $\xi$, we prove that
if $V$ is locally integrable, bounded below, and grows faster than $\log$
at infinity, then the semigroup $\mr e^{-t {\hat H}}$ is trace class and admits a probabilistic representation
via a Feynman-Kac formula. Our result applies to operators acting on the whole line $\mbb R$,
the half line $(0,\infty)$, or a bounded interval $(0,b)$, with a variety of boundary conditions.
Our method of proof consists of a comprehensive generalization of techniques recently
developed in the random matrix theory literature to tackle this problem in the special case where
${\hat H}$ is the stochastic Airy operator.}

\newcommand\al{\alpha}
\newcommand\be{\beta}
\newcommand\dd{\mathrm d}
\newcommand\De{\Delta}
\newcommand\de{\delta}
\newcommand\deq{\stackrel{\mathrm d}{=}}
\newcommand\eps{\varepsilon}

\newcommand\ga{\gamma}
\newcommand\ka{\kappa}
\newcommand\La{\Lambda}
\newcommand\la{\lambda}
\newcommand\Om{\Omega}
\newcommand\om{\omega}
\newcommand\si{\sigma}
\renewcommand\d{~\mathrm d}
\renewcommand\phi{\varphi}
\renewcommand\rho{\varrho}

\newcommand\bs{\boldsymbol}
\newcommand\mbb{\mathbb}
\newcommand\mbf{\mathbf}
\newcommand\mc{\mathcal}
\newcommand\mf{\mathfrak}
\newcommand\mr{\mathrm}
\newcommand\ms{\mathscr}
\begin{document}

%%%%%%%%%%%%%%%%%%%%%%%%%%%%%%%%%%%%%%%%%%%%%%%%%%%%%%%%%%%%%%%%%%%
%%                                                               %%
%% No need for \maketitle.                                       %%
%%                                                               %%
%%%%%%%%%%%%%%%%%%%%%%%%%%%%%%%%%%%%%%%%%%%%%%%%%%%%%%%%%%%%%%%%%%%

%%%%%%%%%%%%%%%%%%%%%%%%%%%%%%%%%%%%%%%%%%%%%%%%%%%%%%%%%%%%%%%%%%%
%%                                                               %%
%% Please replace what follows by the body of your article       %%
%% (up to the bibliography):                                     %%
%%                                                               %%
%%%%%%%%%%%%%%%%%%%%%%%%%%%%%%%%%%%%%%%%%%%%%%%%%%%%%%%%%%%%%%%%%%%

%

%

\section{Introduction}

Let $I\subset\mbb R$ be an open interval (possibly unbounded) and $V:I\to\mbb R$ be a function.
Let $ H:=-\tfrac12\De+V$ denote a Schr\"odinger operator with potential $V$
acting on functions $f:I\to\mbb R$ with
prescribed boundary conditions when $I$ has a boundary. In this paper, we are interested
in random operators of the form
\begin{align}\label{Equation: Perturbed Operator}
{\hat H}:= H+\xi,
\end{align}
where $\xi$ is a stationary Gaussian noise on $\mbb R$.
Informally, we think of $\xi$ as a centered Gaussian process on
$\mbb R$ with a covariance of the form $\mbf E[\xi(x)\xi(y)]=\ga(x-y)$,
where $\ga$ is an even almost-everywhere-defined function
or Schwartz distribution. In many
cases that we consider, $\ga$ is not an actual function, and thus
$\xi$ cannot be defined as a random function on $\mbb R$; in such
cases $\xi$ can be defined rigorously as a random Schwarz distribution,
i.e., a centered Gaussian process on an appropriate function space with covariance
\[\mbf E\big[\xi(f)\xi(g)\big]=\int_\mbb R f(x)\ga(x-y)g(y)\d x\dd y,\qquad f,g:\mbb R\to\mbb R.\]

Among the most powerful tools used to study Schr\"odinger operators are
their semigroups (e.g., \cite{Simon}); we recall that the semigroup generated by $ H$ is
the family of operators formally defined as $\mr e^{-t H}$ for $t>0$. Provided the potentials under consideration
are sufficiently well behaved, there is a remarkable connection between Schr\"odinger semigroups
and the theory of stochastic processes that can be expressed in the form of the Feynman-Kac formula
(e.g., \cite[Theorem A.2.7]{Simon}): Assuming $I=\mbb R$ for simplicity, for every $f\in L^2(\mbb R)$,
$t>0$, and $x\in\mbb R$, one has
\begin{align}
\label{Equation: Deterministic Feynman-Kac}
\mr e^{-t  H}f(x)
=
\mbf E^x\left[\exp\left(-\int_0^t V\big(B(s)\big)\d s\right)
f\big(B(t)\big)\right]
\end{align}
where $B$ is a Brownian motion and $\mbf E^x$ signifies that
we are taking the expected value with respect to $B$ conditioned on the starting point $B(0)=x$. Apart
from the obvious benefit of making Schr\"odinger semigroups amenable to
probabilistic methods, we note that the Feynman-Kac formula can in fact form the
basis of the definition of $ H$ itself, as done, for instance, in \cite{McKean}.

Our purpose in this paper is to lay out the foundations of a general semigroup theory (or Feynman-Kac formulas)
for random Schr\"odinger operators of the form \eqref{Equation: Perturbed Operator}.
We note that, since we consider very irregular noises (i.e., in general $\xi$ is not a proper function that can be evaluated
at points in $\mbb R$), this undertaking is not a direct application or a trivial extension
of the classical theory; see Section \ref{Section: Intro - Overview} for more details.
As a first step in this program, we show that a variety of tools recently developed in the random matrix
theory literature (e.g., \cite{BloemendalVirag,GaudreauLamarreShkolnikov,GorinShkolnikov,
KrishnapurRiderVirag,Minami,RamirezRiderVirag}) to tackle special cases of this problem can be suitably
extended to a rather general setting. The main restriction of our assumptions is that we consider cases
where the semigroup $\mr e^{-t{\hat H}}$ is trace class, which implies in particular that ${\hat H}$ must
have a purely discrete spectrum.

This paper is organized as follows. In the remainder of this introduction, we present a brief outline of our main
results and discuss some motivations and applications. In Section \ref{Section: Main Results},
we give a precise statement of our results (our main result is {\bf Theorem \ref{Theorem: Semigroup}},
and our second main result is {\bf Proposition \ref{Proposition: Operator}}). In Section \ref{Section: Outline},
we provide an outline of the proof of our main results. Finally, in Sections \ref{Section: Operator} and
\ref{Section: Semigroup}, we go over the technical details of the proof of our results.

\subsection{Overview of Results}\label{Section: Intro - Overview}

As mentioned earlier in this introduction,
much of the challenge involved in our program comes from the fact that, in general, Gaussian noises are
Schwartz distributions. This creates two main technical obstacles.

The first obstacle is that it is not immediately obvious how to define the operator ${\hat H}$.
Indeed, if we interpret $\xi$ as being part of the potential of ${\hat H}$, then the action
\[{\hat H}f
\ ``=" \
-\tfrac12 f''+(V+\xi)f\]
of the operator on a function $f$ includes the $``$pointwise product$"$ $\xi f$,
which is not well defined if $\xi$ cannot be evaluated at single points in $\mbb R$. The second obstacle
comes from the definition of $\mr e^{-t{\hat H}}$. Arguably, the most natural guess for this
semigroup would be to add $\xi$ to the potential in the usual Feynman-Kac formula
\eqref{Equation: Deterministic Feynman-Kac}, which yields
\begin{align}\label{Equation: Random Guess Feynman-Kac}
\mr e^{-t {\hat H}}f(x)
\ ``=" \
\mbf E^x\left[\exp\left(-\int_0^t V\big(B(s)\big)+\xi\big(B(s)\big)\d s\right)
f\big(B(t)\big)\right].
\end{align}
However, this again requires the ability to evaluate $\xi$ at every point.

The key to overcoming these obstacles is to interpret $\xi$ as the distributional
derivative of an actual Gaussian process.
More precisely, let $\Xi$ be the Gaussian process on $\mbb R$
defined as
\begin{align}
\label{Equation: Anti-Derivative Process}
\Xi(x):=\begin{cases}
\xi(\mbf 1_{[0,x)}),&x\geq0\\
\xi(-\mbf 1_{[x,0)}),&x\leq0.
\end{cases}
\end{align}
Assuming $\Xi$ has a version with
measurable sample paths (and we neglect boundary values for simplicity),
a formal integration by parts yields
\[\xi(f)=\langle f, \Xi'\rangle:=-\langle f', \Xi\rangle.\]
Following this line of thought, we may then settle on a $``$weak$"$ definition of ${\hat H}$ through the  form
\begin{align}\label{Equation: Intro Quadratic Form}
\langle f,{\hat H}g\rangle:=\langle f, Hg\rangle+\xi(fg)=\langle f, Hg\rangle-\langle f'g+fg', \Xi\rangle.
\end{align}
We note that this type of definition for ${\hat H}$ has previously appeared in the literature (e.g.,
\cite{BloemendalVirag,FukushimaNakao,Minami,RamirezRiderVirag}) for various potentials $V$ on the half line $I=(0,\infty)$
as well as $V=0$ on a bounded interval $I=(0,L)$ ($L>0$).
We also note an alternative approach outlined by Bloemendal in \cite[Appendix A]{BloemendalThesis} that
allows one (in principle) to recast ${\hat H}$ as the classical Sturm-Liouville operator
\begin{align}
\label{Equation: Bloemendal}
S f=-w^{-1}(\tfrac w2 f')'+(V-2\Xi^2)f,\qquad\text{where }w(x):=\exp\left(4\int_0^x\Xi(y)\d y\right)
\end{align}
through a suitable Hilbert space isomorphism.
Our first result (namely, {\bf Proposition \ref{Proposition: Operator}}) is an extension of these statements:
We provide a very succinct proof of the fact that, under fairly general conditions on $\Xi$ and $V$, the form \eqref{Equation: Intro Quadratic Form}
corresponds to a unique self-adjoint operator with compact resolvent, including when $I$ is the whole real line
or a bounded interval with a nonzero potential.

The interpretation $\xi=\Xi'$ also leads to a natural candidate for the semigroup generated by ${\hat H}$:
Let $L^a_t(B)$ ($a\in\mbb R$, $t\geq0$) be the local time process of the  Brownian motion $B$
so that for any measurable function $f$, we have
\[\int_0^t f\big(B(s)\big)\d s=\int_{\mbb R}L^a_t(B) f(a)\d a.\]
Assuming a stochastic integral with respect to $\Xi$ can
meaningfully be defined,
we may then interpret the problematic term in $\mr e^{-t{\hat H}}$'s intuitive derivation
\eqref{Equation: Random Guess Feynman-Kac} thusly:
\[\int_0^t V\big(B(s)\big)+\xi\big(B(s)\big)\d s:=\int_{\mbb R} L^a_t(B)\d Q(a),\]
where $Q$ is the process $\dd Q(x)=V(x)\dd x+\dd \Xi(x)$, which we assume to be
independent of $B$. In the case where $I=\mbb R$, for example, this suggests that
\begin{align}\label{Equation: Intro Random Feynman-Kac}
\mr e^{-t{\hat H}}f(x)=\mbf E^x\left[\exp\left(-\int_{\mathbb R}L^a_t(B)\d Q(a)\right)
f\big(B(t)\big)\right],
\end{align}
where $\mbf E^x$ now denotes the conditional expectation of $\big(B|B(0)=x\big)$ given $\Xi$.
This type of random semigroup has appeared in \cite{GaudreauLamarreShkolnikov,
GorinShkolnikov} in the special case where  $I$ is the positive half line $(0,\infty)$, $V(x)=x$, and
$\Xi$ is a Brownian motion (so that $\xi$ is a Gaussian white noise;
see Example \ref{Example: Four Noises} for more details).
Our second and main result (namely, {\bf Theorem \ref{Theorem: Semigroup}}) provides general sufficient conditions
under which a Feynman-Kac formula of the form \eqref{Equation: Intro Random Feynman-Kac}
holds (we refer to \eqref{Equation: Random Semigroup} for a statement of our Feynman-Kac
formula when $I$ is the half line or a bounded interval).
This result can be seen as a comprehensive
generalization of
\cite[Proposition 1.8 (a)]{GaudreauLamarreShkolnikov} and \cite[Corollary 2.2]{GorinShkolnikov}.
We refer to Section \ref{Section: Semigroup Outline} for a detailed exposition of our method of proof.

One interesting consequence of Theorem \ref{Theorem: Semigroup} is the following connection
between the random functional \eqref{Equation: Intro Random Feynman-Kac} and the spectrum of ${\hat H}$:
Let $\la_1({\hat H})\leq\la_2({\hat H})\leq\cdots$ be the eigenvalues of ${\hat H}$ 
and $\psi_1({\hat H}),\psi_2({\hat H}),\ldots$ be the associated eigenfunctions,
which are defined by the variational principle (i.e., Courant-Fischer)
associated with the form \eqref{Equation: Intro Quadratic Form}.
By Theorem \ref{Theorem: Semigroup},
in many cases the spectral expansion
\[\mr e^{-t\hat H}f=\sum_{k=1}^\infty\mr e^{-t\la_k({\hat H})}\langle\psi_k({\hat H}),f\rangle\psi_k({\hat H}),
\qquad f\in L^2(\mbb R)\]
admits an explicit probabilistic representation of the form \eqref{Equation: Intro Random Feynman-Kac}.
We expect this connection to be fruitful in two directions.

On the one hand,
a good understanding of ${\hat H}$'s spectrum could be used to study the geometric properties of the function
$u(t,x):=\mr e^{-t{\hat H}}f(x)$, which we may interpret as the solution of the SPDE with multiplicative noise
\[\partial_t u=-(Hu+\xi u),\qquad u(0,x)=f(x).\]
We refer to Section \ref{Section: Anderson Hamiltonian} below for
more motivation in this direction.

On the other hand, the Feynman-Kac formula
can be used to study the properties of the eigenvalues and
eigenfunctions of $\hat H$ (we refer to \cite{Simon} for
classical examples of this involving the deterministic operator $H$).
In particular, our Feynman-Kac formula provides
a means of computing the $``$Laplace transforms$"$
\begin{align}
\label{Equation: Laplace Transform}
\mbf E\left[\prod_{i=1}^\ell\sum_{k=1}^\infty\mr e^{-t_i\la_k(\hat H)}\right]
=\mbf E\left[\prod_{i=1}^\ell\mr{Tr}\big[\mr e^{-t_i \hat H}\big]\right],
\qquad t_1,\ldots,t_\ell>0,
\end{align}
which characterize the distribution of $\hat H$'s eigenvalues.
In Sections \ref{Section: Operator Limits} and \ref{Section: Number Rigidity},
we discuss
how the ability to compute \eqref{Equation: Laplace Transform} has led to applications
in the study of operator limits of random matrices and the occurence
of number rigidity in the spectrum of general random Schr\"odinger operators.

\subsection{Motivating Examples and Applications}

\subsubsection{The Anderson Hamiltonian and Parabolic Anderson Model}\label{Section: Anderson Hamiltonian}

The earliest occurrences of an operator of the form \eqref{Equation: Perturbed Operator} in the literature appear to be \cite{FrischLloyd,Halperin}.
The operator that is considered therein is the Anderson Hamiltonian, defined as $ A:=-\De+\xi$,
where $\xi$ is a Gaussian white noise.
The first mathematically rigorous study of this object
appeared in \cite{FukushimaNakao}. Following this, there have been several investigations of $ A$'s spectral
properties \cite{CambroneroMcKean,CambroneroRider,McKeanAH}, culminating in a recent article of Dumaz and Labb\'e
\cite{DumazLabbe}, which provides a comprehensive description of eigenfunction localization and eigenvalue Poisson
statistics in the case where $A$ acts on $I=(0,L)$ for large $L$.

In this context, the Feynman-Kac formula proved in this paper in the case ${\hat H}=A$
creates a rigorous connection between the study of
localization in the Anderson Hamiltonian and the study of intermittency for large times in the parabolic
Anderson model with continuous noise (c.f., \cite[(5) and (6)]{DumazLabbe}
and \cite[Sections 2.2.3--2.2.4]{KonigBook}). We recall that the parabolic Anderson model is the
SPDE
\[\partial_t u(t,x)=\De u(t,x)+\xi u(t,x),\qquad u(0,x)=u_0(x)\]
or, equivalently, $u(t,x):=\mr e^{t(\De+\xi)}u_0(x)$.
Although several previous works have featured Feynman-Kac-type
formulas for the continuum Anderson Hamiltonian or the parabolic Anderson model in one dimension (e.g., \cite[Sections 3--4 and Lemma A.1]{Chen14}
or \cite[Section 3]{HuEtAl}), ours appears to be the first to make an explicit connection
between $A$'s full spectrum and a Feynman-Kac functional of the form
\eqref{Equation: Intro Random Feynman-Kac}.

\subsubsection{Operator Limits of Random Matrices}\label{Section: Operator Limits}

One of the most widely studied example of an operator of the form \eqref{Equation: Perturbed Operator}
is the stochastic Airy operator:
\begin{align}
\label{Equation: SAO}
A_\be:=-\De+x+\xi_\be,\qquad \be>0,
\end{align}
where $\xi_\be$ is a Gaussian white noise with variance $4/\be$,
and $A_\be$ acts on $I=(0,\infty)$ with Dirichlet or Robin boundary condition at the origin.
The interest of studying this operator comes from the fact that its spectrum captures
the asymptotic edge fluctuations of a large class of random matrices and $\be$-ensembles.
This was first observed by Edelman and Sutton in \cite{EdelmanSutton}
and is based on the tridiagonal models of Dumitriu and Edelman \cite{DumitriuEdelman}. The connection
was later rigorously established by Ram\'irez, Rider, and Vir\'ag \cite{RamirezRiderVirag}, and these
developments gave rise to a now very extensive literature concerning operator limits of random matrices,
in which general operators of the form \eqref{Equation: Perturbed Operator} arise as the limits
of a large class of random tridiagonal matrices.
We refer to \cite{Virag} and references therein for a somewhat recent survey.

In \cite{GorinShkolnikov}, Gorin and Shkolnikov introduced an alternative method
of studying operator limits of random matrices by proving that
large powers of generalized Gaussian $\be$-ensembles admit an operator limit of the form
\eqref{Equation: Intro Random Feynman-Kac} (see \cite[(2.4)]{GorinShkolnikov}).
These results were later extended to rank 1 additive perturbations of Gaussian $\be$-ensembles in \cite{GaudreauLamarreShkolnikov}.
Since the Gaussian $\be$-ensembles
converge to the stochastic Airy operator, these results imply a Feynman-Kac
formula of the form \eqref{Equation: Intro Random Feynman-Kac} for $\mr e^{-t A_\be/2}$.
This new Feynman-Kac formula was then used to study the eigenvalues of $A_\be$
(see \cite[Corollary 2.3 and Proposition 2.6]{GorinShkolnikov} and
\cite[Theorem 1.11 and Corollary 1.13]{GaudreauLamarreShkolnikov}).

In this context, our paper can be viewed as providing a streamlined and unified treatment of trace class
semigroups generated by general operators of the form \eqref{Equation: Perturbed Operator}.
In \cite{GaudreauLamarre}, this more general setting is used to extend
the operator limit results in \cite{GaudreauLamarreShkolnikov,GorinShkolnikov}
to much more general random tridiagonal matrices, including some non-symmetric matrices that could not be treated by any previous method.

\subsubsection{Number Rigidity in Random Schr\"odinger Operators}\label{Section: Number Rigidity}

A point process is number rigid if the number of points inside any bounded set is determined by
the configuration of points outside that set. The earliest proof of number rigidity
appears to be the work of Aizenman and Martin in \cite{AizenmanMartin}. More recently,
there has been a notable increase of interest in this
property stemming from the work of Ghosh and Peres \cite{GhoshPeres}. Therein,
it is proved that the zero set of the planar Gaussian analytic function and the Ginibre process are number rigid.
Since then, number rigidity has been shown to be connected
to several other interesting properties of point processes (see, e.g.,
\cite[Section 1.2]{GaudreauLamarreGhosalLiao} and references therein).

Due to their ubiquity in mathematical physics, there is a strong incentive to understand any structure
that appears in the eigenvalues of random Schr\"odinger operators, including number rigidity. Up until
recently, the only random Schr\"odinger operator whose eigenvalue point process was known to be number rigid
was the stochastic Airy operator $A_\be$ in \eqref{Equation: SAO} with $\be=2$ \cite{Bufetov},
thanks to the special algebraic structure present in the eigenvalues of this particular object (i.e., $A_2$'s eigenvalues
generate the determinantal Airy-$2$ point process). In \cite{GaudreauLamarreGhosalLiao}, we use the Feynman-Kac
formula proved in this paper to show that
number rigidity occurs in the spectrum of $\hat H$
under very general assumptions on the domain $I$ on which the operator is defined,
the boundary conditions on that domain, the regularity of the potential $V$,
and the type of noise; thus providing the first method
capable of proving rigidity for general random Schr\"odinger operators.

\section{Main Results}\label{Section: Main Results}

In this section, we provide detailed statements of our main results.
Throughout this paper, we make the following assumption regarding
the interval $I$ on which the operator is defined and its boundary conditions.

\begin{assumption}\label{Assumption: DB}
We consider three different types of domains: The full space $I=\mbb R$ ({\bf Case 1}),
the positive half line $I=(0,\infty)$ ({\bf Case 2}), and the bounded interval $I=(0,b)$ for some $b>0$
({\bf Case 3}).

In Case 2, we consider Dirichlet and Robin boundary conditions
at the origin:
\begin{align}\label{Equation: Case 2 Boundary}
\begin{cases}
f(0)=0&\text{({\bf Case 2-D})}\\
f'(0)+\al f(0)=0&\text{({\bf Case 2-R})}
\end{cases}
\end{align}
where $\al\in\mbb R$ is fixed.

In Case 3, we consider the Dirichlet, Robin, and mixed boundary conditions at
the endpoints $0$ and $b$:
\begin{align}\label{Equation: Case 3 Boundary}
\begin{cases}
f(0)=f(b)=0&\text{({\bf Case 3-D})}\\
f'(0)+\al f(0)=-f'(b)+\be f(b)=0&\text{({\bf Case 3-R})}\\
f'(0)+\al f(0)=f(b)=0&\text{({\bf Case 3-M})}
\end{cases}
\end{align}
where $\al,\be\in\mbb R$ are fixed.
\end{assumption}

\begin{remark}
Case 3-M should technically also include the mixed boundary conditions of the form
$f(0)=-f'(b)+\be f(b)=0$.
However, the latter can easily be
obtained from Case 3-M by considering the transformation $x\mapsto f(b-x)$.
\end{remark}

Throughout the paper,
we make the following assumption on the potential $V$.

\begin{assumption}\label{Assumption: PG}
Suppose that $V:I\mapsto\mbb R$ is nonnegative and locally integrable on $I$'s closure.
If $I$ is unbounded, then we also assume that
\begin{align}\label{Equation: Assumptions}
\liminf_{x\to\pm\infty}\frac{V(x)}{\log|x|}=\infty.
\end{align}
\end{assumption}

\begin{remark}
As is usual in the theory of Schr\"odinger operators and semigroups, the assumption that $V\geq0$
is made for technical ease, and all of our results also apply in the case where $V$ is merely bounded from
below on $I$.
\end{remark}

\subsection{Self-Adjoint Operator}

Our first result concerns the realization of ${\hat H}$ as a self-adjoint operator.
As explained
in the passage following equation \eqref{Equation: Intro Quadratic Form}, this is done through
a sesquilinear form. We begin by introducing the sesquilinear form associated with $H$:

\begin{definition}\label{Definition: Forms}
Let $L^2=L^2(I)$ denote the set of square integrable functions
(equivalence classes up to measure zero) on $I$,
with its usual inner product and norm
\[\langle f,g\rangle:=\int_If(x)g(x)\d x,\qquad\|f\|_2:=\sqrt{\langle f,f\rangle}.\]
Let $\mr{AC}=\mr{AC}(I)$ denote the set of functions that are locally absolutely
continuous on $I$'s closure, and let
\[\mr H^1_V=\mr H^1_V(I):=\big\{f\in\mr{AC}:\|f\|_2,\|f'\|_2,\|V^{1/2}f\|_2<\infty\big\}.\]
We define the following inner product and norm on $\mr H^1_V$:
\[\langle f,g\rangle_*:=\langle f',g'\rangle+\langle fg,V+1\rangle,\qquad
\|f\|_*^2:=\|f'\|_2^2+\|V^{1/2}f\|_2^2+\|f\|_2^2.\]
We define $H$'s sesquilinear form $\mc E$ as well as its
form domain $D(\mc E)\subset\mr H^1_V$ for every case in Assumption \ref{Assumption: DB}
as follows:
\begingroup
\allowdisplaybreaks
\begin{align*}
\textbf{Case 1: }&\quad\begin{cases}
D(\mc E):=\mr H^1_V\\
\mc E(f,g):=\tfrac12\langle f',g'\rangle+\langle fg,V\rangle\\
\end{cases}\\
\textbf{Case 2-D: }&\quad\begin{cases}
D(\mc E):=\big\{f\in \mr H^1_V:f(0)=0\big\}\\
\mc E(f,g):=\tfrac12\langle f',g'\rangle+\langle fg,V\rangle\\
\end{cases}\\
\textbf{Case 2-R: }&\quad\begin{cases}
D(\mc E):=\mr H^1_V\\
\mc E(f,g):=\tfrac12\langle f',g'\rangle-\tfrac\al2 f(0)g(0)+\langle fg,V\rangle\\
\end{cases}\\
\textbf{Case 3-D: }&\quad\begin{cases}
D(\mc E):=\big\{f\in \mr H^1_V:f(0)=f(b)=0\big\}\\
\mc E(f,g):=\tfrac12\langle f',g'\rangle+\langle fg,V\rangle\\
\end{cases}\\
\textbf{Case 3-R: }&\quad\begin{cases}
D(\mc E):=\mr H^1_V\\
\mc E(f,g):=\tfrac12\langle f',g'\rangle-\tfrac\al2 f(0)g(0)-\tfrac\be2 f(b)g(b)+\langle fg,V\rangle\\
\end{cases}\\
\textbf{Case 3-M: }&\quad\begin{cases}
D(\mc E):=\big\{f\in \mr H^1_V:f(b)=0\big\}\\
\mc E(f,g):=\tfrac12\langle f',g'\rangle-\tfrac\al2 f(0)g(0)+\langle fg,V\rangle\\
\end{cases}
\end{align*}
\endgroup
\end{definition}

\begin{remark}
As noted by Bloemendal and Vir\'ag in \cite[Remark 2.5 and
(2.11)]{BloemendalVirag},
the Dirichlet boundary conditions can be specified in the
form domain $D(\mc E)$, but the Robin conditions must be enforced by the form itself,
since the derivative of an absolutely continuous function is only defined almost everywhere.
Taking Case 3-R as an example, by a formal integration by parts we have
\[-\int_0^bf(x)g''(x)\d x=f(b)\big(-g'(b)\big)+f(0)g'(0)+\langle f',g'\rangle.\]
Substituting $g'(0)=-\al g(0)$ and $-g'(b)=-\be g(0)$ then yields $\mc E(f,g)$.
\end{remark}

We now define the form associated with $\hat H$ as a random perturbation of $\mc E$
coming from the noise. We assume that the Gaussian process $\Xi$ driving the noise is as follows:

\begin{assumption}\label{Assumption: FN}
$\Xi:\mbb R\to\mbb R$ is a centered Gaussian process such that:
\begin{enumerate}
\item Almost surely, $\Xi(0)=0$ and $\Xi$ has continuous sample paths.
\item $\Xi$ has stationary increments, that is, for every $x_1,\ldots,x_\ell,y_1,\ldots,y_\ell\in\mbb R$
($\ell\in\mbb N$) such that $x_i\leq y_i$ for all $1\leq i\leq\ell$ and $h\in\mbb R$, the increments
\[\Xi(y_1)-\Xi(x_1),\Xi(y_2)-\Xi(x_2),\ldots,\Xi(y_\ell)-\Xi(x_\ell)\]
have the same joint distribution as the shifted increments
\[\Xi(y_1+h)-\Xi(x_1+h),\Xi(y_2+h)-\Xi(x_2+h),\ldots,\Xi(y_\ell+h)-\Xi(x_\ell+h).\]
\end{enumerate}
\end{assumption}

We may now define $\xi$ as the distributional derivative of $\Xi$:

\begin{definition}
\label{Definition: Operator Noise}
Let $\mr C^\infty_0=\mr C^\infty_0(I)$ denote the set of functions
that are smooth and compactly supported on $I$'s closure.
For every $f\in\mr C^\infty_0$, we define $\xi(f)=\langle f,\Xi'\rangle$ by a
formal integration by parts:
\begin{align}
\label{Equation: Operator Noise}
\xi(f):=\begin{cases}
-\langle f',\Xi\rangle&\text{(Cases 1 and 2)}\\
f(b)\Xi(b)-\langle f',\Xi\rangle&\text{(Case 3)}
\end{cases}
\end{align}
(note that we omit the boundary term $-f(0)\Xi(0)$ in the formal integration by parts for Cases 2 and 3,
since we assume that $\Xi(0)=0$).
\end{definition}

Our first result is that the sesquilinear form $(f,g)\mapsto\xi(fg)$ for $f,g\in\mr C^\infty_0$ can be continuously extended to
$\mr H^1_V$, and thus can be added to the form $\mc E$ as hinted at in \eqref{Equation: Intro Quadratic Form}:

\begin{proposition}\label{Proposition: Noise Continuity}
Suppose that Assumptions \ref{Assumption: DB}, \ref{Assumption: PG}, and
\ref{Assumption: FN} hold. 
There exists a finite random variable $c>0$ such that, almost surely,
\begin{align}
\label{Equation: Noise Continuity}
|\xi(f^2)|\leq c\|f\|_*^2,\qquad f\in\mr C^\infty_0.
\end{align}
Hence $f\mapsto\xi(f^2)$ extends uniquely to a continuous
quadratic form on $\mr H^1_V$ that satisfies \eqref{Equation: Noise Continuity}
for all $f\in\mr H^1_V$,
which we can then also extend to a sesquilinear form by the polarization identity:
\[\xi(fg):=\frac{\xi\big((f+g)^2\big)-\xi\big((f-g)^2\big)}{4},\qquad f,g\in\mr H^1_V.\]
In particular, almost surely, we can
define the sesquilinear form
\begin{align}
\label{Equation: E Hat}
\hat{\mc E}(f,g):=\mc E(f,g)+\xi(fg)
\end{align}
on the same form domain as $\mc E$, that is, for all
for all $f,g\in D(\mc E)$.
\end{proposition}

We may now state our main result regarding our definition of
$\hat H$ as the self-adjoint operator associated with the form
\eqref{Equation: E Hat} on the form domain $D(\mc E)$:

\begin{proposition}\label{Proposition: Operator}
Suppose that Assumptions \ref{Assumption: DB}, \ref{Assumption: PG}, and
\ref{Assumption: FN} hold.
Almost surely, there exists a unique self-adjoint
operator ${\hat H}$ with dense domain $D({\hat H})\subset L^2$ such that
\begin{enumerate}
\item $D({\hat H})\subset D(\mc E)$;
\item For every $f,g\in D({\hat H})$, one has $\langle f,{\hat H}g\rangle=\hat{\mc E}(f,g)$; and
\item ${\hat H}$ has compact resolvent.
\end{enumerate}
\end{proposition}

\begin{remark}
In Case 1, the statement of Proposition \ref{Proposition: Operator} is to the best
of our knowledge completely new. In Case 2, the closest results are \cite[Theorem 2]{Minami},
which assumes that $I=(0,\infty)$ with Dirichlet boundary condition,
that $V$ is continuous,
and that $\Xi$ is a fractional Brownian motion.
In Case 3, the closest result seems to
be \cite[\S2]{FukushimaNakao}, which only considers the case $V=0$ with Dirichlet boundary conditions
and $\Xi$ a Brownian motion. Proposition \ref{Proposition: Noise Continuity} is a generalization
of similar results in \cite[Lemma 2.3]{BloemendalVirag}, \cite[Proposition 1-(i)]{Minami},
and \cite[Proposition 2.4]{RamirezRiderVirag}.
\end{remark}

An immediate corollary of Proposition \ref{Proposition: Operator} is the ability
to study the spectrum of $\hat H$ using the variational characterization coming
from the form $\hat{\mc E}$:

\begin{definition}
Let $ A$ be a semi-bounded self-adjoint operator with discrete spectrum. We use $\la_1( A)\leq\la_2( A)\leq\cdots$
to denote the eigenvalues of $ A$ in increasing order, and we use $\psi_1( A),\psi_2( A),\ldots$ to denote
the associated eigenfunctions.
\end{definition}

\begin{corollary}\label{Corollary: Obvious Form Expansion}
Under the assumptions of Proposition \ref{Proposition: Operator}, almost surely,
\begin{enumerate}
\item $-\infty<\la_1(\hat H)\leq\la_2(\hat H)\leq\cdots\nearrow+\infty$;
\item the $\psi_k(\hat H)$ form an orthonormal basis of $L^2$; and
\item for every $k\in\mbb N$,
\[\la_k(\hat H)=\inf_{\psi\in D(\mc E),~\psi\perp\psi_1(\hat H),\ldots,\psi_{k-1}(\hat H)}
\frac{\hat{\mc E}(\psi,\psi)}{\|\psi\|_2^2},\]
with $\psi_k(\hat H)$ being the minimizer of the above infimum
with unit $L^2$ norm.
\end{enumerate}
\end{corollary}

\subsection{Semigroup}

We now state our main result regarding the Feynman-Kac formula for the semigroup generated by ${\hat H}$.
Thanks to Proposition \ref{Proposition: Operator} and Corollary \ref{Corollary: Obvious Form Expansion},
we know that under Assumptions  \ref{Assumption: DB}, \ref{Assumption: PG},
and \ref{Assumption: FN},
the semigroup of $\hat H$ is the family of
bounded self-adjoint operators with spectral expansions
\begin{align}
\label{Equation: H hat Semigroup Expansion}
\mr e^{-t\hat H}f=\sum_{k=1}^\infty\mr e^{-t\la_k({\hat H})}\langle\psi_k({\hat H}),f\rangle\psi_k({\hat H}),
\qquad t>0,~f\in L^2.
\end{align}
In order to state our Feynman-Kac formula for $\mr e^{-t\hat H}$,
we introduce some notations and further assumptions.

\subsubsection{Preliminary Definitions}

We begin with some preliminary definitions
regarding the covariance of the noise $\xi$
and the stochastic processes required to
define our Feynman-Kac kernels.

\begin{definition}[Covariance]\label{Definition: Covariance}
Let us denote by $\mr{PC}_c=\mr{PC}_c(I)$ the set of functions $f:I\mapsto\mbb R$ that are
c\`adl\`ag and compactly supported on $I$'s closure.
We say that $f\in\mr{PC}_c$ is a step function if it can be written as
\begin{align}
\label{Equation: Step Function}
f=\sum_{i=1}^kc_i\mbf 1_{[x_i,x_{i+1})}\qquad c_i\in\mbb R,~-\infty<x_1< x_2<\cdots<x_{k+1}<\infty.
\end{align}
To simplify forthcoming definitions and statements, we often extend the domain of $f\in\mr{PC}_c$
to $\mbb R$, with the convention that $f(x)=0$ for all $x$ outside of $I$'s closure (noting, however,
that $f$'s extension need not be c\`adl\`ag on all of $\mbb R$).

Let $\ga:\mr{PC}_c\to\mbb R$ be an even almost-everywhere-defined function or Schwartz distribution
(even in the sense that $\langle f,\ga\rangle=\langle \mf r f,\ga\rangle$ for every $f$,
where $\mf rf(x)=f(-x)$ denotes the reflection map),
such that the bilinear map
\begin{align}
\label{Equation: Semi-Inner-Product}
\langle f,g\rangle_\ga:=\int_{\mbb R^2}f(x)\ga(x-y)g(y)\d x\dd y,\qquad f,g\in \mr{PC}_c
\end{align}
is a semi-inner-product.
We denote the seminorm induced by \eqref{Equation: Semi-Inner-Product} as
\[\|f\|_\ga:=\sqrt{\langle f,f\rangle}_\ga,\qquad f\in\mr{PC}_c.\]
\end{definition}

\begin{remark}
If $\ga$ is not an almost-everywhere-defined function, then
the integral over $\ga(x-y)$ in \eqref{Equation: Semi-Inner-Product} may not be well defined.
In such cases, we rigorously interpret \eqref{Equation: Semi-Inner-Product}
as $\langle f*\mf r g,\ga\rangle=\langle \mf r f*g,\ga\rangle$.
\end{remark}

\begin{definition}[Stochastic Processes, etc.]\label{Definition: Stochastic Processes Etc}
We use $B$ to denote a standard Brownian motion on $\mbb R$,
$X$ to denote a reflected standard Brownian motion on $(0,\infty)$,
and $Y$ to denote a reflected standard Brownian motion on $(0,b)$.

Let $Z=B$, $X$, or $Y$.
For every $t>0$ and $x,y\in I$, we
define the conditioned processes
\[Z^x:=\big(Z|Z(0)=x\big)\qquad\text{and}\qquad Z^{x,y}_t:=\big(Z|Z(0)=x\text{ and }Z(t)=y\big),\]
and we use $\mbf E^x$ and $\mbf E^{x,y}_t$ to denote the expected value with respect to the law of
$Z^x$ and $Z^{x,y}_t$, respectively.

We denote the Gaussian kernel by
\[\ms G_t(x):=\frac{\mr e^{-x^2/2t}}{\sqrt{2\pi t}},\qquad t>0,~x\in\mbb R.\]
We denote the transition kernels of $B$, $X$, and $Y$ as $\Pi_B$, $\Pi_X$, and $\Pi_Y$,
respectively. That is, for every $t>0$,
\begin{align*}
\Pi_B(t;x,y)&:=\ms G_t(x-y)&x,y\in\mbb R,\\
\Pi_X(t;x,y)&:=\ms G_t(x-y)+\ms G_t(x+y)&x,y\in(0,\infty),\\
\Pi_Y(t;x,y)&:=\sum_{z\in2b\mbb Z\pm y}\ms G_t(x-z)&x,y\in(0,b).
\end{align*}

Let $Z=B$, $X$, or $Y$.
For any time interval $[u,v]\subset [0,\infty)$, we let $a\mapsto L^a_{[u,v]}(Z)$ ($a\in I$) denote the continuous version
of the local time of $Z$
(or its conditioned versions) on $[u,v]$, that is,
\begin{align}
\label{Equation: Interior Local Time}
\int_u^vf\big(Z(s)\big)\d s=\int_I L^a_{[u,v]}(Z)f(a)\d a=\langle L_{[u,v]}(Z),f\rangle
\end{align}
for any measurable function $f:I\to\mbb R$. In the special case where $u=0$ and $v=t$, we use the shorthand $L_t(Z):=L_{[0,t]}(Z)$.
When there may be ambiguity regarding which conditioning of $Z$ is under consideration, we use
$L_{[u,v]}(Z^x)$ and $L_{[u,v]}(Z^{x,y}_t)$.

As a matter of convention, if $Z=X$ or $Y$, then we distinguish the boundary local time from the above, which we define as
\[\mf L^c_{[u,v]}(Z):=\lim_{\eps\to0}\frac1{2\eps}\int_u^v\mbf 1_{\{c-\eps<Z(s)<c+\eps\}}\d s\]
for $c\in\partial I$ (i.e., $c=0$ if $Z=X$ or $c\in\{0,b\}$ if $Z=Y$),
with $\mf L_t^c(Z):=\mf L_{[0,t]}^c(Z)$.
\end{definition}

\begin{remark}
Since we use the continuous version of local time,
$a\mapsto L^a_{[u,v]}(Z)$ is continuous and compactly
supported on $I$'s closure; in particular, $L_{[u,v]}(Z)\in\mr{PC}_c$.
\end{remark}

\subsubsection{Noise}

We now articulate the assumptions that
the noise $\xi$ must satisfy for our Feynman-Kac formula to hold.
We recall from the introduction that we think of $\xi$ as a centered Gaussian process
with covariance $\mbf E[\xi(x)\xi(y)]=\ga(x-y)$, with $\ga$ as in
Definition \ref{Definition: Covariance}.
Interpreting $\xi(f)``="\int_\mbb R f(x)\xi(x)\d x$ for a function $f$,
this suggests that, as a random Schwartz distribution, $\xi$ is a centered Gaussian process with covariance
$\mbf E\left[\xi(f)\xi(g)\right]=\langle f,g\rangle_\ga.$
In similar fashion to Assumption \ref{Assumption: FN}, we want to interpret
$\xi$ as the distributional derivative of some continuous process $\Xi$,
that is, corresponding to \eqref{Equation: Anti-Derivative Process}.
If $\xi$'s covariance is given by the semi-inner-product
$\langle\cdot,\cdot\rangle_\ga$,
then this suggests that $\Xi$'s covariance is equal to
\begin{align}
\label{Equation: Xi Covariance}
\mbf E[\Xi(x)\Xi(y)]=\begin{cases}
\langle\mbf 1_{[0,x)},\mbf 1_{[0,y)}\rangle_\ga&\text{if }0\leq x,y\\
\langle\mbf 1_{[0,x)},-\mbf 1_{[y,0)}\rangle_\ga&\text{if }y\leq0\leq x\\
\langle-\mbf 1_{[x,0)},\mbf 1_{[0,y)}\rangle_\ga&\text{if }x\leq0\leq y\\
\langle\mbf 1_{[x,0)},\mbf 1_{[y,0)}\rangle_\ga&\text{if }x,y\leq 0.
\end{cases}
\end{align}
This leads us to the following Assumption:

\begin{assumption}
\label{Assumption: SN}
The centered Gaussian process
$\Xi:\mbb R\to\mbb R$ satisfies Assumption \ref{Assumption: FN}.
Moreover, there exists a $\ga:\mr{PC}_c\to\mbb R$ as in
Definition \ref{Definition: Covariance} that
satisfies the following conditions.
\begin{enumerate}
\item $\Xi$'s covariance is given by \eqref{Equation: Xi Covariance}.
\item There exists $c_\ga>0$ and $1\leq q_1,\ldots,q_\ell\leq2$ (for some $\ell\in\mbb N$) such that
\begin{align}
\label{Equation: gamma to Lp Bounds}
\|f\|_{\ga}^2\leq c_\ga\big(\|f\|_{q_1}^2+\cdots+\|f\|_{q_\ell}^2\big),\qquad f\in\mr{PC}_c,
\end{align}
where $\|f\|_q:=\left(\int_\mbb R|f(x)|^q\d x\right)^{1/q}$ denotes the usual $L^q$ norm.
\end{enumerate}
Then, for every $f\in\mr{PC}_c$, we define
\begin{align}
\label{Equation: Stochastic Integral}
\xi(f):=\int_\mbb R f(x)\d\Xi(x),
\end{align}
where $\dd\Xi$ denotes stochastic integration with respect to $\Xi$ interpreted in the
pathwise sense of Karandikar \cite{Karandikar} (see Section \ref{Section: Stochastic Integral}
for the details of this construction).
\end{assumption}

\begin{remark}
Though this is not immediately obvious from the above definition,
the pathwise stochastic integral \eqref{Equation: Stochastic Integral}
actually coincides with \eqref{Equation: Operator Noise}
for every $f\in\mr C^\infty_0$. We note,
however, that the extension of $\xi$ to $\mr{PC}_c$
need not be linear on all of $\mr{PC}_c$, and thus may not
be a Schwartz distribution in the proper sense on that larger domain.
Our interest in defining the stochastic integral in
a pathwise sense is that it allows to construct $\xi$
as a random map from $\mr{PC}_c$ to $\mbb R$
that satisfies the following properties.
\begin{enumerate}
\item We can consider the conditional distribution of $\xi\big(L_t(Z)\big)$
given a fixed realization of $\Xi$,
assuming independence between $Z$ and $\Xi$.
\item $f\mapsto\xi(f)$ is a centered Gaussian process
on $\mr{PC}_c$ with covariance $\langle \cdot,\cdot\rangle_\ga$.
\end{enumerate}
In fact, any other pathwise stochastic integral that is an extension of
\eqref{Equation: Operator Noise} and satisfies these two properties
leads to the same statement in Theorem \ref{Theorem: Semigroup} below.
We point to Section \ref{Section: Stochastic Integral}
and Appendix \ref{Appendix: Measurable} for the details of the proof that
$\xi$ has these two properties, and to Section \ref{Section: Smooth Approximations}
for an explanation of why any stochastic integral having these two properties gives
rise to our main result.
\end{remark}

\begin{remark}
\label{Remark: Redundancy}
The requirement that $\Xi$ be a continuous process
with stationary increments in Assumption \ref{Assumption: SN}
is redundant: Firstly, the covariance \eqref{Equation: Xi Covariance}
implies that $\Xi(x)-\Xi(y)$ corresponds to $\xi(\mbf 1_{[x,y)})$,
which is stationary since the semi-inner-product $\langle\cdot,\cdot\rangle_\ga$
is translation invariant. Secondly, if we construct $\Xi$ using
abstract existence theorems for Gaussian processes
(which is possible since $\langle\cdot,\cdot\rangle_\ga$ is a semi-inner-product),
then the assumption
\eqref{Equation: gamma to Lp Bounds} implies that $\Xi$ has
a continuous version by Kolmogorov's theorem for path continuity
(see Section \ref{Section: Examples Outline} for details). We nevertheless state
these properties as assumptions for clarity.
\end{remark}

\subsubsection{Feynman-Kac Kernels}

We now introduce the Feynman-Kac kernels that describe $\hat H$'s semigroup.

\begin{definition}\label{Definition: Random Semigroup}
In Cases 2 and 3, let us define the quantities
\[\bar\al:=\begin{cases}
-\infty&\text{(Case 2-D)}\\
\al&\text{(Case 2-R)}
\end{cases}
\qquad
(\bar\al,\bar\be):=\begin{cases}
(-\infty,-\infty)&\text{(Case 3-D)}\\
(\al,\be)&\text{(Case 3-R)}\\
(\al,-\infty)&\text{(Case 3-M)}
\end{cases}\]
where $\al,\be\in\mbb R$ are as in \eqref{Equation: Case 2 Boundary} and \eqref{Equation: Case 3 Boundary}.
For every $t>0$, we define the (random) kernel ${\hat K}(t):I^2\to\mbb R$ as
\begin{align}\label{Equation: Random Semigroup}
{\hat K}(t;x,y):=\begin{cases}
\Pi_B(t;x,y)\,\mbf E^{x,y}_t\Big[\mr e^{-\langle L_t(B),V\rangle-\xi(L_t(B))}\Big]&\text{(Case 1)}
\vspace{5pt}
\\
\Pi_X(t;x,y)\,\mbf E^{x,y}_t\Big[\mr e^{-\langle L_t(X),V\rangle-\xi(L_t(X))+\bar\al \mf L^0_t(X)}\Big]&\text{(Case 2)}
\vspace{5pt}
\\
\Pi_Y(t;x,y)\,\mbf E^{x,y}_t\Big[\mr e^{-\langle L_t(Y),V\rangle-\xi(L_t(Y))+\bar\al \mf L^0_t(Y)+\bar\be \mf L^b_t(Y)}\Big]&\text{(Case 3)}
\end{cases}
\end{align}
where we assume that $\Xi$ is independent of $B$, $X$, or $Y$,
and $\mbf E^{x,y}_t$ denotes the expected value conditional on $\Xi$.
\end{definition}

\begin{remark}\label{Remark: Infinity Times Zero}
Let $Z=X$ or $Y$.
In the above definition, we use the convention
\[-\infty\cdot \mf L^c_t(Z)=\begin{cases}
0&\text{if }\mf L^c_t(Z)=0\\
-\infty&\text{if }\mf L^c_t(Z)>0
\end{cases}\]
for any $c\in\partial I$ as well as $\mr e^{-\infty}=0$.
Thus, if we let $\tau_c(Z):=\inf\{t\geq0:Z(t)=c\}$ denote the first hitting time of $c$, then we can interpret
$\mr e^{-\infty\cdot \mf L^c_t(Z)}=\mbf 1_{\{\tau_c(Z)>t\}}$.
In particular, if we remove the term $\xi(L_t(Z))$ from the kernel
\eqref{Equation: Random Semigroup}, then we
recover the classical Feynman-Kac formula
for the semigroup of $H$. See Section
\ref{Section: First Semigroup Technical} for more details.
\end{remark}

\begin{notation}
Given a Kernel $J:I^2\to\mbb R$ (such as ${\hat K}(t)$),
we also use $J$ to denote the integral operator induced by the kernel, that is,
\[Jf(x):=\int_I  J(x,y)f(y)\d y.\]
We say that $J$ is Hilbert-Schmidt if $\|J\|_2<\infty$,
and trace class if $\mr{Tr}[|J|]<\infty$.
\end{notation}

\subsubsection{Main Result}

Our main result is as follows.

\begin{theorem}[Feynman-Kac Formula]
\label{Theorem: Semigroup}
Suppose that Assumptions \ref{Assumption: DB}, \ref{Assumption: PG}, and \ref{Assumption: SN}
hold. Almost surely, $\mr e^{-t\hat H}$ is a Hilbert-Schmidt/trace class integral operator
for every $t>0$. Moreover, for every $t>0$, the following holds with probability one.
\begin{enumerate}
\item $\mr e^{-t\hat H}=\hat K(t)$.
\item $\displaystyle\mr{Tr}\big[\mr e^{-t\hat H}\big]=\int_I\hat K(t;x,x)\d x<\infty.$
\end{enumerate}
\end{theorem}

\begin{remark}
\label{Remark: Measurability of K}
We point to Section \ref{Section: Stochastic Integral} and Appendix \ref{Appendix: Measurable}
for a justification of the well-posedness of the conditional
expectation in \eqref{Equation: Random Semigroup}
and that the kernel $\hat K(t)$
is Borel measurable,
thus making quantities such as
\[\int_I \hat K(t;x,y)f(y)~\dd y,
\quad
\int_{I^2}\hat K(t;x,y)^2~\dd x\dd y,
\quad\text{and}\quad
\int_I\hat K(t;x,x)~\dd x\]
(where $f\in L^2$)
well defined.
\end{remark}

\begin{remark}
The closest analogs of Theorem \ref{Theorem: Semigroup} in the
literature are \cite[Proposition 1.8 (a)]{GaudreauLamarreShkolnikov}
and \cite[Corollary 2.2]{GorinShkolnikov}, which concern Case 2
in the special case where $V(x)=x$ and $\Xi$ is a Brownian motion.
All other cases are new.
\end{remark}

\begin{remark}
Though this direction is not explored in this paper,
we expect that one could prove (in similar fashion to,
e.g., \cite[Theorem 4.12]{HuEtAl}) that the kernels
$\hat K(t;x,y)$ admit continuous modifications in $t$, $x$, and $y$.
\end{remark}

\subsection{Optimality and Examples}
\label{Section: Optimality}

We finish Section \ref{Section: Main Results} by discussing the optimality of
the growth condition \eqref{Equation: Assumptions}
in our results and by providing examples of covariance
functions/distributions $\ga$ that satisfy Assumption
\ref{Assumption: SN}.

\subsubsection{Optimality of Potential Growth}

%\\

On the one hand,
one of the key aspects of our proof of Proposition \ref{Proposition: Operator}
for unbounded domains $I$
is to show that the growth rate of the squared increment process
\[x\mapsto\big(\Xi(x+1)-\Xi(x)\big)^2\] is dominated by $V$ as $|x|\to\infty$
(see \eqref{Equation: Integrated Noise Bounds}, \eqref{Equation: Integrated Noise Bounds 2}, and the passage that follows).
Given that the growth rate of stationary Gaussian processes (such as $\Xi(x+1)-\Xi(x)$)
is at most of order $\sqrt{\log|x|}$ (e.g., Corollary \ref{Corollary: Stationary Noise Suprema}),
and that in many cases there is also a matching lower bound
(e.g., Remark \ref{Remark: Optimality of Stationary Noise Suprema}),
the growth condition \eqref{Equation: Assumptions}
appears to be the best one can hope for with the method we use to prove Proposition
\ref{Proposition: Operator}.
It would be interesting to see if this condition is necessary for $\hat H$ to have
compact resolvent (perhaps by using the Sturm-Liouville interpretation \eqref{Equation: Bloemendal}).
That being said, for the deterministic operator $H=-\tfrac12\De+V$ on $I=(0,\infty)$,
it is well known that having a spectrum of discrete eigenvalues that are bounded below is equivalent to
$\int_{x}^{x+\de}V(y)\d y\to\infty$ as $x\to\infty$ for all $\de>0$; hence it is natural to expect that $V$
must have some kind logarithmic growth to balance the Gaussian potential.

On the other hand,
condition \eqref{Equation: Assumptions} is necessary to have that
that $\mbf E\big[\|\hat K(t)\|_2^2\big]<\infty$
for $t>0$ close to zero, which is crucial in our proof
of Theorem \ref{Theorem: Semigroup}.
Given that the deterministic semigroup $\mr e^{-t H}$
is not trace class for small $t>0$ when \eqref{Equation: Assumptions}
does not hold, we do not expect it is possible to improve Theorem
\ref{Theorem: Semigroup} in that regard.
We refer to
Remark \ref{Remark: Optimality of Log Follow-Up}
for more details.

\subsubsection{Examples}

Given the simplicity of  Assumption \ref{Assumption: FN},
it is straightforward to come up with examples of Gaussian noises to which
Proposition \ref{Proposition: Operator} can be applied. In contrast,
Assumption \ref{Assumption: SN} is a bit more involved.
In what follows, we provide examples of covariance functions/distributions $\ga$
that satisfy Assumption \ref{Assumption: SN}.

\begin{example}\label{Example: Four Noises}
Let $\ga:\mr{PC}_c\to\mbb R$ be an even almost-everywhere-defined
function or Schwartz distribution.
\begin{enumerate}
\item ({\bf Bounded}) If $\ga\in L^\infty(\mbb R)$,
then we call $\xi$ a bounded noise.
Depending on the regularity of $\ga$,
in many such cases $\xi$ can actually be realized as a continuous
Gaussian process on $\mbb R$ with covariance $\mbf E[\xi(x)\xi(y)]=\ga(x-y)$.

\item ({\bf White}) If $\ga=\si^2\de_0$ for some $\si>0$, where $\de_0$
denotes the delta Dirac distribution, then $\xi$ is a
Gaussian white noise with variance $\si^2$.
This corresponds to stochastic integration with respect
to a two-sided Brownian motion $W$ with variance $\si^2$:
\[\xi(f)=\int_{\mbb R}f(x)\d W(x).\]

\item ({\bf Fractional}) If $\ga(x):=\si^2\mf H(2\mf H-1)|x|^{2\mf H-2}$ for $\si>0$ and $\mf H\in(1/2,1)$,
then $\xi$ is a fractional noise with variance $\si^2$ and Hurst parameter $\mf H$.
This noise corresponds to stochastic integration with respect
to a two-sided fractional Brownian motion $W^\mf H$ with variance $\si^2$
and Hurst parameter $\mf H$:
\[\xi(f)=\int_{\mbb R}f(x)\d W^\mf H(x).\]

\item ({\bf $\bs{L^p}$-Singular}) Let $\ell\in\mbb N$ and $1\leq p_1,\ldots,p_\ell<\infty$.
As a generalization of bounded and fractional noise,
we say that $\xi$ is an $L^p$-singular noise if
\[\ga=\ga_1+\cdots+\ga_\ell+\ga_\infty,\]
where $\ga_i\in L^{p_i}(\mbb R)$ for $1\leq i\leq \ell$
and $\ga_\infty\in L^\infty(\mbb R)$. Indeed, the $\ga_i$
may have one or several $p_i$-integrable point singularities,
such as $\ga_i(x)\sim |x|^{-\mf e}$ as $x\to0$ for some $\mf e\in(0,1/p_i)$,
or $\ga_i(x)\sim (-\log |x|)^{\mf e}$ as $x\to0$ for $\mf e>0$.
\end{enumerate}
\end{example}

Our last result in this Section is the following.

\begin{proposition}\label{Proposition: Examples}
For every covariance $\ga$ in Example \ref{Example: Four Noises},
there exists a centered Gaussian process $\Xi$ that satisfies Assumption \ref{Assumption: SN}.
\end{proposition}

\section{Proof Outline}\label{Section: Outline}

In this section, we provide an outline of the proofs of our main results.
Most of the more technical results, which we state here as a string of propositions,
are accounted for in Sections \ref{Section: Operator} and \ref{Section: Semigroup}.
Throughout Section \ref{Section: Outline}, we assume that Assumptions
\ref{Assumption: DB} and \ref{Assumption: PG} are met.

\subsection{Outline for Propositions \ref{Proposition: Noise Continuity} and \ref{Proposition: Operator}}\label{Section: Operator Outline}

In this outline, we assume that Assumption \ref{Assumption: FN} holds.
Let $\mr{FC}\subset\mr C^\infty_0$ be the set of real-valued smooth functions $\phi:I\to\mbb R$
such that
\begin{enumerate}
\item $\mr{supp}(\phi)$ is a compact subset of $I$ in Cases 1, 2-D, and 3-D;
\item $\mr{supp}(\phi)$ is a compact subset of $I$'s closure in Cases 2-R and 3-R; and
\item $\mr{supp}(\phi)$ is a compact subset of $[0,b)$ in Case 3-M.
\end{enumerate}
We begin with two classical results in the theory of Schr\"odinger operators.
(For definitions of the functional analysis terminology used in this section, we refer to
\cite[Section VIII.6]{ReedSimonI}, \cite[Section 7.5]{SimonBook}, or \cite[Section 2.3]{Teschl09}.)

\begin{lemma}\label{Lemma: Pointwise Form Bound}
For every $\ka>0$, there exists $c=c(\ka)>0$
such that for every $f\in\mr{AC}\cap L^2$,
one has $f(x)^2\leq\ka\|f'\|_2^2+c\|f\|_2^2$ for all $x\in I$.
\end{lemma}

\begin{proposition}\label{Proposition: Classical Schrodinger Operator}
$\mc E$ is closed and semibounded on $D(\mc E)$, and
$\mr{FC}$ is a form core for $\mc E$. $H$ is the unique self-adjoint
operator on $L^2$ whose sesquilinear form is $\mc E$, and $H$ has compact resolvent.
Lastly, $\|\cdot\|_*$ is equivalent to the $``$+1 norm$"$ induced by the form $\mc E$,
where we recall that the latter is defined as
\[\|f\|_{+1}^2:=\mc E(f,f)+(c+1)\|f\|_2^2,\qquad f\in D(\mc E),\]
with $c>0$ being a constant large enough so that $\mc E(f,f)+c\|f\|^2\geq0$
for every $f\in D(\mc E)$.
\end{proposition}

Although Lemma \ref{Lemma: Pointwise Form Bound} and Proposition \ref{Proposition: Classical Schrodinger Operator}
can be proved using standard functional-analytic arguments,
we were not able to locate an exact statement in the literature that covers every
case considered in this paper. For the sake of completeness, we provide a proof and references
in Appendix \ref{Appendix: Operator}.

\begin{remark}
\label{Remark: Equivalence of Norms Density}
Since $\|\cdot\|_*$ and $\|\cdot\|_{+1}$ are equivalent, the claim that $\mc E$
is closed on $D(\mc E)$ and that $\mr{FC}$ is a form core is equivalent to the
claim that $\big(D(\mc E),\langle\cdot,\cdot\rangle_*\big)$ is a Hilbert space in
which $\mr{FC}$ is dense.
\end{remark}

The following proposition, which we prove in Section \ref{Section: Operator},
is a generalization of a result that first appeared in \cite{RamirezRiderVirag},
and also uses Lemma \ref{Lemma: Pointwise Form Bound} as a crucial input:

\begin{proposition}\label{Proposition: Form Bound}
The inequality \eqref{Equation: Noise Continuity} holds almost surely, and thus $f\mapsto\xi(f^2)$ extends uniquely to a continuous
quadratic form on $\mr H^1_V$ that satisfies \eqref{Equation: Noise Continuity} for all $f\in\mr H^1_V$.
Moreover, almost surely,
for every $\theta>0$, there exists $c=c(\theta)>0$ such that
\begin{align}
\label{Equation: Infinitesimal Form Bound}
|\xi(f^2)|\leq\theta\mc E(f,f)+c\|f\|_2^2,\qquad f\in D(\mc E).
\end{align}
\end{proposition}

Thanks to \eqref{Equation: Infinitesimal Form Bound}, almost surely, $\xi$
is an infinitesimally form-bounded perturbation of $\mc E$. Therefore, according to the KLMN
theorem (e.g., \cite[Theorem X.17]{ReedSimonII} or \cite[Theorem 7.5.7]{SimonBook}),
$\hat{\mc E}=\mc E+\xi$ is closed and semibounded on $D(\mc E)$, and $\mr{FC}$ is a form
core for $\hat{\mc E}$. Thus, by \cite[Theorem VIII.15]{ReedSimonI}, there exists a unique
self-adjoint operator ${\hat H}$ satisfying conditions (1) and (2) in the statement of
Proposition \ref{Proposition: Operator}. Since $ H$ has compact resolvent and ${\hat H}$
is infinitesimally form-bounded by $ H$, the fact that ${\hat H}$ has compact resolvent
follows from standard variational estimates (e.g., \cite[Theorem XIII.68]{ReedSimonIV}).

\subsection{Outline for Theorem \ref{Theorem: Semigroup}}\label{Section: Semigroup Outline}

We now go over the outline of the proof of our main result.
Throughout, we assume that Assumption \ref{Assumption: SN} holds.
The outline presented here is separated into five steps.
In the first step we provide details on the construction of the pathwise stochastic
integral \eqref{Equation: Stochastic Integral}. In the second step, we introduce
smooth-noise approximations of $\hat H$ and $\hat K(t)$ that serve as the basis
of our proof of Theorem \ref{Theorem: Semigroup}. Then, in the last three steps
we prove Theorem \ref{Theorem: Semigroup} using these smooth approximations.

\subsubsection{Step 1. Stochastic Integral}
\label{Section: Stochastic Integral}

If $f\in\mr{PC}_c$ is a step function of the form \eqref{Equation: Step Function},
then we can define a pathwise stochastic integral in the usual way:
\[\xi(f)=\int_\mbb R f(x)\d\Xi(x):=\sum_{i=1}^kc_i\big(\Xi(x_{i+1})-\Xi(x_i)\big).\]
Thanks to \eqref{Equation: Xi Covariance}, straightforward computations
reveal that for such $f$ we have the isometry $\mbf E\big[\xi(f)^2\big]=\|f\|_\ga^2$.
According to \eqref{Equation: gamma to Lp Bounds},
step functions are dense in $\mr{PC}_c$ with respect to $\|f\|_\ga^2$,
and thus we may then uniquely define a stochastic integral $\xi^*(f)$ for arbitrary $f\in\mr{PC}_c$
as the $L^2(\Omega)$ limit of $\xi(f_n)$, where $f_n$ is a sequence of step functions
that converges to $f$ in $\|\cdot\|_\ga$ and $L^2(\Om)$ denotes the space of square-integrable
random variables on the same probability space
on which $\Xi$ is defined.

We now discuss how $\xi(f)$ for general $f\in\mr{PC}_c$ can be defined in a pathwise sense
as per Karandikar \cite{Karandikar}. Given $f\in\mr{PC}_c$, for every $n\in\mbb N$,
define $k(n)$ and $-\infty<\tau^{(n)}_1\leq\tau^{(n)}_2\leq\cdots\leq\tau^{(n)}_{k(n)+1}<\infty$
as the quantities
\[\tau^{(n)}_1:=\inf\big\{x\in\mbb R:f(x)\neq0\big\},
\qquad
\tau^{(n)}_{k(n)+1}:=\sup\big\{x\in\mbb R:f(x)\neq0\big\}\]
and
\[\tau^{(n)}_k:=\inf\big\{x\geq\tau^{(n)}_{k-1}:\big|f\big(x\big)-f\big(\tau^{(n)}_{k-1}\big)\big|\geq2^{-n}\big\},
\qquad 1<k\leq k(n).\]
Then, we define the approximate step function
\[f^{(n)}:=\sum_{k=1}^{k(n)}f\big(\tau^{(n)}_k\big)\mbf 1_{[\tau^{(n)}_k,\tau^{(n)}_{k+1})}\]
as well as the pathwise stochastic integral
\begin{align}
\label{Equation: Karandikar Definition}
\xi(f)=\int_\mbb Rf(x)\d\Xi(x):=\begin{cases}
\displaystyle\lim_{n\to\infty}\xi(f^{(n)})&\text{if the limit exists}\\
0&\text{otherwise}.
\end{cases}
\end{align}

On the one hand, as argued in Appendix \ref{Appendix: Measurable}
(see also \cite[Section 1]{Karandikar}),
the pathwise definition of $f\mapsto\xi(f)$ in \eqref{Equation: Karandikar Definition}
enables $\hat K(t)$'s definition as a conditional expectation of $\xi\big(L_t(Z)\big)$ given $\Xi$.
On the other hand, $\xi(f)$ retains its meaning as a stochastic integral,
since for every $f\in\mr{PC}_c$, it holds that $\xi(f)=\xi^*(f)$ almost surely. Indeed, by combining
the $L^2(\Omega)$-$\|\cdot\|_\ga$ isometry of $\xi^*$, the definition of $\tau^{(n)}_k$,
and \eqref{Equation: gamma to Lp Bounds},
we get that
\begin{multline*}
\mbf E\big[\big(\xi(f^{(n)})-\xi^*(f)\big)^2\big]=\|f^{(n)}-f\|_\ga^2\\
\leq c_\ga\left(\sum_{i=1}^\ell \|f^{(n)}-f\|^2_{q_i}\right)
\leq c_\ga2^{-2n}\left(\sum_{i=1}^\ell|\mr{supp}(f)|^{2/q_i}\right);
\end{multline*}
since this is summable in $n$ we conclude that $\xi(f^{(n)})\to\xi^*(f)$ almost surely,
as desired.

\begin{remark}
\label{Remark: By Parts}
Let $f\in\mr C^\infty_0(I)$,
and suppose that we restrict our attention to the almost-sure event on which $\Xi$ is continuous.
By a summation by parts, we note that
\begin{multline*}
\xi(f^{(n)})
=\sum_{k=1}^{k(n)}f\big(\tau^{(n)}_k\big)\Big(\Xi\big(\tau^{(n)}_{k+1}\big)-\Xi\big(\tau^{(n)}_k\big)\Big)\\
=f\big(\tau^{(n)}_{k(n)}\big)\Xi\big(\tau^{(n)}_{k(n)+1}\big)-f\big(\tau^{(n)}_{1}\big)\Xi\big(\tau^{(n)}_{1}\big)-\sum_{k=2}^{k(n)}\Xi\big(\tau^{(n)}_k\big)\Big(f\big(\tau^{(n)}_k\big)-f\big(\tau^{(n)}_{k-1}\big)\Big)
\end{multline*}
for all $n\in\mbb N$.
On the one hand,
in Cases 1 and 2, we invariably have that
\[f\big(\tau^{(n)}_{k(n)}\big)\Xi\big(\tau^{(n)}_{k(n)+1}\big)-f\big(\tau^{(n)}_{1}\big)\Xi\big(\tau^{(n)}_{1}\big)=0\]
since $\Xi(0)=0$ and $f$ is compactly supported on $I$'s closure; similarly, in Case 3,
\[f\big(\tau^{(n)}_{k(n)}\big)\Xi\big(\tau^{(n)}_{k(n)+1}\big)-f\big(\tau^{(n)}_{1}\big)\Xi\big(\tau^{(n)}_{1}\big)=f(b)\Xi(b).\]
On the other hand, since $f$ is of bounded variation, we have convergence to the usual Riemann-Stieltjes integral:
\[\lim_{n\to\infty}\sum_{k=2}^{k(n)}\Xi\big(\tau^{(n)}_k\big)\Big(f\big(\tau^{(n)}_k\big)-f\big(\tau^{(n)}_{k-1}\big)\Big)=-\int_I\Xi(x)\d f(x)=-\langle f',\Xi\rangle.\]
In particular, the pathwise stochastic integral defined in \eqref{Equation: Karandikar Definition}
can be seen as an extension of the
Schwartz distribution $\Xi'$ as defined in Definition \ref{Definition: Operator Noise} to all of $\mr{PC}_c$.
However, as noted in an earlier remark,
$\xi$ need not preserve its linearity on all of $\mr{PC}_c$.
\end{remark}

\subsubsection{Step 2. Smooth Approximations}
\label{Section: Smooth Approximations}

A key ingredient in the proof of Theorem \ref{Theorem: Semigroup}
consists of using smooth approximations of $\Xi'$ for which the 
classical Feynman-Kac formula can be applied, thus creating a connection
between $\hat H$ as defined via a quadratic form and the kernels $\hat K(t)$.

\begin{definition}
Let $\rho:\mbb R\to\mbb R$ be a mollifier, that is,
\begin{enumerate}
\item $\rho$ is smooth, compactly supported, nonnegative, even (i.e., $\rho(x)=\rho(-x)$),
and such that $\int\rho(x)\d x=1$; and
\item if we define $\rho_\eps(x):=\eps^{-1}\rho(x/\eps)$ for every $\eps>0$,
then $\rho_\eps\to\de_0$ as $\eps\to0$ in the space of Schwartz distributions,
where $\de_0$ denotes the delta Dirac distribution.
\end{enumerate}
For every $\eps>0$, we define
the stochastic process $\Xi_\eps:=\Xi*\rho_\eps(x)$,
where $*$ denotes the convolution.
\end{definition}

\begin{remark}
\label{Remark: Covariance of Xi Epsilon}
Since $\rho_\eps$ is smooth, the process
$\Xi_\eps'=(\Xi*\rho_\eps)'=\Xi*\rho_\eps'$ has continuous sample paths.
Thanks to \eqref{Equation: Xi Covariance},
straightforward computations reveal that
$\Xi_\eps'$ is a stationary Gaussian process with mean zero and covariance
\begin{multline}\label{Equation: Eps Covariance Function}
\mbf E[\Xi_\eps'(x)\Xi_\eps'(y)]
=\mbf E[(\Xi*\rho_\eps')(x)(\Xi*\rho_\eps')(y)]\\
=\int_{\mbb R^2}\mbf E\big[\Xi(a)\Xi(b)\big]\rho_\eps'(a-x)\rho_\eps'(b-y)\d a\dd b
=\big(\ga*\rho^{*2}_\eps\big)(x-y)
\end{multline}
for every $x,y\in\mbb R$,
where the last equality follows from integration by parts.

Moreover, following-up on Remark
\ref{Remark: By Parts}, we note that
the pathwise stochastic integral $\xi$
is coupled to the random Schwartz distribution
\[f\mapsto\int_\mbb Rf(x)\Xi_\eps'(x)\d x,\qquad f\in \mr{PC}_c\]
in the following way: For every $f\in\mr{PC}_c$,
the function $f*\rho_\eps$ is smooth and compactly
supported on $I+\mr{supp}(\rho_\eps)\subset\mbb R$, and thus by Remark \ref{Remark: By Parts}
we have that
\begin{multline}
\label{Equation: Crucial Coupling}
\int_{I} f(x)\Xi_\eps'(x)\d x
=\int_{I} f(x)(\Xi*\rho_\eps)'(x)\d x
=\int_{I} f(x)(\Xi*\rho_\eps')(x)\d x\\
=-\int_{\mbb R} (f*\rho_\eps')(x)\Xi(x)\d x
=-\int_{\mbb R} (f*\rho_\eps)'(x)\Xi(x)\d x
=\xi(f*\rho_\eps).
\end{multline}
\end{remark}

\begin{definition}\label{Definition: Approximate}
For every $\eps>0$, let us define the sesquilinear form
\[\hat{\mc E}_\eps(f,g):=\mc E(f,g)+\langle fg,\Xi_\eps'\rangle\]
on the form domain $D(\mc E)$,
and the random kernel
\[\hat K_\eps(t;x,y):=\begin{cases}
\Pi_B(t;x,y)\,\mbf E^{x,y}_t\Big[\mr e^{-\langle L_t(B),V+\Xi_\eps'\rangle}\Big]&\text{(Case 1)}
\vspace{5pt}
\\
\Pi_X(t;x,y)\,\mbf E^{x,y}_t\Big[\mr e^{-\langle L_t(X),V+\Xi_\eps'\rangle+\bar\al \mf L^0_t(X)}\Big]&\text{(Case 2)}
\vspace{5pt}
\\
\Pi_Y(t;x,y)\,\mbf E^{x,y}_t\Big[\mr e^{-\langle L_t(Y),V+\Xi_\eps'\rangle+\bar\al \mf L^0_t(Y)+\bar\be \mf L^b_t(Y)}\Big]&\text{(Case 3)}
\end{cases}\]
\end{definition}

Since $\Xi_\eps'$ has regular sample paths, applying classical
operator theory to $\hat H_\eps$ yields the following result:

\begin{proposition}\label{Proposition: Approximate Operator Properties}
For every $\eps>0$, the following holds almost surely:
There exists a unique self-adjoint
operator ${\hat H}_\eps$ with dense domain $D({\hat H}_\eps)\subset L^2$ such that
\begin{enumerate}
\item $D({\hat H}_\eps)\subset D(\mc E)$;
\item For every $f,g\in D({\hat H}_\eps)$, one has $\langle f,{\hat H}_\eps g\rangle=\hat{\mc E}_\eps(f,g)$; and
\item ${\hat H}_\eps$ has compact resolvent.
\end{enumerate}
For every $t>0$, $\mr e^{-t\hat H_\eps}$
is a self-adjoint Hilbert-Schmidt/trace class operator,
and we have the Feynman-Kac formula $\mr e^{-t\hat H_\eps}=\hat K_\eps(t)$.
In particular,
\begin{align}
\label{Equation: Approximate Symmetry Property}
&\hat K_\eps(t;x,y)=\hat K_\eps(t;y,x),&& t>0,~x,y\in I;\\
\label{Equation: Approximate Semigroup Property}
&\int_I\hat K_\eps(t;x,z)\hat K_{\eps}(\bar t;z,y)\d z=\hat K_{\eps}(t+\bar t;x,y),&& t,\bar t>0,~x,y\in I;\\
\label{Equation: Approximate Spectral Expansion}
&\hat K_\eps(t)f=\sum_{i=1}^k\mr e^{-t\la_k(\hat H_\eps)}\langle\psi_k(\hat H_\eps),f\rangle\psi_k(\hat H_\eps),&& f\in L^2.
\end{align}
\end{proposition}

Moreover, as a direct consequence of the coupling \eqref{Equation: Crucial Coupling}
and the fact that $\xi$ is a Gaussian process with covariance $\langle\cdot,\cdot\rangle_\ga$,
we can show that the objects introduced in Definition
\ref{Definition: Approximate} serve as good approximations of ${\hat H}$ and ${\hat K}(t)$ in the following sense:

\begin{proposition}\label{Proposition: Operator Convergence}
Almost surely,
every vanishing sequence in $(0,1]$ has
a further subsequence $(\eps_n)_{n\in\mbb N}$ along which
\begin{align}\label{Equation: Operator Convergence}
\lim_{n\to\infty}\la_k({\hat H}_{\eps_n})=\la_k({\hat H})
\qquad\text{and}\qquad
\lim_{n\to\infty}\|\psi_k({\hat H}_{\eps_n})-\psi_k({\hat H})\|_2=0
\end{align}
for all $k\in\mbb N$, up to possibly relabeling the eigenfunctions of
${\hat H}$ if it has repeated eigenvalues.
\end{proposition}

\begin{proposition}\label{Proposition: Semigroup Convergence}
For every $t>0$, it holds that
\begin{align}
\label{Equation: Semigroup L2 in Expectation}
\lim_{\eps\to0}\mbf E\big[\|\hat K_\eps(t)-\hat K(t)\|_2^2\big]=0
\end{align}
and
\begin{align}
\label{Equation: Semigroup Diagonal in Expectation}
\lim_{\eps\to0}\mbf E\bigg[\bigg(\int_I \hat K_\eps(t;x,x)-\hat K(t;x,x)\d x\bigg)^2\bigg]=0.
\end{align}
\end{proposition}

\subsubsection{Step 3. Feynman-Kac Formula}
\label{Section: Main Outline Part 3}

We are now in a position to prove Theorem \ref{Theorem: Semigroup}.
We begin by proving that for every $t>0$, $\mr e^{-t\hat H}=\hat K(t)$
almost surely.
Let us fix some $t>0$.
By Propositions \ref{Proposition: Approximate Operator Properties}--\ref{Proposition: Semigroup Convergence},
almost surely, there exists a vanishing sequence $(\eps_n)_{n\in\mbb N}$
such that \eqref{Equation: Approximate Symmetry Property}--\eqref{Equation: Approximate Spectral Expansion}
holds for every $\eps_n$, and along which the limits
\eqref{Equation: Operator Convergence} and
\begin{align}
\label{Equation: Semigroup Convergence Subsequence}
\lim_{n\to\infty}\|\hat K_{\eps_n}(t)-\hat K(t)\|_2=0
\end{align}
hold.
For the remainder of this step, we assume that we are working with
an outcome in this probability-one event.

Since the space $L^2(I\times I)$ of Hilbert-Schmidt integral
operators on $L^2$ is complete, \eqref{Equation: Semigroup Convergence Subsequence}
means that $\|\hat K(t)\|_2<\infty$. In particular, $\hat K(t)$ is compact.
Furthermore, given that convergence in Hilbert-Schmidt norm implies
weak operator convergence and every $\hat K_{\eps_n}(t)=\mr e^{-t\hat H_{\eps_n}}$ is
nonnegative and symmetric, this implies that $\hat K(t)$ is nonnegative and symmetric, hence self-adjoint
(e.g., \cite[Theorems 4.28 and 6.11]{Weidmann}).
By the spectral theorem for compact self-adjoint operators
(e.g., \cite[Theorems 5.4 and 5.6]{Teschl12}),
we then know that there exists an orthonormal basis $(\Psi_k)_{k\in\mbb N}\subset L^2$
and nonnegative numbers $\La_1\geq \La_2\geq \La_3\geq\cdots \geq0$ such that $\hat K(t)$ satisfies
\[\hat K(t)f=\sum_{k=1}^\infty \La_k\langle\Psi_k,f\rangle\Psi_k,\qquad f\in L^2.\]
Consequently, to prove that $\mr e^{-t\hat H}=\hat K(t)$, we need only show that
$\hat K(t)$'s spectral expansion is equivalent to \eqref{Equation: H hat Semigroup Expansion}.

On the one hand, since the Hilbert-Schmidt norm dominates the operator norm,
it follows from \eqref{Equation: Semigroup Convergence Subsequence} that
$\|\hat K_{\eps_n}(t)-\hat K(t)\|_{\mr{op}}\to0$;
hence $\mr e^{-t\la_k({\hat H}_{\eps_n})}\to \La_k$
for all $k\in\mbb N$ by \eqref{Equation: Approximate Spectral Expansion}.
Given that $\la_k({\hat H}_{\eps_n})\to\la_k(\hat H)$
by \eqref{Equation: Operator Convergence}, we conclude that
$\La_k=\mr e^{-t\la_k(\hat H)}$ for all $k\in\mbb N$.
On the other hand, we note that
\begin{align*}
&\|\hat K_{\eps_n}(t)\psi_k({\hat H}_{\eps_n})- \hat K(t) \psi_k({\hat H})\|_2\\
&\leq	\|\hat K_{\eps_n}(t)\psi_k({\hat H}_{\eps_n})-\hat K_{\eps_n}(t)\psi_k({\hat H})\|_2
+\|\hat K_{\eps_n}(t)\psi_k({\hat H})- \hat K(t) \psi_k({\hat H})\|_2\\
&\leq\|\hat K_{\eps_n}(t)\|_{\mr{op}}\|\psi_k({\hat H}_{\eps_n})-\psi_k({\hat H})\|_2+\|\hat K_{\eps_n}(t)- \hat K(t) \|_{\mr{op}}.
\end{align*}
This vanishes as $n\to\infty$ for all $k\in\mbb N$.
Moreover, the spectral expansion \eqref{Equation: Approximate Spectral Expansion}
and the limit \eqref{Equation: Operator Convergence} imply that
\[\lim_{n\to\infty}\hat K_{\eps_n}(t)\psi_k({\hat H}_{\eps_n})
=\lim_{n\to\infty}\mr e^{-t\la_k({\hat H}_{\eps_n})}\psi_k({\hat H}_{\eps_n})
=\mr e^{-t\la_k(\hat H)}\psi_k({\hat H})\]
in $L^2$; hence $ \hat K(t) \psi_k({\hat H})=\mr e^{-t\la_k(\hat H)}\psi_k({\hat H})$. 
Thus $\big(\mr e^{-t\la_k({\hat H})},\psi_k({\hat H})\big)_{k\in\mbb N}$ can be taken as the
eigenvalue-eigenfunction pairs for $\hat K(t)$, concluding the proof
that $\hat K(t)=\mr e^{-t\hat H}$.

\subsubsection{Step 4. Trace Formula}
\label{Section: Main Outline Part 4}

Next, we prove
Theorem \ref{Theorem: Semigroup} (2), that is,
for every $t>0$,
$\mr{Tr}[\mr e^{-t\hat H}]=\int_I\hat K(t;x,x)\d x<\infty$ almost surely.
Let $t>0$ be fixed.
By Propositions \ref{Proposition: Approximate Operator Properties} and \ref{Proposition: Semigroup Convergence},
we can find a vanishing sequence $(\eps_n)_{n\in\mbb N}$
such that \eqref{Equation: Approximate Symmetry Property}--\eqref{Equation: Approximate Spectral Expansion}
hold for all $\eps_n$ and along which
\begin{align}
\label{Equation: Semigroup Convergence Subsequence 2}
\lim_{n\to\infty}\|\hat K_{\eps_n}(t/2)-\hat K(t/2)\|_2=0
\quad\text{and}\quad
\lim_{n\to\infty}\left|\int_I \hat K_{\eps_n}(t;x,x)-\hat K(t;x,x)\d x\right|=0
\end{align}
almost surely. Since $\mr e^{-t\hat H}$ is by definition
a semigroup, we have that
\begin{align}
\label{Equation: Trace = HS halved}
\mr{Tr}[\mr e^{-t\hat H}]=\sum_{k=1}^\infty\left(\mr e^{-(t/2)\la_k(\hat H)}\right)^2=\|\mr e^{-(t/2)\hat H}\|_2^2.
\end{align}
Then, by combining the symmetry and semigroup properties
\eqref{Equation: Approximate Symmetry Property} and \eqref{Equation: Approximate Semigroup Property},
the almost sure convergences  \eqref{Equation: Semigroup Convergence Subsequence 2}, and
the almost sure equality $\hat K(t/2)=\mr e^{-(t/2)\hat H}$ established in the previous step
of this proof, we obtain that
\begin{align*}
&\|\mr e^{-(t/2)\hat H}\|_2^2
=\|\hat K(t/2)\|_2^2
=\lim_{n\to\infty}\|\hat K_{\eps_n}(t/2)\|_2^2\\
&=\lim_{n\to\infty}\int_{I^2}\hat K_{\eps_n}(t/2;x,y)^2\d y\dd x
=\lim_{n\to\infty}\int_I\left(\int_I\hat K_{\eps_n}(t/2;x,y)\hat K_{\eps_n}(t/2;y,x)\d y\right)\dd x\\
&=\lim_{n\to\infty}\int_I\hat K_{\eps_n}(t;x,x)\d x
=\int_I\hat K(t;x,x)\d x
\end{align*}
almost surely. Since we know that $\|\hat K(t/2)\|_2<\infty$
almost surely from the previous step,
this concludes the proof of Theorem \ref{Theorem: Semigroup} (2).

\subsubsection{Step 5. Last Properties}
\label{Section: Main Outline Part 5}

We now conclude the proof of Theorem \ref{Theorem: Semigroup}
by showing that, almost surely, $\mr e^{-t\hat H}$ is a
Hilbert-Schmidt/trace class integral operator for every $t>0$.
By combining \eqref{Equation: Trace = HS halved} with the
fact that every Hilbert-Schmidt operator on $L^2$ has an
integral kernel in $L^2(I\times I)$ (e.g., \cite[Theorem 6.11]{Weidmann}),
we need only prove that, almost surely, $\mr e^{-t\hat H}$ is trace
class for all $t>0$.

In the previous step of this proof, we have already shown the weaker statement
that, for every $t>0$, $\mr{Tr}[\mr e^{-t\hat H}]<\infty$ almost surely. By a countable intersection
we can extend this to the statement that there exists a probability-one event
on which $\mr{Tr}[\mr e^{-t\hat H}]<\infty$ for every $t\in\mbb Q\cap(0,\infty)$.
Since $\la_k(\hat H)\to\infty$ as $k\to\infty$,
there exists some $k_0\in\mbb N$ such that $\la_k(\hat H)>0$ for every $k>k_0$.
Since $\sum_{k=1}^{k_0}\mr e^{-t\la_k(\hat H)}$
is finite for every $t$ and
$\sum_{k=k_0+1}^\infty\mr e^{-t\la_k(\hat H)}$
is monotone decreasing in $t$, the fact that $\mr{Tr}[\mr e^{-t\hat H}]<\infty$
holds for $t\in\mbb Q\cap(0,\infty)$ implies that it holds for all $t>0$,
concluding the proof of Theorem \ref{Theorem: Semigroup}.

\begin{remark}\label{Remark: Comparison 2}
In contrast to the proofs of \cite[Proposition 1.8 (a)]{GaudreauLamarreShkolnikov}
and \cite[Corollary 2.2]{GorinShkolnikov} (which we recall apply to Case 2 with $V(x)=x$),
the argument presented here uses smooth approximations of ${\hat K}(t)$ rather than
random matrix approximations. Since the present paper does not deal with convergence of random matrices,
this choice is natural, and it allows to sidestep several technical difficulties involved with discrete models.
With this said, the proof of \eqref{Equation: Operator Convergence} is inspired by the convergence result
for the spectrum of random matrices in \cite[Section 2]{BloemendalVirag} and \cite[Section 5]{RamirezRiderVirag}.
We refer to Section \ref{Section: Semigroup} for the details.
\end{remark}

\subsection{Proof of Proposition \ref{Proposition: Examples}}\label{Section: Examples Outline}

The main technical result in the proof of Proposition \ref{Proposition: Examples}
is the following estimate, which is a direct consequence of \cite[Lemma 4.2]{GaudreauLamarreGhosalLiao}
(as shown in \cite{GaudreauLamarreGhosalLiao}, \eqref{Equation: gamma to Lp bounds Example} is a
straightforward consequence
of Young's convolution inequality).

\begin{proposition}\label{Proposition: Gamma to Lp Bounds}
Using the notations of Example \ref{Example: Four Noises},
there exists a constant $c_\ga>0$ such that
for every $f\in \mr{PC}_c$, it holds that
\begin{align}
\label{Equation: gamma to Lp bounds Example}
\|f\|_\ga^2\leq
\begin{cases}
c_\ga\|f\|_1^2&\text{(bounded noise)}\\
c_\ga\|f\|_2^2&\text{(white noise)}\\
c_\ga \big(\|f\|_2^2+\|f\|_1^2\big)&\text{(fractional noise with $\mf H\in(\tfrac12,1)$)}\\
c_\ga\big(\sum_{i=1}^\ell \|f\|_{1/(1-1/2p_i)}^2+\|f\|_1^2\big)&\text{($L^p$-singular noise with $p_i\geq1$)}.
\end{cases}
\end{align}
\end{proposition}

Whenever $\ga$ is such that $\langle\cdot,\cdot\rangle_\ga$ is a semi-inner-product,
we know from standard existence theorems that there exists a Gaussian process $\Xi$
on $\mbb R$ with covariance 
\eqref{Equation: Xi Covariance}. As argued in Remark \ref{Remark: Redundancy}, such a process
must have stationary increments.
To see that such $\Xi$ have continuous versions, we note that for any $1\leq q\leq 2$ and
and $x<y$ such that $y-x\leq 1$, one has
$\|\mbf 1_{[x,y)}\|_q^4=(y-x)^{4/q}$ with $4/q>1$.
Thus, given that $1/(1-1/2p)\in(1,2]$ for every $p\geq1$,
it follows from Proposition \ref{Proposition: Gamma to Lp Bounds} that
there exists some constants $c,r>0$ such that
\[\mbf E\big[\big(\Xi(x)-\Xi(y)\big)^4\big]=3!\,\|\mbf 1_{[x,y)}\|_\ga^4\leq c|x-y|^{1+r}\]
for every $x<y\in\mbb R$. The existence of
a continuous version then follows from the
classical Kolmogorov criterion (e.g., \cite[Section 14.1]{MarcusRosen}).

\section{Proof of Propositions \ref{Proposition: Noise Continuity} and \ref{Proposition: Operator}}\label{Section: Operator}

In this section, we complete the proof 
of Propositions \ref{Proposition: Noise Continuity} and \ref{Proposition: Operator}.
Following-up on the outline in Section \ref{Section: Operator Outline},
it only remains to prove Proposition \ref{Proposition: Form Bound}.

\subsection{Step 1. Reduction to a Simple Inequality}

We begin by showing that Proposition \ref{Proposition: Form Bound} can be entirely reduced
to the following claim: Almost surely, for every $\theta>0$, there exists $c=c(\theta)>0$ such that
\begin{align}\label{Equation: Form Bound Reduction}
|\xi(f^2)|\leq\theta\big(\tfrac12\|f'\|_2^2+\|V^{1/2}f\|_2^2\big)+ c\|f\|_2^2
\end{align}
for every $f\in\mr{C}^\infty_0$.

This is easiest to see
in Cases 1, 2-D, and 3-D: On the one hand, in those cases
\eqref{Equation: Form Bound Reduction} directly implies
\eqref{Equation: Infinitesimal Form Bound} for all $f\in\mr{FC}$, which we can
then extend to every $f\in D(\mc E)$ since $\mr{FC}$ is a form core for $\mc E$.
On the other hand,
\eqref{Equation: Form Bound Reduction} implies that
$|\xi(f^2)|\leq\max\{\theta,c\}\|f\|_*^2$, which yields \eqref{Equation: Noise Continuity}.
With \eqref{Equation: Noise Continuity} established, the unique continuous extension
of $\xi(f^2)$ to $\mr H^1_V$ then follows from the fact that
$C_0^\infty$ is dense in the Hilbert space $(\mr H^1_V,\langle\cdot,\cdot\rangle_*)$.

To see how \eqref{Equation: Form Bound Reduction} implies the desired estimate in other cases,
let us consider for example Case 2-R: By \eqref{Equation: Form Bound Reduction},
almost surely,
for every $\bar\theta>0$ there exists $\bar c>0$ such that
\[|\xi(f^2)|\leq\bar\theta\big(\tfrac12\|f'\|_2^2+\|V^{1/2}f\|_2^2\big)+ \bar c\|f\|_2^2
=\bar\theta\mc E(f,f)+\tfrac{\bar\theta\al}2f(0)^2+\bar c\|f\|_2^2.\]
At this point, controlling $f(0)^2$ with Lemma \ref{Lemma: Pointwise Form Bound} yields
the desired estimate (with the straightforward substitution $\theta:=\bar\theta(1+\tfrac{\al\ka}2)$).
Cases 3-R and 3-M can be dealt with in the same way.

\subsection{Step 2. Proof of \eqref{Equation: Form Bound Reduction}}

We now complete the proof of Proposition \ref{Proposition: Form Bound}
by proving \eqref{Equation: Form Bound Reduction}.
We begin with Cases 1 and 2. Following \cite{Minami,RamirezRiderVirag},
we define the integrated process
\[\tilde\Xi(x):=\int_x^{x+1}\Xi(y)\d y,\qquad x\in\mbb R\]
so that we can write $\Xi(x)=\tilde\Xi(x)+\big(\Xi(x)-\tilde\Xi(x)\big)$; hence for every $f\in\mr C^\infty_0$, one has
\[\xi(f^2)=-\langle 2f'f,\tilde\Xi\rangle-\langle 2f'f,\Xi-\tilde\Xi\rangle=
f^2(0)\tilde\Xi(0)+\langle f^2,\tilde\Xi'\rangle+2\langle f'f,\tilde\Xi-\Xi\rangle\]
by Definition \ref{Definition: Operator Noise} and an integration by parts.
By applying Lemma \ref{Lemma: Pointwise Form Bound} to the term
$f^2(0)\tilde\Xi(0)$ (since $|\tilde\Xi(0)|<\infty$ whenever $\Xi$'s path is continuous),
it suffices to prove that almost surely, for every $\theta>0$, there exists $c>0$ such that
\[|\langle f^2,\tilde\Xi'\rangle|+2|\langle f'f,\tilde\Xi-\Xi\rangle|\leq\theta\big(\tfrac12\|f'\|_2^2+\|V^{1/2}f\|_2^2\big)+ c\|f\|_2^2\]
for all $f\in\mr C^\infty_0$.
Thanks to Assumption \ref{Assumption: FN},
the processes $x\mapsto\tilde \Xi'(x)$ and $x\mapsto\tilde \Xi(x)-\Xi(x)$ are
continuous stationary centered Gaussian processes on $\mbb R$,
and thus it follows from standard Gaussian suprema estimates (e.g., Corollary \ref{Corollary: Stationary Noise Suprema}) that
there exists a finite random variable $C>0$ such that,
almost surely,
\begin{align}\label{Equation: Integrated Noise Bounds}
|\tilde\Xi'(x)|,
\big(\Xi(x)-\tilde\Xi(x)\big)^2\leq C\log(2+|x|)
\end{align}
for all $x\in I$. 
Since $V(x)\gg\log|x|$ as $|x|\to\infty$, for every $\theta>0$, there exists $\tilde c_1,\tilde c_2>0$
depending on $\theta$ such that
\begin{align}
\label{Equation: Integrated Noise Bounds 2}
C\log(2+|x|)\leq\tfrac\theta2\big(\tilde c_1+V(x)\big),\qquad \sqrt{C\log(2+|x|)}\leq\tfrac\theta2\sqrt{\tilde c_2+V(x)}
\end{align}
for all $x\in I$.
On the one hand, \eqref{Equation: Integrated Noise Bounds} and the above inequality imply that
\[\int_I f(x)^2|\tilde\Xi'(x)|\d x\leq\tfrac\theta2\|V^{1/2}f\|_2^2+\tfrac{\theta\tilde c_1}2\|f\|_2^2.\]
On the other hand, the same inequalities and $|z\bar z|\leq \frac12(z^2+\bar z^2)$ imply
\begin{multline*}
\int_I|f'(x)f(x)|\big|\tilde\Xi(x)-\Xi(x)\big|\d x
\leq\frac\theta2\int_I|f'(x)f(x)|\sqrt{\tilde c_2+V(x)}\d x\\
\leq\frac\theta2\left(\int_If'(x)^2\d x+\int_If(x)^2\big(\tilde c_2+V(x)\big)\d x\right)
\leq\tfrac\theta2\big(\|f'\|_2^2+\|V^{1/2}f\|_2^2\big)+\tfrac{\theta\tilde c_2}2\|f\|_2^2,
\end{multline*}
concluding the proof.

Suppose then that we are in Case 3. Since $\Xi$ is almost surely continuous
by Assumption \ref{Assumption: FN},
the random variable $C:=\sup_{0\leq x\leq b}|\Xi(x)|$ is finite, and thus
\[|\xi(f^2)|\leq 2C\int_0^b|f'(x)|\,|f(x)|\d x+|\Xi(b)|f(b)^2.\]
An application of the bound $|f'|\,|f|\leq \frac{\ka}2(f')^2+\frac{1}{2\ka}f^2$
for arbitrary $\ka>0$ followed by
Lemma \ref{Lemma: Pointwise Form Bound} to $f(b)^2$ then
yield an upper bound of the form $|\xi(f^2)|\leq\tfrac\theta2\|f'\|_2^2+ c\|f\|_2^2$,
which is better than \eqref{Equation: Form Bound Reduction}.

\section{Proof of Theorem \ref{Theorem: Semigroup}}\label{Section: Semigroup}

In this section, we complete the outline for the proof of Theorem \ref{Theorem: Semigroup}
provided in Section \ref{Section: Semigroup Outline} by proving Propositions
\ref{Proposition: Approximate Operator Properties}--\ref{Proposition: Semigroup Convergence}.
This is done in Sections
\ref{Section: First Semigroup Proposition}--\ref{Section: Last Semigroup Proposition} below.
Before we do this, however, we need several technical results regarding the deterministic semigroup $\mr e^{-t H}$
and the behaviour of the local times $L_t(Z)$ and $\mf L_t(Z)$.
This is done in Sections \ref{Section: First Semigroup Technical}--\ref{Section: Last Semigroup Technical}.

\subsection{Feynman-Kac Formula for Deterministic Operators}\label{Section: First Semigroup Technical}

We begin by recording some standard results in semigroup theory.
By the Feynman-Kac formula, we expect that $\mr e^{-t H}=K(t)$
for the kernels $K(t)$ defined as follows:

\begin{definition}\label{Definition: Deterministic Kernel}
With the same notations as in Definitions \ref{Definition: Stochastic Processes Etc}
and \ref{Definition: Random Semigroup},
for every $t>0$, we define the kernel $ K(t):I^2\to\mbb R$ as
\begin{align}
\label{Equation: Deterministic Kernel}
K(t;x,y):=\begin{cases}
\Pi_B(t;x,y)\,\mbf E^{x,y}_t\Big[\mr e^{-\langle L_t(B),V\rangle}\Big]&\text{(Case 1)}
\vspace{5pt}
\\
\Pi_X(t;x,y)\,\mbf E^{x,y}_t\Big[\mr e^{-\langle L_t(X),V\rangle+\bar\al \mf L^0_t(X)}\Big]&\text{(Case 2)}
\vspace{5pt}
\\
\Pi_Y(t;x,y)\,\mbf E^{x,y}_t\Big[\mr e^{-\langle L_t(Y),V\rangle+\bar\al \mf L^0_t(Y)+\bar\be \mf L^b_t(Y)}\Big]&\text{(Case 3)}
\end{cases}
\end{align}
\end{definition}

To prove this,
we begin with a reminder regarding the Kato class of potentials.

\begin{definition}
We define the Kato class, which we denote by $\mc K=\mc K(I)$, as the collection of nonnegative functions $f:I\to\mbb R$
such that
\begin{align}\label{Equation: Kato Class}
\sup_{x\in I}\int_{\{y\in I:|x-y|\leq 1\}}f(y)\d y<\infty.
\end{align}
We use $\mc K_{\mr{loc}}=\mc K_{\mr{loc}}(I)$ to denote the class of $f$'s such that
$f\mbf 1_K\in \mc K$ for every compact subset $K$ of $I$'s closure.
\end{definition}

\begin{remark}
There is a large diversity of equivalent definitions of the Kato class, some of which
are probabilistic. See, for instance,
\cite[Section A.2]{Simon}.
\end{remark}

\begin{theorem}\label{Theorem: Deterministic Feynman-Kac}
If $V\in\mc K_{\mr{loc}}$, then $\mr e^{- tH}=K(t)$ for all $t>0$.
Moreover,
\begin{align}
\label{Equation: Deterministic Symmetry Property}
&K(t;x,y)=K(t;y,x),&& t>0,~x,y\in I;\\
\label{Equation: Deterministic Semigroup Property}
&\int_IK(t;x,z)K(\bar t;z,y)\d z=K(t+\bar t;x,y),&& t,\bar t>0,~x,y\in I.
\end{align}
\end{theorem}

While Theorem \ref{Theorem: Deterministic Feynman-Kac}
follows from standard functional-analytic methods
(e.g., \cite{ChungZhao}), we were not able to locate an exact statement in the literature
that covers Cases 2-R and 3-M.
We provide a full proof and references in Appendix \ref{Appendix: F-K}.

It is easy to see from \eqref{Equation: Kato Class} that locally integrable functions are in $\mc K_{\mr{loc}}$
so that, by Assumption \ref{Assumption: PG}, $V\in\mc K_{\mr{loc}}$. Therefore, we have the following immediate
consequence of Theorem \ref{Theorem: Deterministic Feynman-Kac}:

\begin{corollary}\label{Corollary: Feynman-Kac with V}
Theorem \ref{Theorem: Deterministic Feynman-Kac}
holds under Assumptions \ref{Assumption: DB} and \ref{Assumption: PG}.
\end{corollary}

\subsection{Reflected Brownian Motion Couplings}\label{Section: Couplings}

The local time process of the Brownian motion $B$ is much more well studied than that of its reflected
versions $X$ or $Y$. Thus, it is convenient to reduce statements regarding the local times of the latter into statements
concerning the local time of $B$. In order to achieve this, we use the following couplings of $B$ with $X$ and $Y$.

\subsubsection{Half-Line}

For any $x>0$, we can couple $B$ and $X$ in such a way that
$X^x(t)=|B^x(t)|$ for every $t\geq0$. In particular, for any functional $F$ of
Brownian paths, one has
\begin{align}\label{Equation: Reflected Motion 0}
\mbf E^x[F(X)]=\mbf E^x[F(|B|)].
\end{align}
Under the same coupling, we observe that for every positive $x$, $y$, and $t$, one has 
\[X_t^{x,y}\deq\big(|B^x|\,\big|\,B^x(t)\in\{-y,y\}\big).\]
Note that
\[\mbf P\big[B^x(t)=y\,\big|\,B^x(t)\in\{-y,y\}\big]
=\frac{\ms G_t(x-y)}{\ms G_t(x-y)+\ms G_t(x+y)}
=\frac{\Pi_B(t;x,y)}{\Pi_X(t;x,y)},\]
and similarly,
\[\mbf P\big[B^x(t)=-y\,\big|\,B^x(t)\in\{-y,y\}\big]
=\frac{\Pi_B(t;x,-y)}{\Pi_X(t;x,y)}.\]
Therefore, for any path functional $F$, it holds that
\begin{multline}\label{Equation: Reflected Motion 1}
\Pi_X(t;x,y)\,\mbf E^{x,y}_t[F(X)]=\Pi_X(t;x,y)\,\mbf E\big[F(|B^x|)|B^x(t)\in\{-y,y\}\big]\\
=\Pi_B(t;x,y)\,\mbf E^{x,y}_t\big[F(|B|)\big]+\Pi_B(t;x,-y)\,\mbf E^{x,-y}_t\big[F(|B|)\big].
\end{multline}

According to the strong Markov property and the symmetry about $0$ of
Brownian motion, we note the equivalence of conditionings
\begin{align}\label{Equation: Equivalence of Conditionings}
\big(|B^x|\,\big|\,B^x(t)=-y\big)
\deq
\big(|B^x|\,\big|\,\tau_0(B^x)<t\text{ and }B^x(t)=y\big),
\end{align}
where we define the hitting time $\tau_0$ as in Remark \ref{Remark: Infinity Times Zero}.
Indeed, we can obtain the left-hand side of \eqref{Equation: Equivalence of Conditionings} from
the right-hand side by reflecting $(B^x|B^x(t)=-y)$ after it first
hits zero and then taking an absolute value (see Figure \ref{Figure: Reflected half-space} below for an illustration).
Since
\[\mbf P[\tau_0(B^x)<t|B^x(t)=y]^{-1}\,\Pi_B(t;x,-y)=\mr e^{2xy/t}\,\Pi_B(t;x,-y)=\Pi_B(t;x,y)\]
(this is easily computed from the joint density of the running maximum and current
value of a Brownian motion \cite[Chapter III, Exercise 3.14]{RevuzYor}),
we see that
\[\Pi_B(t;x,-y)\,\mbf E^{x,-y}_t\big[F(|B|)\big]=\Pi_B(t,x,y)\,\mbf E^{x,y}_t\big[\mbf 1_{\{\tau_0(B)<t\}}F(|B|)\big].\]
Thus \eqref{Equation: Reflected Motion 1} becomes
\begin{align}\label{Equation: Reflected Motion 2}
\Pi_X(t;x,y)\,\mbf E^{x,y}_t[F(X)]
=\Pi_B(t;x,y)\,\mbf E^{x,y}_t\big[(1+\mbf 1_{\{\tau_0(B)<t\}})F(|B|)\big].
\end{align}
Finally, given that $\Pi_B(t;x,y)/\Pi_X(t;x,y)\leq1$,
if $F\geq0$, then
\eqref{Equation: Reflected Motion 2} yields the inequality
\begin{align}\label{Equation: Reflected Motion 3}
\mbf E^{x,y}_t[F(X)]\leq2\mbf E^{x,y}_t[F(|B|)].
\end{align}

\begin{figure}[htbp]
\begin{center}
\input{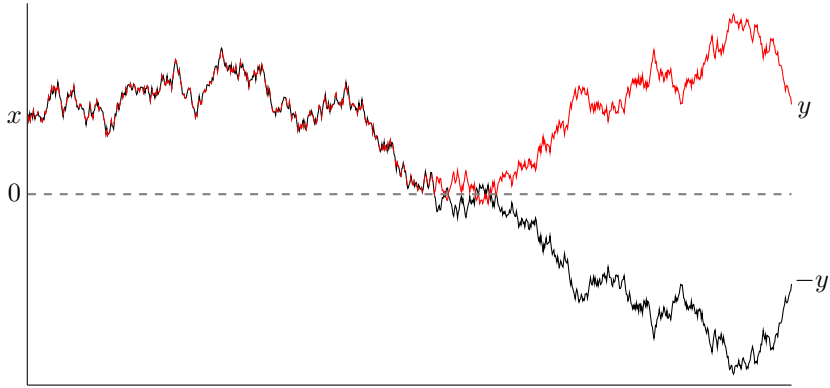}
\caption{Reflection Principle: The path of $B^{x,-y}_t$ (black) and its reflection after the first passage to zero
(red).}
\label{Figure: Reflected half-space}
\end{center}
\end{figure}

\subsubsection{Bounded Interval}

For any $x\in(0,b)$, we can couple $Y^x$ and $B^x$ by reflecting the path of the latter
on the boundary of $(0,b)$, that is,
\begin{align}
\label{Equation: B Y Coupling}
Y^x(t):=\begin{cases}
B^x(t)-2kb&\text{if }B^x(t)\in[2kb,(2k+1)b],\quad k\in\mbb Z,\\
|B^x(t)-2kb|&\text{if }B^x(t)\in[(2k-1)b,2kb],\quad k\in\mbb Z.
\end{cases}
\end{align}
(See Figure \ref{Figure: Reflected Interval} below for an illustration of this coupling.)
Under this coupling, it is clear that for any $z\in(0,b)$, we have
\begin{align}\label{Equation: Interval Local Time}
L_t^z(Y^x)=\sum_{a\in2b\mbb Z\pm z}L^a_t(B^x).
\end{align}

\begin{figure}[htbp]
\begin{center}
\input{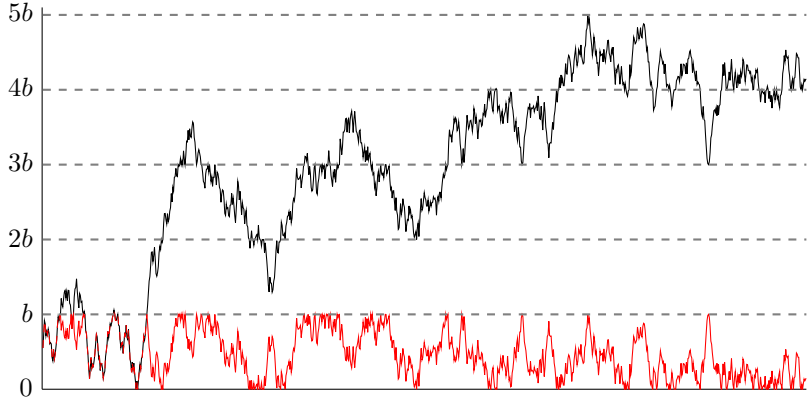}
\caption{Path of $B^x$ (black) and its reflection on the boundary of $(0,b)$ (red).}
\label{Figure: Reflected Interval}
\end{center}
\end{figure}

\subsection{Boundary Local Time}

In this section, we control the exponential moments of the boundary local time of the reflected
paths $X$ and $Y$.

\begin{lemma}\label{Lemma: Regular Local Time}
For every $\theta,t>0$ and $c\in\{0,b\}$, it holds that
\begin{align}\label{Equation: Regular Local Time 1}
\sup_{x\in(0,\infty)}\mbf E^x\big[\mr e^{\theta \mf L^0_t(X)}\big],
\sup_{x\in(0,b)}\mbf E^x\big[\mr e^{\theta \mf L^c_t(Y)}\big]&<\infty.
\end{align}
\end{lemma}
\begin{proof}
We begin by proving \eqref{Equation: Regular Local Time 1}
in Case 2 (i.e., the process $X$).
By \eqref{Equation: Reflected Motion 0} it suffices to prove that
\[\sup_{x\in(0,\infty)}\mbf E^x\big[\mr e^{\theta \mf L^0_t(B)}\big]<\infty\]
for every $\theta,t>0$,
where
\[\mf L^0_t(B):=\lim_{\eps\to0}\frac1{2\eps}\int_0^t\mbf 1_{\{-\eps<B(s)<\eps\}}\d s.\]
On the one hand,
by Brownian scaling, we have the equality in law
\begin{align}\label{Equation: Local Time Rescaling}
\mf L_t^0(B^x_t)\deq t^{1/2}\mf L_1^0(B^{t^{-1/2}x}_1).
\end{align}
On the other hand, according to \cite[(1)]{Pitman},
for every $x,y\in\mbb R$ and $\ell>0$, one has
\[\mbf P[\mf L_1^0(B^x)\in\dd\ell,B^x(1)\in\dd y]
=\frac{(|x|+|y|+\ell)\mr e^{-(|x|+|y|+\ell)^2/2}}{\sqrt{2\pi}}\d\ell\dd y;\]
integrating out the $y$ variable then yields
\begin{align}\label{Equation: Local Time Density}
\mbf P[\mf L^0_1(B^x)\in\dd\ell]=\frac{2\mr e^{-(|x|+\ell)^2/2}}{\sqrt{2\pi}}.
\end{align}
Thanks to \eqref{Equation: Local Time Rescaling} and 
\eqref{Equation: Local Time Density}, we see that
\[\sup_{x\in(0,\infty)}\mbf E^x_t\big[\mr e^{\theta\mf L^0_t(B)}\big]
\leq\mbf E^0_1\big[\mr e^{\theta t^{1/2}\mf L^0_1(B)}\big]<\infty\]
for every $\theta,t>0$;
hence \eqref{Equation: Regular Local Time 1} holds in Case 2.

The proof of \eqref{Equation: Regular Local Time 1} for Case 3 (i.e., the process $Y$)
follows directly from \cite[(2.18) and (3.11')]{PapanicolaouThesis},
which states that there exists constants $K,K'>0$ (depending
on $\theta$) such that
$\mbf E^x\big[\mr e^{\theta \mf L^c_t(Y)}\big]\leq K'\mr e^{Kt}$ for all $t>0$
and $x\in(0,b)$.
\end{proof}

Next, we aim to extend the result of Lemma \ref{Lemma: Regular Local Time}
to the local time of the bridge processes $Z^{x,x}_t$. Before we can do this,
we need the following estimate on $\Pi_Z$.

\begin{lemma}\label{Lemma: Transition Bounds}
For every $t>0$, it holds that
\begin{align}\label{Equation: Midpoint Transition Bound}
\mf s_t(Z):=\sup_{x,y\in I}\frac{\Pi_Z(t/2;x,y)}{\Pi_Z(t;x,x)}<\infty.
\end{align}
\end{lemma}
\begin{proof}
In all three cases, $\Pi_Z(t;x,x)\geq1/\sqrt{2\pi t}$,
and thus it suffices to prove that
\begin{align}
\label{Equation: Midpoint Transition Bound 2}
\sup_{(x,y)\in I^2}\Pi_Z(t;x,y)<\infty.
\end{align}
In Cases 1 \& 2, this is trivial. In Case 3,
we recall that, by definition,
\[\Pi_Y(t;x,y):=\sum_{z\in2b\mbb Z\pm y}\ms G_t(x-z)=\frac{1}{\sqrt{2\pi t}}\left(\sum_{k\in\mbb Z}\mr e^{-(x+y-2bk)^2/2t}+\mr e^{-(x-y-2bk)^2/2t}\right).\]
According to the integral test for series convergence, we note that
for every $b,t>0$ and $z\in\mbb R$, it holds that
\[\sum_{k=\lceil -z/2b\rceil}^\infty\frac{\mr e^{-(z+2bk)^2/2t}}{\sqrt{2\pi t}}
\leq\frac{\mr e^{-(z+2b\lceil -z/2b\rceil)^2/2t}}{\sqrt{2\pi t}}+\int_{\lceil -z/2b\rceil}^\infty\frac{\mr e^{-(z+2bu)^2/2t}}{\sqrt{2\pi t}}\d u
\leq\frac{1}{\sqrt{2\pi t}}+\frac1b,\]
and similarly for the sum from $k=-\infty$ to $\lfloor-z/2b\rfloor$;
hence \eqref{Equation: Midpoint Transition Bound 2} holds.
\end{proof}

We finish this section with the following.

\begin{lemma}\label{Lemma: Bridge Local Time}
For every $\theta,t>0$ and $c\in\{0,b\}$, it holds that
\begin{align}\label{Equation: Bridge Local Time 1}
\sup_{x\in(0,\infty)}\mbf E^{x,x}_t\big[\mr e^{\theta \mf L^0_t(X)}\big],
\sup_{x\in(0,b)}\mbf E^{x,x}_t\big[\mr e^{\theta \mf L^c_t(Y)}\big]&<\infty.
\end{align}
\end{lemma}
\begin{proof}
As it turns out,
\eqref{Equation: Bridge Local Time 1}
follows from Lemma \ref{Lemma: Regular Local Time}.
The trick that we
use to prove this makes several other appearances in this paper:
Since the exponential function is nonnegative, for every $\theta>0$,
an application of the tower property and the Doob $h$-transform yields
\begin{multline}
\label{Equation: Tower+Doob}
\mbf E^{x,x}_{t}\left[\mr e^{\theta \mf L^c_t(Z)}\right]
=\mbf E\left[\mbf E^{x,x}_{t}\Big[\mr e^{\theta \mf L^c_t(Z)}\big|Z^{x,x}_{t}(t/2)\Big]\right]\\
=\int_I\mbf E^{x,x}_{t}\Big[\mr e^{\theta \mf L^c_t(Z)}\big|Z^{x,x}_{t}(t/2)=y\Big]\frac{\Pi_Z(t/2;x,y)\Pi_Z(t/2;y,x)}{\Pi_Z(t;x,x)}\d y.
\end{multline}
If we condition on $Z^{x,x}_{t}(t/2)=y$, then the path segments
\[\big(Z^{x,x}_{t}(s):0\leq s\leq t/2\big)
\qquad\text{and}\qquad
\big(Z^{x,x}_{t}(t/2+s):0\leq s\leq t/2\big)\]
are independent of each other and have respective distributions $Z^{x,y}_{t/2}$
and $Z^{y,x}_{t/2}$. Since $\Pi_Z(t/2;\cdot,\cdot)$ is symmetric for every $t>0$, the time-reversed process
$s\mapsto Z^{y,x}_{t/2}(t/2-s)$ (with $0\leq s\leq t/2$)
is equal in distribution to $Z^{x,y}_{t/2}$.
Thus,
\begin{align}\label{Equation: Midpoint Trick Final}
\nonumber
\mbf E^{x,x}_{t}\Big[\mr e^{\theta \mf L^c_t(Z)}\big|Z^{x,x}_{t}(t/2)=y\Big]
&=\mbf E^{x,x}_{t}\Big[\mr e^{\theta(\mf L^c_{[0,t/2]}(Z)+\mf L^c_{[t/2,t]}(Z))}\big|Z^{x,x}_{t}(t/2)=y\Big]\\
&=\mbf E^{x,y}_{t/2}\Big[\mr e^{\theta \mf L^c_{t/2}(Z)}\Big]^2
\leq\mbf E^{x,y}_{t/2}\Big[\mr e^{2\theta \mf L^c_{t/2}(Z)}\Big],
\end{align}
where
the equality in \eqref{Equation: Midpoint Trick Final} follows from independence and the fact that local time is invariant
with respect to time reversal,
and the last term in \eqref{Equation: Midpoint Trick Final} follows from Jensen's inequality.

Let us define the constant $\mf s_t(Z)<\infty$ as in \eqref{Equation: Midpoint Transition Bound}.
According to \eqref{Equation: Tower+Doob} and \eqref{Equation: Midpoint Trick Final},
we then have that for $t>0$,
\begin{multline}
\label{Equation: Midpoint Trick 3}
\mbf E^{x,x}_{t}\left[\mr e^{\theta \mf L^c_t(Z)}\right]
\leq \mf s_t(Z)\int_I\mbf E^{x,y}_{t/2}\left[\mr e^{2\theta \mf L^c_{t/2}(Z)}\right]\Pi_Z(t/2;x,y)\d y\\
=\mf s_t(Z)\,\mbf E^{x}_{t/2}\left[\mr e^{2\theta \mf L^c_{t/2}(Z)}\right].
\end{multline}
Hence the present result is a direct consequence of Lemma \ref{Lemma: Regular Local Time}.
\end{proof}

\subsection{$L^q$ Norms of Local Time}

In this section, we obtain bounds on the exponential moments of the $L^q$ norms of the local times of $B$, $X$, and $Y$.
Such results for $B^x$ are well known (see, for instance, \cite[Section 4.2]{Chen}).
For $X$ and $Y$ and the bridge processes, we rely on the couplings introduced in Section \ref{Section: Couplings}
and the midpoint sampling trick used in the proof Lemma \ref{Lemma: Bridge Local Time},
respectively. Before we state our result, we need the following.

\begin{lemma}
\label{Equation: Non Bridge Midpoint Trick}
For every $\theta,u,v>0$ and $q\geq1$, it holds that
\[\sup_{x\in I}\mbf E^x\left[\mr e^{\theta\|L_{u+v}(Z)\|_q^2}\right]
\leq\left(\sup_{x\in I}\mbf E^x\left[\mr e^{2\theta\|L_u(Z)\|_q^2}\right]\right)
\left(\sup_{x\in I}\mbf E^x\left[\mr e^{2\theta\|L_v(Z)\|_q^2}\right]\right)\]
\end{lemma}
\begin{proof}
Let $x,y\in I$ be fixed.
Conditional on $Z^x(u)=y$,
the path segments
\[\big(Z^x(s):0\leq s\leq u\big)
\qquad\text{and}\qquad
\big(Z^x(u+t):0\leq t\leq\infty\big)\]
are independent of each other and
have respective distributions $Z^{x,y}_u$ and $Z^y$.
Therefore, by the tower property, we have that
\begin{align*}
\mbf E^x\left[\mr e^{\theta\|L_{u+v}(Z)\|_q^2}\right]
&=\int_I\mbf E^x\left[\mr e^{\theta\|L_{u+v}(Z)\|_q^2}\Big|Z^x(u)=y\right]\Pi_Z(u;x,y)\d y\\
&\leq\int_I\mbf E^x\left[\mr e^{2\theta\|L_u(Z)\|_q^2+2\theta\|L_{[u,u+v]}(Z)\|_q^2}\Big|Z^x(u)=y\right]\Pi_Z(u;x,y)\d y\\
&=\int_I\mbf E^{x,y}_u\left[\mr e^{2\theta\|L_u(Z)\|_q^2}\right]\mbf E^y\left[\mr e^{2\theta\|L_v(Z)\|_q^2}\right]\Pi_Z(u;x,y)\d y\\
&\leq\left(\sup_{y\in I}\mbf E^y\left[\mr e^{2\theta\|L_v(Z)\|_q^2}\right]\right)
\int_I\mbf E^{x,y}_u\left[\mr e^{2\theta\|L_u(Z)\|_q^2}\right]\Pi_Z(u;x,y)\d y\\
&=\left(\sup_{y\in I}\mbf E^y\left[\mr e^{2\theta\|L_v(Z)\|_q^2}\right]\right)
\mbf E^{x}\left[\mr e^{2\theta\|L_u(Z)\|_q^2}\right],
\end{align*}
where the second line
follows from Minkowski's inequality and $(z+\bar z)^2\leq 2(z^2+\bar z^2)$,
and the third line follows from conditional
independence of the path segments.
The result then follows by taking a supremum over $x\in I$.
\end{proof}

\begin{lemma}\label{Lemma: Regular Self-Intersection}
Let $1\leq q\leq 2$.
For every $\theta,t>0$, one has
\begin{align}\label{Equation: Regular Self-Intersection}
\sup_{x\in I}\mbf E^x\left[\mr e^{\theta\|L_t(Z)\|_q^2}\right]<\infty.
\end{align}
\end{lemma}
\begin{proof}
We begin by noting that $\|L_t(Z)\|_1=t$
by \eqref{Equation: Interior Local Time},
and thus the result is trivial if $q=1$.
To prove the result for $1<q\leq 2$, we claim that it suffices to show that there exists nonnegative random variables
$R_1,R_2\geq0$ with finite exponential moments in some neighbourhood of zero,
as well as constants $\ka_1,\ka_2>1$ such that
\begin{align}
\label{Equation: Lp Local Time Scaling}
\sup_{x\in I}\mbf E^x\left[\mr e^{\theta\|L_t(Z)\|_q^2}\right]\leq\mbf E\left[\mr e^{\theta t^{\ka_1}R_1}\right]
\end{align}
or
\begin{align}
\label{Equation: Lp Local Time Scaling 2}
\sup_{x\in I}\mbf E^x\left[\mr e^{\theta\|L_t(Z)\|_q^2}\right]\leq\mbf E\left[\mr e^{\theta t^{\ka_1}R_1}\right]^{1/2}
\mbf E\left[\mr e^{\theta t^{\ka_2}R_2}\right]^{1/2}
\end{align}
for all $t>0$.
To see this, suppose \eqref{Equation: Lp Local Time Scaling} holds,
and let $\theta_0>0$ be such that $\mbf E[\mr e^{\theta R_1}]<\infty$
for all $\theta<\theta_0$.
Then, for any fixed $\theta>0$,
\[\sup_{x\in I}\mbf E^x\left[\mr e^{\theta\|L_t(Z)\|_q^2}\right]
\leq\mbf E\left[\mr e^{\theta t^{\ka_1} R_1}\right]<\infty\]
for every $t<(\theta_0/\theta)^{1/\ka_1}$.
In particular, if $u,v\leq(\theta_0/2\theta)^{1/\ka_1}$,
we get from Lemma \ref{Equation: Non Bridge Midpoint Trick} that
\[\sup_{x\in I}\mbf E^x\left[\mr e^{\theta\|L_{u+v}(Z)\|_q^2}\right]
\leq\left(\sup_{x\in I}\mbf E^x\left[\mr e^{2\theta\|L_u(Z)\|_q^2}\right]\right)
\left(\sup_{x\in I}\mbf E^x\left[\mr e^{2\theta\|L_v(Z)\|_q^2}\right]\right)<\infty.\]
Thus, \eqref{Equation: Regular Self-Intersection} now holds
for $t<2(\theta_0/2\theta)^{1/\ka_1}=2^{1-1/\ka_1}(\theta_0/\theta)^{1/\ka_1}$.
Since $\ka_1>1$, $2^{1-1/\ka_1}>1$, and thus
by repeating this procedure infinitely often, we obtain by induction that 
\eqref{Equation: Regular Self-Intersection} holds for all $t>0$,
as desired. Essentially the same argument gives the result
if we instead have \eqref{Equation: Lp Local Time Scaling 2}.

We then prove \eqref{Equation: Lp Local Time Scaling}/\eqref{Equation: Lp Local Time Scaling 2}. We argue
on a case-by-case basis.
Let us begin with {Case 1}.
If we couple $B^x=x+B^0$ for all $x$, then
changes of variables with a Brownian scaling imply that
\[\|L_t(B^x)\|_q^2=\|L_t(B^0)\|_q^2
\deq t\left(\int_{\mbb R} L^{t^{-1/2}a}_1(B^0)^q\d a\right)^{2/q}
=t^{1+1/q}\|L_1(B^0)\|_q^2\]
for every $q>1$.
Thanks to the large deviation result \cite[Theorem 4.2.1]{Chen},
we know that for every $q>1$, there exists some $c_q>0$ such that
\[\mbf P\big[\|L_1(B^0)\|_q^2>u\big]=\mr e^{-c_qu^{q/(q-1)}(1+o(1))},\qquad u\to\infty.\]
Thus, in {Case 1} \eqref{Equation: Lp Local Time Scaling} holds with $R_1=\|L_1(B^0)\|_q^2$
and $\ka_1=1+1/q$.

Consider now {Case 2}. By coupling $X^x(t)=|B^x(t)|$ for all $t>0$, we note that
for every $a>0$, one has $L_t^a(X^x)=L_t^a(|B^x|)=L_t^a(B^x)+L_t^{-a}(B^x).$
Therefore,
\begin{multline*}
\|L_t(X^x)\|_q^2
=\left(\int_0^\infty L^a_t(X^x)^q\d a\right)^{2/q}\\
\leq 2^{2(q-1)/q}\left(\int_0^\infty L^a_t(B^x)^q+L^{-a}_t(B^x)^q\d a\right)^{2/q}
=2^{2(q-1)/2}\|L_t(B^x)\|_q^2.
\end{multline*}
Thus, the proof in {Case 2} follows from {Case 1}.

Finally, consider {Case 3}.
Recall the coupling of $Y^x$ and $B^x$ in \eqref{Equation: B Y Coupling},
which yields the local time identity \eqref{Equation: Interval Local Time}.
The argument that follows is inspired from the proof of \cite[Lemma 2.1]{ChenLi}:
Under the coupling \eqref{Equation: Interval Local Time},
\begin{multline*}
\left(\int_0^b L_t^z(Y^x)^q\d z\right)^{1/q}
=\left(\int_0^b \Big(\sum_{k\in2b\mbb Z} L_t^{k+z}(B^x)+L_t^{k-z}(B^x)\Big)^q\d z\right)^{1/q}\\
\leq2^{(q-1)/q}\sum_{k\in2b\mbb Z}\left(\int_{-b}^bL^{k+z}_t(B^x)^q\d z\right)^{1/q}.
\end{multline*}
Let us denote the maximum and minimum of $B^x$ as
\[M^x(t):=\sup_{s\in[0,t]}B^x(s)
\qquad\text{and}\qquad
m^x(t):=\inf_{s\in[0,t]}B^x(s).\]
In order for the integral $\int_{-b}^bL^{k+z}_t(B^x)^2\d z$ to be different from zero,
it is necessary that
$M^x(t)\geq k-b$ and $m^x(t)\leq k+b$, that is,
$M^x(t)+b\geq k \geq m^x(t)-b$. Consequently,
for every $q>1$, one has
\begin{align*}
&\sum_{k\in2b\mbb Z}\Big(\int_{-b}^bL^{k+z}_t(B^x)^q\d z\Big)^{1/q}\\
&=\sum_{k\in2b\mbb Z}\Big(\int_{-b}^bL^{k+z}_t(B^x)^q\d z\Big)^{1/q}\mbf 1_{\{M^x(t)+b\geq k \geq m^x(t)-b\}}\\
&\leq\Big(\sum_{k\in2b\mbb Z}\int_{-b}^bL^{k+z}_t(B^x)^q\d z\Big)^{1/q}\Big(\sum_{k\in2b\mbb Z}\mbf 1_{\{M^x(t)+b\geq k \geq m^x(t)-b\}}\Big)^{\frac{q-1}{q}}\\
&=\Big(\int_\mbb RL^a_t(B^x)^q\d a\Big)^{1/q}\Big(\sum_{k\in2b\mbb Z}\mbf 1_{\{M^x(t)+b\geq k \geq m^x(t)-b\}}\Big)^{\frac{q-1}{q}}\\
&\leq c_1t^{1/q}\Big(\sup_{a\in\mbb R}L^a_t(B^x)\Big)^{\frac{q-1}{q}}\big(M^x(t)-m^x(t)+c_2\big)^{\frac{q-1}{q}}\\
&\leq c_1t^{1/q}\Bigg(c_2^{\frac{q-1}{q}}\big(\sup_{a\in\mbb R}L^a_t(B^x)\big)^{\frac{q-1}{q}}+\Big(\sup_{a\in\mbb R}L^a_t(B^x)\cdot\big(M^x(t)-m^x(t)\big)\Big)^{\frac{q-1}{q}}\Bigg)
\end{align*}
where $c_1,c_2>0$ only depend on $b$ and $q$:
The inequality on the third line follows from H\"older's inequality;
the equality on the fourth line follows from the fact that $\sum_{k \in 2b\mathbb{Z}}\int^{b}_{-b} L^{a}_t(B^x)^q \d a$ is equal to $\int_\mbb RL^a_t(B^x)^q\d a$;
the inequality on the fifth line follows from the fact that $\int_\mbb RL^a_t(B^x)^q\d a$ is bounded by $(\sup_{a\in\mbb R}L^a_t(B^x))^{q-1} \|L_t(B^x)\|_1$, where $\|L_t(B^x)\|_1=t$;
and the inequality on the last line follows from the fact that
\[\big(M^x(t)-m^x(t)+c_2\big)^{\frac{q-1}{q}}\leq\big(M^x(t)-m^x(t)\big)^{\frac{q-1}{q}}+c^{\frac{q-1}{q}}_2.\]   

By Brownian scaling and translation invariance, we have that
\[t^{1/q}\big(\sup_{a\in\mbb R}L^a_t(B^x)\big)^{\frac{q-1}{q}}\deq t^{1/2+1/2q}\big(\sup_{a\in\mbb R}L^a_1(B^0)\big)^{\frac{q-1}{q}}\]
and
\begin{align}
t^{1/q}\bigg(\sup_{a\in\mbb R}L^a_t(B^x)\cdot\big(M^x(t)-m^x(t)\big)\bigg)^{\frac{q-1}{q}}
\deq t\,\bigg(\sup_{a\in\mbb R}L^a_1(B^0)\cdot\big(M^0(1)-m^0(1)\big)\bigg)^{\frac{q-1}{q}}.\nonumber
\end{align}
Given that $4(\frac{q-1}{q})\leq 2$ for all $q\in(1,2]$
and that there exists $\theta_0>0$ small enough so that
\[\mbf E\left[\exp\left(\theta_0 \sup_{a\in\mbb R}L^a_1(B^0)^2\right)\right],\mbf E\left[\mr e^{\theta_0(M^0(1)-m^0(1))^2}\right]<\infty,\]
(e.g., the proof of \cite[Lemma 2.1]{ChenLi} and references therein)
we finally conclude by H\"older's inequality that \eqref{Equation: Lp Local Time Scaling 2} holds
in Case 3 with
\[R_1=4c_1^2\big(c_2\sup_{a\in\mbb R}L^a_1(B^0)\big)^{2\frac{q-1}{q}},
\quad
R_2=4\bigg(\sup_{a\in\mbb R}L^a_1(B^0)\cdot\big(M^0(1)-m^0(1)\big)\bigg)^{2\frac{q-1}{q}},\]
and $\ka_1=1+1/q$ and $\ka_2=2$.
\end{proof}

\begin{lemma}\label{Lemma: Bridge Self-Intersection}
Let $1\leq q\leq 2$.
For every $\theta,t>0$, one has
\[\sup_{x\in I}\mbf E^{x,x}_t\left[\mr e^{\theta\|L_t(Z)\|_q^2}\right]<\infty.\]
\end{lemma}
\begin{proof}
Once again, the present result follows from
Lemma \ref{Lemma: Regular Self-Intersection}.
To see this, we use the same trick employed in the proof
of Lemma \ref{Lemma: Bridge Local Time}:
For every $\theta>0$,
the tower property and the Doob $h$-transform yields
\[\mbf E^{x,x}_{t}\left[\mr e^{\theta\|L_{t}(Z)\|_q^2}\right]
=\int_I\mbf E^{x,x}_{t}\Big[\mr e^{\theta\|L_{t}(Z)\|_q^2}\big|Z^{x,x}_{t}(t/2)=y\Big]\frac{\Pi_Z(t/2;x,y)\Pi_Z(t/2;y,x)}{\Pi_Z(t;x,x)}\d y.\]
Arguing as in the passage following \eqref{Equation: Tower+Doob},
\begin{align*}
\mbf E^{x,x}_{t}\Big[\mr e^{\theta\|L_{t}(Z)\|_q^2}\big|Z^{x,x}_{t}(t/2)=y\Big]
&=\mbf E^{x,x}_{t}\Big[\mr e^{\theta\|L_{t/2}(Z)+L_{[t/2,t]}(Z)\|_q^2}\big|Z^{x,x}_{t}(t/2)=y\Big]\\
&\leq\mbf E^{x,x}_{t}\Big[\mr e^{2\theta(\|L_{t/2}(Z)\|_q^2+\|L_{[t/2,t]}(Z)\|_q^2)}\big|Z^{x,x}_{t}(t/2)=y\Big]\\
&=\mbf E^{x,y}_{t/2}\Big[\mr e^{2\theta\|L_{t/2}(Z)\|_q^2}\Big]^2\leq\mbf E^{x,y}_{t/2}\Big[\mr e^{4\theta\|L_{t/2}(Z)\|_q^2}\Big],
\end{align*}
where the inequality on the second line follows from a combination the triangle inequality
and $(z+\bar z)^2\leq2(z^2+\bar z^2)$,
the equality on the third line follows from independence and invariance of local time
under time reversal,
and the inequality on the third line follows from Jensen's inequality.

With $\mf s_t(Z)$ as in \eqref{Equation: Midpoint Transition Bound},
similarly to \eqref{Equation: Midpoint Trick 3} we then have the upper bound
\[\mbf E^{x,x}_{t}\left[\mr e^{\theta\|L_{t}(Z)\|_q^2}\right]\leq\mf s_t(Z)\,\mbf E^x\big[\mr e^{4\theta\|L_{t/2}(Z)\|_q^2}\big]\]
for every $t>0$;
whence the present result readily follows from
Lemma \ref{Lemma: Regular Self-Intersection}.
\end{proof}

\subsection{Compactness Properties of Deterministic Kernels}\label{Section: Last Semigroup Technical}

We now conclude the proofs of our technical results
with some estimates regarding the integrability/compactness
of the deterministic kernels \eqref{Definition: Deterministic Kernel}.
In this section and several others, to alleviate notation,
we introduce the following shorthand.

\begin{notation}
For every $t>0$, we define the path functional
\begin{align}
\label{Equation: Mathfrak A,B Shortcut}
\mf A_t(Z)&:=
\begin{cases}
-\langle L_t(B),V\rangle&\text{(Case 1)}\\
-\langle L_t(X),V\rangle+\bar\al\mf L^{0}_t(X)&\text{(Case 2)}\\
-\langle L_t(Y),V\rangle+\bar\al\mf L^{0}_t(Y)+\bar\be\mf L^{b}_t(Y)&\text{(Case 3)}
\end{cases}
\end{align}
\end{notation}

\begin{lemma}\label{Lemma: Deterministic Compactness}
For every $p\geq1$ and $t>0$,
\[\int_I \Pi_Z(t;x,x)\,\mbf E^{x,x}_t\left[\mr e^{p\mf A_t(Z)}\right]^{1/p}\d x<\infty.\]
\end{lemma}
\begin{proof}
Let us begin with Case 1. By Assumption \ref{Assumption: PG}, for every $c_1>0$, there exists
$c_2>0$ large enough so that $V(x)\geq c_1\log(1+|x|)-c_2$ for every $x\in \mbb R$.
Therefore, we have
\begin{align*}
\Pi_Z(t;x,x)\,\mbf E^{x,x}_t\left[\mr e^{p\mf A_t(B)}\right]^{1/p}
&\leq \frac{\mr e^{c_2 t}}{\sqrt{2\pi t}}\,\mbf E^{x,x}_t\left[\exp\left(-pc_1\int_0^t\log\big(1+|B(s)|\big)\d s\right)\right]^{1/p}\\
&=\frac{\mr e^{c_2 t}}{\sqrt{2\pi t}}\,\mbf E^{0,0}_t\left[\exp\left(-pc_1\int_0^t\log\big(1+|x+B(s)|\big)\d s\right)\right]^{1/p}.
\end{align*}
By using the inequalities
\begin{align}
\label{Equation: Log Lower Bound}
\log(1+|x+z|)\geq\log(1+|x|)-\log(1+|z|)\geq\log(1+|x|)-|z|,
\end{align}
which are valid for all $z\in\mbb R$, we get the further upper bound
\[\frac{\mr e^{c_2 t-c_1 t\log(1+|x|)}}{\sqrt{2\pi t}}\,\mbf E^{0,0}_t\left[\exp\left(pc_1\int_0^t|B(s)|\d s\right)\right]^{1/p}.\]
On the one hand, a Brownian scaling implies that
\begin{multline}\label{Equation: Bessel Bridge Area}
\mbf E^{0,0}_t\left[\exp\left(pc_1\int_0^t|B(s)|\d s\right)\right]\\
=\mbf E^{0,0}_1\left[\exp\left(t^{3/2}pc_1\int_0^1|B(s)|\d s\right)\right]
\leq \mbf E\left[\exp\left(t^{3/2}pc_1\mc S\right)\right],
\end{multline}
where $\mc S=\sup_{s\in[0,1]}|B_1^{0,0}(s)|$. Note that $s\mapsto|B^{0,0}_1(s)|$ is a Bessel bridge of dimension one (see, for instance, \cite[Chapter XI]{RevuzYor}).
Consequently, we know that \eqref{Equation: Bessel Bridge Area} is finite for any $t,p,c_1>0$ thanks to the tail asymptotic for $\mc S$ in \cite[Remark 3.1]{GruetShi}
(the Bessel bridge is denoted by $\rho$ in that paper).
On the other hand, for any $t>0$, we can choose $c_1>0$ large enough so that
\[\int_{\mbb R}e^{-c_1 t\log(1+|x|)}\d x=\int_\mbb R(1+|x|)^{-c_1t}\d x<\infty,\]
concluding the proof in Case 1.

For Case 2, by H\"older's inequality, we have that
\begin{multline*}
\Pi_X(t;x,x)\,\mbf E^{x,x}_t\left[\mr e^{p\mf A_t(X)}\right]^{1/p}\\
\leq\Pi_X(t;x,x)\,\mbf E^{x,x}_t\Big[\mr e^{-\langle L_t(X),2pV\rangle}\Big]^{1/2p}\sup_{x\in(0,\infty)}\mbf E^{x,x}_t\Big[\mr e^{2p\bar\al \mf L^0_t(X)}\Big]^{1/2p}.
\end{multline*}
The supremum of exponential moments of local time can be bounded by a direct application of Lemma \ref{Lemma: Bridge Local Time}. Then, by \eqref{Equation: Reflected Motion 3}, we have that
\[\int_0^\infty\Pi_X(t;x,x)\,\mbf E^{x,x}_t\Big[\mr e^{-\langle L_t(X),2pV\rangle}\Big]^{1/2p}\d x
\leq\frac{2\sqrt{2}}{\sqrt{\pi t}}\int_0^\infty\mbf E^{x,x}_t\Big[\mr e^{-\langle L_t(|B|),2pV\rangle}\Big]^{1/2p}\d x.\]
This term can be controlled in the same way as Case 1.

For Case 3, since $I=(0,b)$ is finite and $V\geq0$ (hence $\mr e^{-\langle L_t(Y),pV\rangle}\leq1$),
\begin{multline*}
\int_0^b\Pi_Y(t;x,x)\,\mbf E^{x,x}_t\Big[\mr e^{-\langle L_t(Y),pV\rangle+p\bar\al \mf L^0_t(Y)+p\bar\be \mf L^b_t(Y)}\Big]^{1/p}\d x\\
\leq b\left(\sup_{x\in(0,b)}\Pi_Y(t;x,x)\right)\left(\sup_{x\in(0,b)}\mbf E^{x,x}_t\Big[\mr e^{p\bar\al \mf L^0_t(Y)+p\bar\be \mf L^b_t(Y)}\Big]^{1/p}\right).
\end{multline*}
This is finite by Lemmas \ref{Lemma: Transition Bounds} and \ref{Lemma: Bridge Local Time}.
\end{proof}

\subsection{Proof of Proposition \ref{Proposition: Approximate Operator Properties}}
\label{Section: First Semigroup Proposition}
\label{Section: Proof of Operator Convergence}

Suppose we can prove that for every $\eps>0$, the potential $V+\Xi_\eps'$
satisfies Assumption \ref{Assumption: PG} with probability one
(up to a random additive constant, making it nonnegative). Then,
by Proposition \ref{Proposition: Classical Schrodinger Operator},
the ${\hat H}_\eps$ are self-adjoint with compact resolvent.
Moreover, $\hat K_\eps(t)=\mr e^{-t\hat H_\eps}$
and the properties
\eqref{Equation: Approximate Symmetry Property}--\eqref{Equation: Approximate Spectral Expansion}
then follow from Corollary \ref{Corollary: Feynman-Kac with V},
and the fact that $\mr e^{-t\hat H}$ is trace class follows
from Lemma \ref{Lemma: Deterministic Compactness}
in the case $p=1$. Thus, it only remains to prove the following:

\begin{lemma}\label{Lemma: Smoothed Potential Assumptions}
For every $\eps>0$, there exists a random $c=c(\eps)\geq0$ such that
the potential $V+\Xi_\eps'+c$ satisfies Assumption \ref{Assumption: PG}
with probability one.
\end{lemma}
\begin{proof}
Since $\Xi_\eps'$ is continuous, $V+\Xi_\eps'$ is locally integrable on $I$'s closure.
Moreover, if we prove that $|\Xi_\eps'(x)|\ll\log|x|$ as $x\to\pm\infty$,
then the continuity of $\Xi_\eps'$ also implies that $V+\Xi_\eps'$ is bounded below and is such that
\[\lim_{x\to\pm\infty}\frac{V(x)+\Xi_\eps'(x)}{\log|x|}=\infty;\]
hence we can take
\[c(\eps):=\max\left\{0,-\inf_{x\in I}\big(V(x)+\Xi_\eps'(x)\big)\right\}<\infty.\]
The fact that $|\Xi_\eps'(x)|\ll\log|x|$ follows from Corollary \ref{Corollary: Stationary Noise Suprema},
since $\Xi_\eps'$ is stationary.
\end{proof}

\subsection{Proof of Proposition \ref{Proposition: Operator Convergence}}

Our proof of this result is similar to \cite[Section 2]{BloemendalVirag}
and \cite[Section 5]{RamirezRiderVirag}, save for the fact that we use smooth approximations
instead of discrete ones. We provide the argument in full.
Following \cite[Fact 2.2]{BloemendalVirag} and \cite[Fact 2.2]{RamirezRiderVirag},
we begin by recording some compactness properties of $\|\cdot\|_*$.

\begin{lemma}\label{Lemma: L Star Compactness}
If $(f_n)_{n\in\mbb N}\subset D(\mc E)$ is such that $\sup_n\|f_n\|_*<\infty$, then
there exists $f\in D(\mc E)$ and a subsequence $(n_i)_{i\in\mbb N}$ along which
\begin{enumerate}
\item $\displaystyle\lim_{i\to\infty}\|f_{n_i}-f\|_2=0;$
\item $\displaystyle\lim_{i\to\infty}\langle g,f'_{n_i}\rangle=\langle g,f'\rangle$ for every $g\in L^2$;
\item $\displaystyle\lim_{i\to\infty}f_{n_i}=f$ uniformly on compact sets; and
\item $\displaystyle\lim_{i\to\infty}\langle g,f_{n_i}\rangle_*=\langle g,f\rangle_*$
for every $g\in D(\mc E)$.
\end{enumerate}
\end{lemma}
\begin{proof}
(2) and (4) follow from the Banach-Alaoglu theorem. Next, by combining
Lemma \ref{Lemma: Pointwise Form Bound} with the estimate
\begin{align}
\label{Equation: Uniformly Equicontinuous}
|f_n(x)-f_n(y)|\leq\int_x^y|f'_n(x)|\d x\leq\|f_n'\|_2\sqrt{|y-x|},
\end{align}
we may extract a further subsequence along which (3) holds
by the Arzel\`a-Ascoli theorem.
Finally, in Case 3, (3) immediately implies (1), whereas in Cases 1 and 2,
by combining (3) with the Vitali convergence
theorem, to prove (1) it suffices to show that for every $\eps>0$, there exists
$K>0$ large enough and $\de>0$ small enough so that
\[\int_{I\setminus[-K,K]}f_{n_i}(x)^2\d x\leq\eps^2
\qquad\text{and}\qquad
\sup_{x\in I}\int_{I\cap[x-\de,x+\de]}f_{n_i}(x)^2\d x\leq\eps^2.\]
The first of these conditions follows from the fact that $\sup_n\|V^{1/2}f_n\|_2<\infty$
and that $V(x)\gg\log|x|$; the second follows from the uniform bound in Lemma \ref{Lemma: Pointwise Form Bound}.
\end{proof}

\begin{remark}\label{Remark: Compactness to Deterministic Convergence}
It is easy to see by definition of $\langle\cdot,\cdot\rangle_*$
that if $f_n\to f$ in the sense of Lemma \ref{Lemma: L Star Compactness} (1)--(4),
then for every $g\in \mr{FC}$, one has
\[\lim_{n\to\infty}\mc E(g,f_n)=\mc E(g,f).\]
\end{remark}

We can reformulate Proposition \ref{Proposition: Form Bound} in terms of $\|\cdot\|_*$ thusly:

\begin{lemma}\label{Lemma: Convergence Form Bound}
There exist finite random variables $c_1,c_2,c_3>0$ such that
\[c_1\|f\|_*^2-c_2\|f\|_2^2\leq\mc {\hat E}(f,f)\leq c_3\|f\|_*^2,\qquad f\in D(\mc E).\]
\end{lemma}

We also have the following finite $\eps$ variant:

\begin{lemma}\label{Lemma: Convergence Approximate Form Bound}
There exist finite random variables $\tilde c_1,\tilde c_2,\tilde c_3>0$ such that
for every $\eps\in(0,1]$,
\[\tilde c_1\|f\|_*^2-\tilde c_2\|f\|_2^2\leq\mc {\hat E}_\eps(f,f)\leq\tilde c_3\|f\|_*^2,\qquad f\in D(\mc E).\]
\end{lemma}
\begin{proof}
By repeating the proof of Proposition \ref{Proposition: Form Bound}, we only need to prove
that for every $\theta>0$, there exists $c>0$ large enough so that
\[|\langle f^2,\Xi_\eps'\rangle|\leq\theta\big(\tfrac12\|f'\|_2^2+\|V^{1/2}f\|_2^2\big)+c\|f\|_2^2\]
for every $\eps\in(0,1]$ and $f\in\mr C^\infty_0$.
Let us define
\[\tilde \Xi_\eps(x):=\int_x^{x+1}\Xi_\eps(y)\d y.\]
Arguing as in the proof of Proposition \ref{Proposition: Form Bound},
it suffices to show that
\[\sup_{\eps\in(0,1]}\sup_{0\leq x\leq b}|\Xi_\eps(x)|<\infty\]
almost surely and that there exist finite random variables $C>0$ and $u>1$ independent of $\eps\in(0,1]$
such that for every $x\in\mbb R$,
\[\sup_{y\in[0,1]}|\Xi_\eps(x+y)-\Xi_\eps(x)|\leq C\sqrt{\log(u+|x|)}.\]

Let $K>0$ be such that $\mr{supp}(\rho)\subset[-K,K]$ so that
$\mr{supp}(\rho_\eps)\subset[-K,K]$ for all $\eps\in(0,1]$.
On the one hand, since the $\rho_\eps$ integrate to one,
\[\sup_{\eps\in(0,1]}\sup_{0\leq x\leq b}\left|\int_{\mbb R}\Xi(x-y)\rho_\eps(y)\d y\right|
\leq\sup_{-K\leq x\leq b+K}|\Xi(x)|<\infty.\]
On the other hand, by Corollary \ref{Corollary: Stationary Noise Suprema}
and Remark \ref{Remark: Stronger Stationary Noise Suprema}, for every $x\in I$ and $\eps\in(0,1]$,
one has
\begin{multline*}
\sup_{y\in[0,1]}\left|\int_{\mbb R}\big(\Xi(x+y-z)-\Xi(x-z)\big)\rho_\eps(z)\d y\right|\\
\leq\sup_{w\in[x-K,x+K]}\sup_{y\in[0,1]}|\Xi(w+y)-\Xi(w)|
\leq\sup_{w\in[x-K,x+K]}C\sqrt{\log(2+|w|)},
\end{multline*}
which yields the desired estimate.
\end{proof}

\begin{remark}
We see from Lemma \ref{Lemma: Convergence Approximate Form Bound} that
the forms
$(f,g)\mapsto\langle fg,\Xi_\eps'\rangle$ are {\it uniformly form-bounded} in $\eps\in(0,1]$ by $\mc E$,
 in the sense that there exists a $0<\theta<1$ and a random $c>0$ independent
of $\eps$ such that
\[|\langle f^2,\Xi_\eps'\rangle|\leq\theta\mc E(f,f)+c\|f\|_2^2,\qquad f\in D(\mc E),~\eps\in(0,1].\]
Among other things, this implies by the variational principle
(see, for example, the estimate in \cite[Theorem XIII.68]{ReedSimonIV})
that for every $k\in\mbb N$ and $\eps\in(0,1]$, one has
\begin{align}\label{Equation: Uniform Eigenvalue Bounds}
(1-\theta)\la_k( H)-c\leq\la_k({\hat H}_\eps)\leq(1+\theta)\la_k( H)+c.
\end{align}
\end{remark}

Finally, we need the following convergence result.

\begin{lemma}
Almost surely, for every $f,g\in \mr{FC}$, it holds that
\begin{align}\label{Equation: Noise Convergence 1}
\lim_{\eps\to0}\langle fg,\Xi_\eps'\rangle=\xi(fg).
\end{align}
Moreover, if $(\eps_n)_{n\in\mbb N}\subset(0,1]$ converges
to zero, $\sup_n\|f_n\|_*<\infty$, and $f_n\to f$ in the sense of Lemma
\ref{Lemma: L Star Compactness} (1)--(4),
then almost surely,
\begin{align}\label{Equation: Noise Convergence 2}
\lim_{n\to\infty}\langle f_ng,\Xi_{\eps_n}'\rangle=\xi(fg)
\end{align}
for every $g\in \mr{FC}$.
\end{lemma}
\begin{proof}
Clearly $\Xi_\eps\to\Xi$ pointwise, hence
for \eqref{Equation: Noise Convergence 1} it suffices to prove that
\[\lim_{\eps\to0}\int_\mbb R\big((f'g+fg')*\rho_\eps\big)(x)\Xi(x)\d x=\langle f'g+fg',\Xi\rangle.\]
Since $f'g+fg'$ is compactly supported and $\Xi$ is continuous
(hence bounded on compacts), the result follows by dominated convergence.

Let us now prove \eqref{Equation: Noise Convergence 2}.
Using again the fact that $g$ and $g'$ are compactly supported, we know that there
exists a compact $K\subset\mbb R$ (in Case 3 we may simply take $K=[0,b]$) such that
\[\langle f_n'g+f_ng',\Xi*\rho_{\eps_n}\rangle\\
=\int_K \big(f_n'(x)g(x)+f_n(x)g'(x)\big)(\Xi*\rho_{\eps_n})(x)\d x\]
and similarly with $f_n$ replaced by $f$ and $\Xi*\rho_{\eps_n}$ replaced by $\Xi$.
Given that, as $n\to\infty$, $\Xi*\rho_{\eps_n}\mbf 1_K\to \Xi\mbf 1_K$ in $L^2$,
$f_n'g+f_ng'\to f'g+fg'$ weakly in $L^2$, and $\sup_n\|f_n'g+f_ng'\|_2<\infty$,
we conclude that
\[\lim_{n\to\infty}\langle f_n'g+f_ng',\Xi*\rho_{\eps_n}\rangle=\langle f'g+fg',\Xi\rangle.\]
Hence \eqref{Equation: Noise Convergence 2} holds.
\end{proof}

We finally have all the necessary ingredients to prove the spectral convergence.
We first prove that there exists a subsequence $(\eps_n)_{n\in\mbb N}$ such that
\begin{align}\label{Equation: Spectrum Convergence Part 1}
\liminf_{n\to\infty}\la_k({\hat H}_{\eps_n})\geq\la_k({\hat H})
\end{align}
for every $k\in\mbb N$.

\begin{remark}
For the sake of readability, we henceforth denote any subsequence and further subsequences
of $(\eps_n)_{n\in\mbb N}$ as $(\eps_n)_{n\in\mbb N}$ itself.
\end{remark}

According to \eqref{Equation: Uniform Eigenvalue Bounds},
the $\la_k({\hat H}_\eps)$ are uniformly bounded, and thus it follows from the
Bolzano-Weierstrass theorem that, along a subsequence $\eps_n$, the limits
\[\lim_{n\to\infty}\la_k({\hat H}_{\eps_n})=:l_k\]
exist and are finite for every $k\in\mbb N$, where $-\infty<l_1\leq l_2\leq\cdots$.
Since the eigenvalues are bounded, it follows from Lemma
\ref{Lemma: Convergence Approximate Form Bound} that the eigenfunctions $\psi_k({\hat H}_\eps)$
are bounded in $\|\cdot\|_*$-norm uniformly in $\eps\in[0,1)$, and thus there exist
functions $f_1,f_2,\ldots$ and a further subsequence along which $\psi_k({\hat H}_{\eps_n})\to f_k$
for every $k$ in the sense of Lemma \ref{Lemma: L Star Compactness} (1)--(4). By combining Remark
\ref{Remark: Compactness to Deterministic Convergence} and \eqref{Equation: Noise Convergence 2},
this means that
\[l_k\langle g,f_k\rangle=\lim_{n\to\infty}\la_k(\hat H_{\eps_n})\langle g,\psi_k(\hat H_{\eps_n})\rangle=\lim_{n\to\infty}\hat{\mc E}_{\eps_n}(g,\psi_k({\hat H}_{\eps_n}))=\hat{\mc E}(g,f_k)\]
for all $k\in\mbb N$ and $g\in \mr{FC}$.
That is, $(l_k,f_k)_{k\in\mbb N}$ consists of eigenvalue-eigenfunction pairs of ${\hat H}$, though these pairs
may not exhaust the full spectrum. Since the $l_k$ are arranged in increasing order,
this implies that $l_k\geq\la_k({\hat H})$ for every $k\in\mbb N$, which proves \eqref{Equation: Spectrum Convergence Part 1}.

We now prove that we can take ${\hat H}'s$ eigenfunctions in such a way that, along a further subsequence,
\begin{align}\label{Equation: 1D Convergence Induction Hypothesis}
\limsup_{n\to\infty}\la_k({\hat H}_{\eps_n})\leq\la_k({\hat H})
\qquad\text{and}\qquad
\lim_{n\to\infty}\|\psi_k({\hat H}_{\eps_n})-\psi_k({\hat H})\|_2=0
\end{align}
for every $k\in\mbb N$.
We proceed by induction. Suppose that \eqref{Equation: 1D Convergence Induction Hypothesis} holds up to $k-1$ (if $k=1$ then we consider the base case). Let $\psi$ be an eigenfunction of $\la_k({\hat H})$ orthogonal to $\psi_1({\hat H}),\ldots,\psi_{k-1}({\hat H})$,
and for every $\theta>0$, let $\phi_\theta\in \mr{FC}$ be such that
$\|\phi_\theta-\psi\|_*<\theta$.
Let us define the projections
\[\pi_{\eps_n}(\phi_\theta):=\phi_\theta-\sum_{\ell=1}^{k-1}\langle \psi_\ell({\hat H}_{\eps_n}),\phi_\theta\rangle\psi_\ell({\hat H}_{\eps_n})\]
of $\phi_\theta$ onto the orthogonal of $\psi_1({\hat H}_{\eps_n}),\ldots,\psi_{k-1}({\hat H}_{\eps_n})$
(if $k=1$, then we simply have $\pi_{\eps_n}(\phi_\theta)=\phi_\theta$). Then, by the variational principle, for any $\theta>0$,
\begin{align}\label{Equation: 1D Convergence Induction 1}
\limsup_{n\to\infty}\la_k({\hat H}_{\eps_n})\leq\limsup_{n\to\infty}\frac{\hat{\mc E}_{\eps_n}(\pi_{\eps_n}(\phi_\theta),\pi_{\eps_n}(\phi_\theta))}{\|\pi_{\eps_n}(\phi_\theta)\|_2^2}.
\end{align}

Given that $\|\psi_\ell({\hat H}_{\eps_n})-\psi_\ell({\hat H})\|_2\to0$ for every $\ell\leq k-1$,
one has
\[\lim_{\theta\to0}\lim_{n\to\infty}\pi_{\eps_n}(\phi_\theta)=\psi\]
in $L^2$.
Moreover, the convergence of the $\la_\ell({\hat H}_{\eps_n})$ and Lemma \ref{Lemma: Convergence Approximate Form Bound}
imply that the $(\psi_\ell({\hat H}_{\eps_n}))_{\ell=1,\ldots,k-1}$ are uniformly bounded in
$\|\cdot\|_*$-norm,
and thus
\[\lim_{\theta\to0}\limsup_{n\to\infty}\left\|\sum_{\ell=1}^{k-1}\langle \psi_\ell({\hat H}_{\eps_n}),\phi_\theta\rangle\psi_\ell({\hat H}_{\eps_n})\right\|_*=0.\]
We recall that, by Lemma \ref{Lemma: Convergence Approximate Form Bound}, the maps 
$f\mapsto\hat{\mc E}_\eps(f,f)$
are continuous with respect to $\|\cdot\|_*$ uniformly in $\eps\in(0,1]$. Consequently,
\[\limsup_{n\to\infty}\la_k({\hat H}_{\eps_n})\leq\limsup_{\theta\to0}\limsup_{n\to\infty}\frac{\hat{\mc E}_{\eps_n}(\pi_{\eps_n}(\phi_\theta),\pi_{\eps_n}(\phi_\theta))}{\|\pi_{\eps_n}(\phi_\theta)\|_2^2},\]
since \eqref{Equation: 1D Convergence Induction 1} holds for any $\theta>0$.
Then, if we use \eqref{Equation: Noise Convergence 1} to compute the supremum limit in $n$,
followed by Lemma \ref{Lemma: Convergence Form Bound} for the limit in $\theta$ (recall that $\|\phi_\theta-\psi\|_*\to0$ as $\theta\to0$),
we conclude that
\[\limsup_{n\to\infty}\la_k({\hat H}_{\eps_n})\leq\hat{\mc E}\big(\psi,\psi\big)=\la_k({\hat H}).\]

Since $\liminf_{n\to\infty}\la_k(\hat H_{\eps_n})\geq\la_k(\hat H)$ by the previous step,
we now know that $\la_k({\hat H}_{\eps_n})\to\la_k({\hat H})$ as $n\to\infty$. Thus, according to Lemma \ref{Lemma: Convergence Approximate Form Bound},
the eigenfunctions $(\psi_k({\hat H}_{\eps_n}))_{n\in\mbb N}$ are uniformly bounded in $\|\cdot\|_*$-norm.
Thus, there exists $\bar\psi\in D(\mc E)$ such that $\psi_k(\hat H_{\eps_n})\to\bar\psi$
in the sense of Lemma \ref{Lemma: L Star Compactness} along a further subsequence.
Combining this with
Remark \ref{Remark: Compactness to Deterministic Convergence} and \eqref{Equation: Noise Convergence 2},
and the fact that $\la_k({\hat H}_{\eps_n})\to\la_k({\hat H})$,
we then also have
\[\hat{\mc E}(g,\bar\psi)
=\lim_{n\to\infty}\hat{\mc E}_{\eps_n}(g,\psi_k({\hat H}_{\eps_n}))
=\lim_{n\to\infty}\la_k(\hat H_{\eps_n})\langle g,\psi_k({\hat H}_{\eps_n})\rangle
=\la_k(\hat H)\langle g,\bar\psi\rangle\]
for all $g\in \mr{FC}$.
In particular, $\bar\psi$
must be an eigenfunction for $\la_k({\hat H})$, which is orthogonal to $\psi_1({\hat H}),\ldots,\psi_{k-1}({\hat H})$.
Thus we may take $\psi_k({\hat H}):=\bar\psi$, concluding the proof of the proposition
since Lemma \ref{Lemma: L Star Compactness} includes $L^2$ convergence.

\subsection{Proof of Proposition \ref{Proposition: Semigroup Convergence} Part 1}

We begin by proving \eqref{Equation: Semigroup L2 in Expectation}.

\subsubsection{Step 1. Computation of Expected $L^2$ Norm}

Our first step in the proof of \eqref{Equation: Semigroup L2 in Expectation}
is to obtain a formula for $\mbf E\big[\|\hat K_\eps(t)-\hat K(t)\|_2^2\big]$
that is amenable to analysis, namely:
\begin{align}
\label{Equation: L2 Convergence Formula}
&\mbf E\big[\|\hat K_\eps(t)-\hat K(t)\|_2^2\big]\\
\nonumber
&=\int_I
\Pi_Z(2t;x,x)\,\mbf E^{x,x}_{2t}\bigg[\mr e^{\mf A_{2t}(Z)}
\bigg(\mr e^{\frac12\| L_{2t}(Z)*\rho_\eps\|_\ga^2}
-2\mr e^{\frac12\| L_t(Z)*\rho_\eps+L_{[t,2t]}(Z)\|_\ga^2}
+\mr e^{\frac12\| L_{2t}(Z)\|_\ga^2}\bigg)\bigg]\dd x,
\end{align}
where we recall the notation of $\mf A_t(Z)$ from \eqref{Equation: Mathfrak A,B Shortcut}.
We now prove \eqref{Equation: L2 Convergence Formula}.

By Fubini's theorem,
\begin{align*}
&\mbf E\big[\|\hat K_\eps(t)-\hat K(t)\|_2^2\big]\\
&=\int_{I^2}\mbf E[\hat K_\eps(t;x,y)^2]
-2\mbf E[\hat K_\eps(t;x,y)\hat K(t;x,y)]
+\mbf E[\hat K(t;x,y)^2]\d y\dd x\\
&=\int_{I^2}\Pi_Z(t;x,y)^2\,
\mbf E\bigg[\mr e^{\mf A_t(Z^{1;x,y}_t)+\mf A_t(Z^{2;x,y}_t)}
\bigg(\mbf E_\Xi\bigg[\mr e^{-\langle L_t(Z^{1;x,y}_t)+L_t(Z^{2;x,y}_t),\Xi_\eps'\rangle}\bigg]\\
&\hspace{8pt}-2\mbf E_\Xi\bigg[\mr e^{-\langle L_t(Z^{1;x,y}_t),\Xi_\eps'\rangle-\xi(L_t(Z^{2;x,y}_t))}\bigg]
+\mbf E_\Xi\bigg[\mr e^{-\xi(L_t(Z^{1;x,y}_t))-\xi(L_t(Z^{2;x,y}_t))}\bigg]\bigg)\bigg]\d y\dd x,
\end{align*}
where $Z^{i;x,y}_t$ ($i=1,2$) are i.i.d. copies of $Z^{x,y}_t$
that are independent of $\Xi$, and $\mbf E_\Xi$ denotes
the expected value with respect to $\Xi$ conditional on $Z^{i;x,y}_t$.
For every $f_1,f_2\in\mr{PC}_c$, the sum $\xi(f_1)+\xi(f_2)$
is Gaussian with mean zero and variance $\sum_{i,j=1}^2\langle f_i,f_j\rangle_\ga=\|f_1+f_2\|_\ga^2$.
Thanks to \eqref{Equation: Crucial Coupling}, a straightforward
Gaussian moment generating function computation then yields
that $\mbf E\big[\|\hat K_\eps(t)-\hat K(t)\|_2^2\big]$ is equal to
\begin{multline*}
\int_{I^2}
\Pi_Z(t;x,y)^2\mbf E\bigg[\mr e^{\mf A_t(Z^{1;x,y}_t)+\mf A_t(Z^{2;x,y}_t)}
\bigg(\mr e^{\frac12\| L_t(Z^{1;x,y}_t)*\rho_\eps+L_t(Z^{2;x,y}_t)*\rho_\eps\|_\ga^2}\\
-2\mr e^{\frac12\| L_t(Z^{1;x,y}_t)*\rho_\eps+L_t(Z^{2;x,y}_t)\|_\ga^2}
+\mr e^{\frac12\| L_t(Z^{1;x,y}_t)+L_t(Z^{2;x,y}_t)\|_\ga^2}\bigg)\bigg]\d y\dd x.
\end{multline*}
By symmetry of $\Pi_Z(t)$, the time-reversed process
$s\mapsto Z^{2;x,y}_t(t-s)$ is equal in distribution to $Z^{2;y,x}_t$;
since local time is invariant under time-reversal,
$\mbf E\big[\|\hat K_\eps(t)-\hat K(t)\|_2^2\big]$ now equals
\begin{multline}
\label{Equation: L2 Convergence Pre}
\int_{I^2}
\Pi_Z(t;x,y)\Pi_Z(t;y,x)\mbf E\bigg[\mr e^{\mf A_t(Z^{1;x,y}_t)+\mf A_t(Z^{2;y,x}_t)}
\bigg(\mr e^{\frac12\| L_t(Z^{1;x,y}_t)*\rho_\eps+L_t(Z^{2;y,x}_t)*\rho_\eps\|_\ga^2}\\
-2\mr e^{\frac12\| L_t(Z^{1;x,y}_t)*\rho_\eps+L_t(Z^{2;y,x}_t)\|_\ga^2}
+\mr e^{\frac12\| L_t(Z^{1;x,y}_t)+L_t(Z^{2;y,x}_t)\|_\ga^2}\bigg)\bigg]\d y\dd x.
\end{multline}

As noted in the proof of Lemma \ref{Lemma: Bridge Local Time},
for every $x,y\in I$ and $t>0$,
if we condition the path $Z^{x,x}_{2t}$ on $Z^{x,x}_{2t}(t)=y$, then the path segments
\begin{align*}
\big(Z^{x,x}_{2t}(s):0\leq s\leq t\big)
\qquad\text{and}\qquad
\big(Z^{x,x}_{2t}(t+s):0\leq s\leq t\big)
\end{align*}
have the same joint distribution as
$Z^{1;x,y}_t$ and $Z^{2;y,x}_{t}$.
Moreover, $Z^{x,x}_{2t}(t)$ has density
\[y\mapsto\frac{\Pi_Z(t;x,y)\Pi_Z(t;y,x)}{\Pi_Z(2t;x,x)}\]
by the Doob $h$-transform.
Given that the functions $f\mapsto\langle f,V\rangle$,
$f\mapsto \bar \al f$ and $f\mapsto\bar\be f$ are all linear in $f$,
and that local time is additive in the sense that
\[L_{[u,v]}(Z)+L_{[v,w]}(Z)=L_{[u,w]}(Z)
\qquad\text{and}\qquad
\mf L^c_{[u,v]}(Z)+\mf L^c_{[v,w]}(Z)=\mf L^c_{[u,w]}(Z)\] for all $0<u<v<w$,
\eqref{Equation: L2 Convergence Formula} is then
a consequence of applying Fubini's theorem to \eqref{Equation: L2 Convergence Pre} with the rearrangement
\[\Pi_Z(t;x,y)\Pi_Z(t;y,x)=\Pi_Z(2t;x,x)\frac{\Pi_Z(t;x,y)\Pi_Z(t;y,x)}{\Pi_Z(2t;x,x)}.\]
Indeed, we note for instance that
\begin{align*}
&\int_I\Pi_Z(2t;x,x)\Bigg(\int_I\frac{\Pi_Z(t;x,y)\Pi_Z(t;y,x)}{\Pi_Z(2t;x,x)}\\
&\hspace{1.5in}\cdot\mbf E\bigg[\mr e^{\mf A_t(Z^{1;x,y}_t)+\mf A_t(Z^{2;y,x}_t)+\frac12\| L_t(Z^{1;x,y}_t)*\rho_\eps+L_t(Z^{2;y,x}_t)*\rho_\eps\|_\ga^2}\bigg]\d y\Bigg)\d x\\
&=\int_I\Pi_Z(2t;x,x)\Bigg(\int_I\mbf E\bigg[\mr e^{\mf A_{2t}(Z^{x,x}_{2t})+\frac12\|L_{2t}(Z^{x,x}_{2t})*\rho_\eps\|_\ga^2}\bigg|Z^{x,x}_{2t}(t)=y\bigg]
\mbf P[Z^{x,x}_{2t}(t)\in\dd y]\d y\Bigg)\d x\\
&=\int_I\Pi_Z(2t;x,x)\mbf E\bigg[\mr e^{\mf A_{2t}(Z^{x,x}_{2t})+\frac12\| L_{2t}(Z^{x,x}_{2t})*\rho_\eps\|_\ga^2}\bigg]\d x;
\end{align*}
a similar argument applied to the terms on the second line of \eqref{Equation: L2 Convergence Pre}
then yields \eqref{Equation: L2 Convergence Formula}.

\subsubsection{Step 2. Convergence Inside Expectation}

With \eqref{Equation: L2 Convergence Formula} in hand, our second step to prove
\eqref{Equation: Semigroup L2 in Expectation} is to show that,
for every $x\in I$, we have the almost sure limit
\begin{align}
\label{Equation: Convergence Inside Expected Value}
\lim_{\eps\to0}\mr e^{\frac12\| L_t(Z^{x,x}_{2t})*\rho_\eps\|_\ga^2}
-2\mr e^{\frac12\| L_t(Z^{x,x}_{2t})*\rho_\eps+L_{[t,2t]}(Z^{x,x}_{2t})\|_\ga^2}
+\mr e^{\frac12\| L_{2t}(Z^{x,x}_{2t})\|_\ga^2}=0.
\end{align}
This is a simple consequence of \eqref{Equation: gamma to Lp Bounds}
coupled with the fact that if $f\in L^q$ for some $q\geq1$,
then $\|f*\rho_\eps-f\|_q\to0$ as $\eps\to0$.

\subsubsection{Step 3. Convergence Inside Integral}
\label{Section: Convergence Inside Integral}

Our next step is to prove that for every $x\in I$,
we have the limit in expectation
\begin{align}
\label{Equation: Convergence Inside Integral}
\lim_{\eps\to0}\mbf E^{x,x}_{2t}\bigg[\mr e^{\mf A_{2t}(Z)}
\bigg(\mr e^{\frac12\| L_{2t}(Z)*\rho_\eps\|_\ga^2}
-2\mr e^{\frac12\| L_t(Z)*\rho_\eps+L_{[t,2t]}(Z)\|_\ga^2}
+\mr e^{\frac12\| L_{2t}(Z)\|_\ga^2}\bigg)\bigg]=0.
\end{align}
Thanks to \eqref{Equation: Convergence Inside Expected Value},
for this it suffices to prove that the prelimit variables in
\eqref{Equation: Convergence Inside Integral} are uniformly
integrable in $\eps>0$, which itself can be reduced to the
claim that
\[\sup_{\eps>0}\mbf E^{x,x}_{2t}\bigg[\mr e^{2\mf A_{2t}(Z)}
\bigg(\mr e^{\frac12\| L_{2t}(Z)*\rho_\eps\|_\ga^2}
-2\mr e^{\frac12\| L_t(Z)*\rho_\eps+L_{[t,2t]}(Z)\|_\ga^2}
+\mr e^{\frac12\| L_{2t}(Z)\|_\ga^2}\bigg)^2\bigg]<\infty.\]
By combining H\"older's inequality
with $(z-2\bar z+\tilde z)^2\leq16(z^2+\bar z^2+\tilde z^2)$, it is enough to prove that
\begin{align}
\label{Equation: Convergence Inside Integral UI 1}
\mbf E^{x,x}_{2t}\bigg[\mr e^{4\mf A_{2t}(Z)}\bigg]<\infty
\end{align}
and
\begin{multline}
\label{Equation: Convergence Inside Integral UI 2}
\sup_{\eps>0}\mbf E^{x,x}_{2t}\left[\mr e^{\| L_{2t}(Z)*\rho_\eps\|_\ga^2}\right],~
\sup_{\eps>0}\mbf E^{x,x}_{2t}\left[\mr e^{\| L_t(Z)*\rho_\eps+L_{[t,2t]}(Z)\|_\ga^2}\right],\\
\mbf E^{x,x}_{2t}\left[\mr e^{\| L_{2t}(Z)\|_\ga^2}\right]<\infty.
\end{multline}

By combining the assumption $V\geq0$ (hence $\mr e^{-4\langle L_{2t}(Z),V\rangle}\leq1$)
and Lemma \ref{Lemma: Bridge Local Time}, we immediately obtain \eqref{Equation: Convergence Inside Integral UI 1}.
Next, it follows from \eqref{Equation: gamma to Lp Bounds} that
\[\mbf E^{x,x}_{2t}\left[\mr e^{\| L_{2t}(Z)\|_\ga^2}\right]\leq\mbf E^{x,x}_{2t}\left[\mr e^{c_\ga\sum_{i=1}^\ell\|L_{2t}(Z)\|_{q_i}^2}\right].\]
This is finite by Lemma \ref{Lemma: Bridge Self-Intersection} since $1\leq q_i\leq 2$ for all $1\leq i\leq\ell$.
According to Young's convolution inequality, the fact that the $\rho_\eps$ integrate to one
implies that $\|f*\rho_\eps\|_q\leq\|f\|_q\|\rho_\eps\|_1\leq\|f\|_q$. Thus,
it follows from \eqref{Equation: gamma to Lp Bounds} that
\[\sup_{\eps>0}\mbf E^{x,x}_{2t}\left[\mr e^{\| L_{2t}(Z)*\rho_\eps\|_\ga^2}\right]
\leq\mbf E^{x,x}_{2t}\left[\mr e^{c_\ga\sum_{i=1}^\ell\|L_{2t}(Z)\|_{q_i}^2}\right]<\infty.\]
Since $\|\cdot\|_\ga$ is a seminorm,
it satisfies the triangle inequality, and thus
\begin{align*}
\| L_t(Z^{x,x}_{2t})*\rho_\eps+L_{[t,2t]}(Z^{x,x}_{2t})\|_\ga^2
\leq2\|L_t(Z^{x,x}_{2t})*\rho_\eps\|_\ga^2+2\|L_{[t,2t]}(Z^{x,x}_{2t})\|_\ga^2.
\end{align*}
Given that $L_t(Z^{x,x}_{2t})$ and $L_{[t,2t]}(Z^{x,x}_{2t})$ are both
smaller than $L_{2t}(Z^{x,x}_{2t})$, applying once again \eqref{Equation: gamma to Lp Bounds}
and Young's inequality yields
\[\sup_{\eps>0}\mbf E^{x,x}_{2t}\left[\mr e^{\| L_t(Z)*\rho_\eps+L_{[t,2t]}(Z)\|_\ga^2}\right]
\leq\mbf E^{x,x}_{2t}\left[\mr e^{4c_\ga\sum_{i=1}^\ell\|L_{2t}(Z)\|_{q_i}^2}\right],\]
which is finite by Lemma \ref{Lemma: Bridge Self-Intersection}.
We therefore conclude that \eqref{Equation: Convergence Inside Integral UI 2} holds,
and thus \eqref{Equation: Convergence Inside Integral} as well.

\subsubsection{Step 4. Convergence of the Integral}

Our final step in the proof of \eqref{Equation: Semigroup L2 in Expectation}
is to show that \eqref{Equation: L2 Convergence Formula} converges to zero.
Given \eqref{Equation: Convergence Inside Integral}, by applying the dominated
convergence theorem, it suffices to find an integrable function that
dominates
\begin{align}
\label{Equation: Dominate Function}
\Pi_Z(2t;x,x)\mbf E^{x,x}_{2t}\bigg[\mr e^{\mf A_{2t}(Z)}
\bigg(\mr e^{\frac12\| L_{2t}(Z)*\rho_\eps\|_\ga^2}
-2\mr e^{\frac12\| L_t(Z)*\rho_\eps+L_{[t,2t]}(Z)\|_\ga^2}
+\mr e^{\frac12\| L_{2t}(Z)\|_\ga^2}\bigg)\bigg]
\end{align}
for every $\eps>0$.
By Holder's inequality, this is bounded by
\begin{multline}
\label{Equation: Convergence of Integral}
\Pi_Z(2t;x,x)\mbf E^{x,x}_{2t}\bigg[\mr e^{2\mf A_{2t}(Z)}\bigg]^{1/2}\\
\cdot\sup_{\eps>0,~x\in I}\mbf E^{x,x}_{2t}\bigg[\bigg(\mr e^{\frac12\| L_{2t}(Z)*\rho_\eps\|_\ga^2}
-2\mr e^{\frac12\| L_t(Z)*\rho_\eps+L_{[t,2t]}(Z)\|_\ga^2}
+\mr e^{\frac12\| L_{2t}(Z)\|_\ga^2}\bigg)^2\bigg]^{1/2}.
\end{multline}
uniformly in $\eps>0$.
Thanks to Lemma \ref{Lemma: Deterministic Compactness}
with $p=2$, the first line of \eqref{Equation: Convergence of Integral} is integrable.
Then, by arguing in exactly the same way as in Section \ref{Section: Convergence Inside Integral}
(i.e., Young's convolution inequality, \eqref{Equation: gamma to Lp Bounds}, etc.),
the term on the second line of \eqref{Equation: Convergence of Integral} is bounded by
\[\sup_{x\in I}C\mbf E^{x,x}_{2t}\left[\mr e^{\theta c_\ga\sum_{i=1}^\ell\|L_{2t}(Z)\|_{q_i}^2}\right]^{1/2}\]
for some constants $C,\theta>0$.
This is finite by Lemma \ref{Lemma: Bridge Self-Intersection}.
We therefore conclude that \eqref{Equation: Dominate Function} is dominated by
an integrable function for all $\eps>0$;
hence \eqref{Equation: Semigroup L2 in Expectation} holds.

\begin{remark}\label{Remark: Optimality of Log Follow-Up}
Arguing as in \eqref{Equation: L2 Convergence Formula}, we have the formula
\[\mbf E[\|\hat K(t)\|_2^2]
=\int_I\Pi_Z(2t;x,x)\mbf E^{x,x}_{2t}\left[\mr e^{\mf A_{2t}(Z)+\frac12\|L_{2t}(Z)\|_\ga^2}\right]\d x.\]
Considering Case 1 for simplicity, it follows from (the reverse) H\"older's inequality that
for every $p>1$, the above is bounded below by
\[\int_{\mbb R}\frac{1}{\sqrt{4\pi t}}\mbf E^{x,x}_{2t}\left[\mr e^{-\langle L_{2t}(B),V/p\rangle}\right]^p
\mbf E^{0,0}_{2t}\left[\mr e^{-\frac{1}{2(p-1)}\|L^a_{2t}(B)\|_\ga^2}\right]^{-(p-1)}\d x\]
for every $x\in\mbb R$. If $V(x)\leq c_1\log(1+|x|)+c_2$ for some $c_1>0$ and large enough $c_2>0$,
then an argument similar to the proof of Lemma \ref{Lemma: Deterministic Compactness}
(using the bound $\log(1+|z+\bar z|)\leq\log(1+|z|)+|\bar z|$ instead of \eqref{Equation: Log Lower Bound})
yields the further lower bound
\[\zeta_t\int_\mbb R(1+|x|)^{-c_1t}\d x\]
for some finite $\zeta_t>0$ that only depends on $t$;
this blows up whenever $t\leq1/c_1$.
Thus, if we do not assume \eqref{Equation: Assumptions}, then
there is always some $t_0>0$ such that $\mbf E[\|\hat K(t)\|_2^2]=\infty$ for all $t\leq t_0$.
Essentially the same argument implies that
$\|\mr e^{-t H}\|_2=\infty$ for all $t\leq t_0$
for the deterministic operator $H$ as well.
\end{remark}

\subsection{Proof of Proposition \ref{Proposition: Semigroup Convergence} Part 2}
\label{Section: Last Semigroup Proposition}

We now prove \eqref{Equation: Semigroup Diagonal in Expectation}.
By Fubini,
\begin{multline*}
\mbf E\bigg[\bigg(\int_I \hat K_\eps(t;x,x)-\hat K(t;x,x)\d x\bigg)^2\bigg]
=\int_{I^2}\mbf E[\hat K_\eps(t;x,x)\hat K_\eps(t;y,y)]\\
-2\mbf E[\hat K_\eps(t;x,x)\hat K(t;y,y)]+\mbf E[\hat K(t;x,x)\hat K(t;y,y)]\d x\dd y.
\end{multline*}
Arguing as in the previous section, the above is seen to be equal to
\begin{multline*}
\int_{I^2}\Pi_Z(t;x,x)\Pi(t;y,y)\mbf E\bigg[\mr e^{\mf A_t(Z^{1;x,x}_t)+\mf A_t(Z^{2;y,y}_t)}\bigg(\mr e^{\frac12\| L_t(Z^{1;x,x}_t)*\rho_\eps+L_t(Z^{2;y,y}_t)*\rho_\eps\|_\ga^2}\\
-2\mr e^{\frac12\| L_t(Z^{1;x,x}_t)*\rho_\eps+L_t(Z^{2;y,y}_t)\|_\ga^2}
+\mr e^{\frac12\| L_t(Z^{1;x,x}_t)+L_t(Z^{2;y,y}_t)\|_\ga^2}\bigg)\bigg]\d x\dd y,
\end{multline*}
where $Z^{1;x,x}_t$ and $Z^{2;y,y}_t$ are independent processes with respective
distributions $Z^{x,x}_t$ and $Z^{y,y}_t$.
At this point, essentially the same argument that we used to prove
\eqref{Equation: Semigroup L2 in Expectation} in the previous section
yields \eqref{Equation: Semigroup Diagonal in Expectation}.

\appendix

\section{Measurability of Kernel}
\label{Appendix: Measurable}

We begin by proving that, in Case 1, for every realization of $\Xi$
as a continuous function,
$(x,y)\mapsto\hat K(t;x,y)$ can be made a Borel measurable function on $\mbb R^2$.

\begin{notation}
Let $\mr C_{[0,t]}$ be the set of continuous functions $f:[0,t]\to\mbb R$, which we
equip with the uniform topology. Let $\mr C=\mr C(\mbb R)$ be the space
of continuous functions $f:\mbb R\to\mbb R$, equipped with the
uniform-on-compacts topology;
and let $\mr C_0=\mr C_0(\mbb R)$ be the space of continuous and compactly supported
functions $f:\mbb R\to\mbb R$, equipped with the uniform topology.
We use $\mbf P^{0,0}_t$ to denote the probability measure
of the Brownian bridge on $\mr C_{[0,t]}$, and assume
that $\mr C$ is equipped with the probability measure of $\Xi$.
\end{notation}

By Fubini's theorem, it suffices to prove that there exists a measurable map
\[F:\mbb R^2\otimes\mr C_{[0,t]}\otimes\mr C\to\mbb R\]
such that for every $(x,y)\in\mbb R^2$, $\om\in\mr C_{[0,t]}$, and
$\bar\om\in\mr C$, we can interpret
\begin{align}
\label{Equation: Measurable Kernel}
\mr e^{-\langle L_t(B^{x,y}_t),V\rangle-\xi(L_t(B^{x,y}_t))}=F\big((x,y),\omega,\bar\omega\big)
\end{align}
(here, $\bar\om\in\mr C$ corresponds to a realization of $\Xi$,
and $\big((x,y),\om\big)\in\mbb R^2\otimes\mr C_{[0,t]}$ corresponds to a realization
of the Brownian bridge $B^{x,y}_t$ with deterministic endpoints $x$ and $y$ and random dynamics
given by the Brownian path $B^{0,0}_t$).
Indeed, if this holds, then for every realization of the noise $\bar\om\in\mr C$,
we can define the Borel measurable function
\[\hat K(t;x,y):=\int_{\mr C_{[0,t]}}\Pi_B(t;x,y)\,F\big((x,y),\om,\bar\om\big)~\dd\mbf P^{0,0}_t(\om),\qquad x,y\in\mbb R,\]
which corresponds to the expected value of $\Pi_B(t;x,y)\,\mr e^{-\langle L_t(B^{x,y}_t),V\rangle-\xi(L_t(B^{x,y}_t))}$
given $\Xi$.

Given a realization $\om\in\mr C_{[0,t]}$ of $B^{0,0}_t$ and $x,y\in\mbb R$,
we can construct a realization of $B^{x,y}_t$ by using the
measurable map
$F_1:=\mbb R^2\otimes\mr C_{[0,t]}\to\mr C_{[0,t]}$ defined as
\[F_1\big((x,y),\om\big):=\left(\om(s)+\tfrac{(t-s)x}{t}+\tfrac{sy}{t}:0\leq s\leq t\right).\]

Next, we let $F_2:\mr C_{[0,t]}\to\mr C_0$
be the measurable function that maps $\om$ to its (continuous) local time. More
precisely, let $E\subset\mr C_{[0,t]}$ be the event on which the limit
\[L^a_t(\om):=\lim_{\eps\to0}\frac{1}{\eps}\int_0^t\mbf 1_{\{a\leq\om(s)<a+\eps\}}\d s\]
exists and is finite for all $a\in\mbb R$, and the resulting function $a\mapsto L^a_t(\om)$
is an element of $\mr C_0$. (We know from \cite[Chapter VI, Corollaries 1.8 and 1.9]{RevuzYor}
that $E$ has probability one under the law of $B_t^{x,y}$.) Then, for every $\om\in\mr C_{[0,t]}$, we define the function
$F_2(\om)\in\mr C_0$ as
\[\big(F_2(\om)\big)(a):=L_t^a(\om)\mbf 1_{\{\om\in E\}},\qquad a\in\mbb R.\]
(To see that this is measurable, note that $\om\mapsto L_t^a(\om)\mbf 1_{\{\om\in E\}}$
is measurable for every fixed $a\in\mbb R$, and that
the Borel $\si$-algebra on $\mr C_0$ is generated by evaluation maps.)

Finally, let
\[F_3(f,\bar\om)=\int_\mbb Rf(x)~\dd\bar\om(x):=\begin{cases}
\displaystyle\lim_{n\to\infty}F_3^{(n)}(f,\bar\om)&\text{if the limit exists}\\
0&\text{otherwise}
\end{cases}\]
be the limit of the measurable maps
$F_3^{(n)}:\mr C_0\otimes\mr C\to\mbb R$ defined as
\[F_3^{(n)}(f,\bar\om):=\sum_{k=1}^{k(n)}f\big(\tau^{(n)}_k\big)\Big(\bar\om\big(\tau^{(n)}_{k+1}\big)-\bar\om\big(\tau^{(n)}_k\big)\Big),\]
as per Section \ref{Section: Stochastic Integral}/Karandikar \cite{Karandikar}.
We may then define \eqref{Equation: Measurable Kernel} using the compositions
of measurable maps
\[F\big((x,y),\om,\bar\om\big):=\mr e^{-\langle F_2\circ F_1((x,y),\om),V\rangle-F_3\circ F_2\circ F_1((x,y),\om,\bar\om)}.\]

In order to prove that the diagonal $x\mapsto K(t;x,x)$ is Borel measurable,
we apply the same argument, except that $x=y$. Then, in order to prove
the measurability in Cases 2 and 3, we can use the same
argument, except that we add a few additional steps to construct the conditioned
processes
\[\big(B^x~\big|~ B^x(t)\in\{y,-y\}\big)
\qquad\text{or}\qquad
\big(B^x~\big|~ B^x(t)\in 2b\mbb Z\pm y\big),\]
and then use the couplings discussed in Section \ref{Section: Couplings}
to construct $X^{x,y}_t$ and $Y^{x,y}_t$ and their local times from the latter
in the space of continuous and compactly supported functions on $[0,\infty)$
and $[0,b]$, respectively.

\section{Tails of Gaussian Suprema}

Throughout this section, we assume that
$\big(\mf X(\mf x)\big)_{\mf x\in\mf T}$ is a continuous centered Gaussian process
on some index space $\mf T$. We have the following result regarding the behaviour
of the tails of $\mf X$'s supremum.

\begin{theorem}[{\cite[(5.151)]{MarcusRosen}}]\label{Theorem: Gaussian Suprema}
Let us define
\[v^2:=\sup_{\mf x\in\mf T}\mbf E\big[\mf X(\mf x)^2\big]
\qquad\text{and}\qquad
m:=\mbf{Med}\left[\sup_{\mf x\in\mf T}\mf X(\mf x)\right],\]
where $\mbf{Med}$ denotes the median.
It holds that
\[\mbf P\left[\sup_{\mf x\in\mf T}\mf X(\mf x)\geq t\right]\leq1-\Phi\big((t-m)/v\big)\leq\mr e^{-(t-m)^2/2v^2},\qquad t\geq0,\]
where $\Phi$ denotes the standard Gaussian CDF.
\end{theorem}

Using this Gaussian tails result, we can control the asymptotic growth
of functions involving Gaussian Suprema.

\begin{corollary}\label{Corollary: Stationary Noise Suprema}
Let $\mf T=\mbb R$, and suppose that $\mf X$ is stationary.
There exists a finite random variable $C>0$
such that, almost surely,
\[|\mf X(x)|\leq C\sqrt{\log(2+|x|)},\qquad x\in\mbb R.\]
\end{corollary}
\begin{proof}
For every $n\in\mbb Z\setminus\{0\}$ and $c>0$, define the events
\begin{multline*}
E^{(c)}_n:=\left\{\sup_{x\in[n,n+1]}|\mf X(x)|\geq c\sqrt{\log|n|}\right\}\\
=\left\{\sup_{x\in[n,n+1]}\mf X(x)\geq c\sqrt{\log|n|}\right\}\cup\left\{\sup_{x\in[n,n+1]}-\mf X(x)\geq c\sqrt{\log|n|}\right\}.
\end{multline*}
By the Borel-Cantelli lemma, it suffices to prove that $\sum_{n}\mbf P[E^{(c)}_n]<\infty$
for a large enough $c>0$.
Since $\mf X$ is stationary, for every $n$, it holds that
\[\sup_{x\in[n,n+1]}\mbf E\big[\mf X(x)^2\big]=\mbf E[\mf X(0)^2]=:\si^2\]
and
\[\mbf{Med}\left[\sup_{x\in[n,n+1]}\mf X(x)\right]=\mbf{Med}\left[\sup_{x\in[0,1]}\mf X(x)\right]=:\mu;\]
the same holds true for $-\mf X$.
Thus, by applying Theorem \ref{Theorem: Gaussian Suprema}
to the suprema of $\mf X$ and $-\mf X$ on $[n,n+1]$ and a union bound,
$\mbf P[E^{(c)}_n]\leq2\exp\big((c\sqrt{\log |n|}-\mu)^2/2\si^2\big)$.
Since this is summable in $n$ for large enough $c>0$, the result is proved.
\end{proof}

\begin{remark}\label{Remark: Stronger Stationary Noise Suprema}
By examining the proof of Corollary \ref{Corollary: Stationary Noise Suprema},
we note that we can easily also prove the stronger statement that, almost surely,
\[\sup_{y\in[x,x+1]}|\mf X(y)|\leq C\sqrt{\log(2+|x|)},\]
since
\[\sup_{y\in[x,x+1]}|\mf X(y)|\leq\sup_{y\in[\lfloor x\rfloor,\lfloor x\rfloor+1]}|\mf X(y)|
+\sup_{y\in[\lfloor x\rfloor+1,\lfloor x\rfloor+2]}|\mf X(y)|.\]
\end{remark}

\begin{remark}\label{Remark: Optimality of Stationary Noise Suprema}
In the setting of Corollary \ref{Corollary: Stationary Noise Suprema},
if we also assume that
\[\lim_{|x|\to\infty}\mbf E\big[\mf X(0)\mf X(x)\big]=0,\]
then we can prove that the upper bound in Corollary \ref{Corollary: Stationary Noise Suprema}
is optimal, in the sense that we also have a matching lower bound of the form
\[\sup_{|y|\leq x}|\mf X(y)|\geq \tilde C\sqrt{\log(2+|x|)},\qquad x\in\mbb R\]
for some $0<\tilde C<C$
(see, e.g., \cite[Section 2.1]{CarmonaMolchanov} and references therein).
\end{remark}

\section{Schr\"odinger Operator Theory}
\label{Appendix: Operator}

\subsection{Proof of Lemma \ref{Lemma: Pointwise Form Bound}}

Consider first Cases 1 and 2. Since $f$ is continuous and square-integrable, $f(x)\to0$ as $x\to\infty$,
and thus
\[f(x)^2\leq\int_x^\infty\big|\big(f(y)^2\big)'\big|\d y\leq2\int_I|f(y)|\,|f'(y)|\d y.\]
The result then follows from the fact that for every $\ka>0$, we have the inequality
$|z\bar z|\leq\tfrac\ka2 z^2+\tfrac1{2\ka} \bar z^2$. Suppose then that we are in Case 3.
Define the function
\[h(x):=\begin{cases}
1&\text{if }x\in[0,b/2],\\
1-\frac{x-b/2}{b/2}&\text{if }x\in(b/2,b].
\end{cases}\]
Then, for every $x\in[0,b/2]$, one has
\begin{multline*}
f(x)^2
=f(x)^2h(x)\leq\int_x^b\big|\big(f(y)^2h(y)\big)'\big|\d y\\
\leq2\int_I|f(y)f'(y)|\,h(y)\d y+\int_If(y)^2|h'(y)|\d y.
\end{multline*}
Since $h\leq 1$ and $|h'|\leq2/b$, the same inequality used in Cases 1 and 2 yields the result.
To prove the bound for $x\in(b/2,b]$, we apply the same method with the function
\[h(x):=\begin{cases}
\frac{2x}b&\text{if }x\in[0,b/2],\\
1&\text{if }x\in(b/2,b].
\end{cases}\]

\subsection{Proof of Proposition \ref{Proposition: Classical Schrodinger Operator}}

\subsubsection{Step 1. Norm Equivalence and $\mc E$ is Semibounded}

We begin by proving that $\|\cdot\|_{+1}$ and $\|\cdot\|_*$ are equivalent
and that $\mc E$ is semibounded. In Cases 1, 2-D, and 3-D, it suffices to observe that, because $V\geq0$, one has
\[\mc E(f,f)=\tfrac12\|f'\|_2^2+\|V^{1/2}f\|_2^2\geq0.\]
In the other cases, where $\mc E(f,f)$ contains boundary terms of the form $-\al f(0)^2$ and $-\be f(b)^2$,
we get that $\mc E$ is semibounded and the equivalence of norms from Lemma \ref{Lemma: Pointwise Form Bound}.

\subsubsection{Step 2. $\mc E$ is Closed}

Knowing that $\|\cdot\|_{+1}$ and $\|\cdot\|_*$ are equivalent,
to prove that $\mc E$ is closed,
it suffices to show that $\big(D(\mc E),\langle\cdot,\cdot\rangle_*\big)$ is a Hilbert space.
This follows from the fact that Sobolev spaces and the $L^2$ space with measure $V(x)\dd x$ are complete,
noting further that $\|g'_n-g'\|_2\to0$ as $n\to\infty$ for a sequence $(g_n)_{n\in\mbb N}\subset D(\mc E)$
and $g\in\mr H^1_V$
implies by the fundamental theorem of calculus
that $g_n\to g$ pointwise. Hence in cases 2-D, 3-D, and 3-M, boundary
conditions of the form $g_n(0)=0$ and $g_n(b)=0$ are preserved in the limit $g$.

\subsubsection{Step 3. Form Core}

By equivalence of $\|\cdot\|_{+1}$ and $\|\cdot\|_*$,
to prove that $\mr{FC}$ is a form core for $\mc E$, it suffices
to show that $\mr{FC}$ is dense in $\big(D(\mc E),\langle\cdot,\cdot\rangle_*\big)$.

We begin by noting that it suffices to prove the result in Cases 1,
2-D, and 3-D. To illustrate this, consider Case 2-R: Let $\tilde I:=(-1,\infty)$,
and define $\tilde V:\tilde I\to[0,\infty)$ as $\tilde V(x)=0$ for $x\in(-1,0]$
and $\tilde V(x)=V(x)$ for $x\in (0,\infty)$. If the result is proved in Case 2-D, then
we know that for every locally absolutely continuous $\tilde f:\tilde I\to\mbb R$ such that
\begin{align}
\label{Equation: Extension of AC}
\tilde f(-1)=0
\qquad\text{and}\qquad
\int_{-1}^\infty \tilde f'(x)^2+\big(\tilde V(x)+1\big)\tilde f(x)^2\d x<\infty,
\end{align}
there exists a sequence
$(\tilde\phi_n)_{n\in\mbb N}$ of smooth functions compactly supported in $\tilde I$ such that
\[\lim_{n\to\infty}\int_{-1}^\infty\big(\tilde f'(x)-\tilde\phi_n'(x)\big)^2+\big(\tilde V(x)+1\big)\big(\tilde f(x)-\tilde\phi_n(x)\big)^2\d x=0.\]
We then get the result for Case 2-R by noting that the restriction of $\tilde\phi_n$ to $(0,\infty)$ is
an element of $\mr{FC}$, and that every function $f\in D(\mc E)$ can be extended to an $\tilde f$ of the form
\eqref{Equation: Extension of AC}. A similar extension argument can be used in Cases 3-R and 3-M.

Next, we argue that it suffices to prove the result in Case 3-D. We illustrate this in Case 2-D:
Let $\psi$ be a smooth cutoff function such that $\psi(x)=1$ for $x\in(0,1/2]$ and
$\psi(x)=0$ for $x\geq1$. Then, for every $R>0$, we let $\psi_R(x):=\psi(x/R)$.
Given that $\psi_R'(x)^2=(1/R)^2\psi'(x/R)^2\to0$ and $\big(\psi_R(x)-1\big)^2\to0$ pointwise in $x\in(0,\infty)$
as $R\to\infty$, for every $f\in D(\mc E)$, it is easy to check that
$\|\psi_Rf-f\|_*\to0$ as $R\to0$ by dominated convergence.
Next, since $\mr{supp}(\psi_Rf)$ is compact, if the result
holds in Case 3-D, then we can find a smooth $\phi_R:(0,\infty)\to\mbb R$ with
$\mr{supp}(\phi_R)\subset\mr{supp}(\psi_Rf)$ such that $\|\phi_R-\psi_Rf\|_*<1/R$. Taking
$R\to\infty$ then yields the result in Case 2-D; a similar cutoff argument holds for Case 1.

It now only remains to prove the result in Case 3-D. By \cite[Lemma 7.1.1]{Davies},
we know that, in Case 3-D, for every $f\in D(\mc E)$, there exists a sequence
$(\phi_n)_{n\in\mbb N}\subset\mr{FC}$ such that $\|\phi_n'-f'\|_2\to0$ and $\|\phi_n-f\|_2\to0$
as $n\to\infty$. Since $\sup_n\|\phi_n'\|_2<\infty$ and $\sup_n\|\phi_n\|_2<\infty$, it follows
from Lemma \ref{Lemma: Pointwise Form Bound} and \eqref{Equation: Uniformly Equicontinuous}
that $(\phi_n)_{n\in\mbb N}$ is uniformly bounded and equicontinuous, hence we get
$\|(\phi_{n_i}-f)V^{1/2}\|_2\to0$ along some subsequence $(n_i)_{i\in\mbb N}$ by Arzel\`a-Ascoli.
Therefore, $\|\phi_{n_i}-f\|_*\to0$ as $i\to\infty$, concluding the proof.

\subsubsection{Step 4. Unique Form for $H$ and Compact Resolvent}

Since $\mc E$ is closed and semibounded on $D(\mc E)$,
the fact that $H$ is the unique operator with form $\mc E$
follows from \cite[Theorem VIII.15]{ReedSimonI}.
It only remains to prove that $H$ has compact resolvent:
In Case 3, this follows from the fact that $ H$ is in this case a regular Sturm-Liouville
operator, and in Cases 1 and 2, from the fact that $V(x)\gg\log|x|$ as $x\to\pm\infty$:
Indeed, $ H$ is in those cases limit point (e.g., \cite[Chapter 7.3 and Theorem 7.4.3]{Zettl}),
and compactness of the resolvent is given by
 \cite[Theorem XIII.67]{ReedSimonIV} or the {\it Molchanov criterion} as stated in
\cite[Page 213]{Zettl}.

\section{Proof of Theorem \ref{Theorem: Deterministic Feynman-Kac}}
\label{Appendix: F-K}

\subsection{Proof of \eqref{Equation: Deterministic Symmetry Property}}

On the one hand, since the Gaussian kernel $\ms G_t$ is even,
the transition kernels satisfy $\Pi_Z(t;x,y)=\Pi_Z(t;y,x)$ for every $t>0$
and $x,y\in I$. On the other hand, given that
\[\big(Z^{x,y}_t(t-s):0\leq s\leq t\big)\deq \big(Z^{y,x}_t(s):0\leq s\leq t\big)\]
$\mbf E^{x,y}_t[F(Z)]=\mbf E^{y,x}_t[F(Z)]$
for any path functional $F$ that is invariant under time reversal.
In particular,
since local time is invariant under time reversal,
we have \eqref{Equation: Deterministic Symmetry Property}.

\subsection{Proof of \eqref{Equation: Deterministic Semigroup Property}}
\label{Section: Semigroup Property}

For every $x,y,z\in I$ and $t,\bar t>0$,
if we condition the path $Z^{x,y}_{t+\bar t}$ on $Z^{x,y}_{t+\bar t}(t)=z$, then the path segments
\begin{align}
\label{Equation: Segments Appendix}
\big(Z^{x,y}_{t+\bar t}(s):0\leq s\leq t\big)
\qquad\text{and}\qquad
\big(Z^{x,y}_{t+\bar t}(t+s):0\leq s\leq \bar t\big)
\end{align}
are independent and have respective distributions
$Z^{x,z}_t$ and $Z^{z,y}_{\bar t}$. Moreover, by Doob's $h$-transform,
$Z^{x,y}_{t+\bar t}(t)$ has density
\[z\mapsto\frac{\Pi_Z(t;x,z)\Pi_Z(\bar t;z,y)}{\Pi_Z(t+\bar t;x,y)}.\]
\eqref{Equation: Deterministic Semigroup Property} is then
a consequence of Fubini's theorem and additivity of local time:
Letting $Z^{1;x,z}_t$ and $Z^{2;z,y}_{\bar t}$
denote independent processes with respective distributions $Z^{x,z}_t$
and $Z^{z,y}_{\bar t}$ for all $z$, we have that $\int_IK(t;x,z)K(\bar t;z,y)\d z$
is equal to (recall the notation $\mf A_t$ from \eqref{Equation: Mathfrak A,B Shortcut})
\begin{align*}
&\Pi_Z(t+\bar t;x,y)\int_I\mbf E\bigg[\mr e^{\mf A_t(Z^{1;x,z}_t)+\mf A_{\bar t}(Z^{2;z,y}_{\bar t})}\bigg]
\frac{\Pi_Z(t;x,z)\Pi_Z(\bar t;z,y)}{\Pi_Z(t+\bar t;x,y)}\d z\\
&=\Pi_Z(t+\bar t;x,y)\int_I\mbf E\bigg[\mr e^{\mf A_{t+\bar t}(Z^{x,y}_{t+\bar t})}\bigg|Z^{x,y}_{t+\bar t}(t)=z\bigg]
\mbf P\big[Z^{x,y}_{t+\bar t}(t)\in\d z\big]\d z\\
&=K(t+\bar t;x,y),
\end{align*}
as desired.

\subsection{Feynman-Kac Formula}

We now complete the proof of Theorem \ref{Theorem: Deterministic Feynman-Kac} by
showing that $\mr e^{-t H}=K(t)$ for all $t>0$.
The proof of this in Case 1 can be found in \cite[Theorem 4.9]{Sznitman}.
For Cases 2-D and 3-D, we refer to \cite[(34) and Theorem 3.27]{ChungZhao}.
For Case 3-R, we have \cite[(3.3') and (3.4), Theorem 3.4 (b), and Lemmas 4.6 and 4.7]{Papanicolaou}.
It now only remains to prove the result in cases Cases 3-M and 2-R:

\subsubsection{Case 3-M}

Let us assume that we are considering Case 3-M, that is,
the operator $ H=-\tfrac12\De+V$ is acting on $(0,b)$ with mixed boundary conditions
(as in Assumption \ref{Assumption: DB}) and
\[ K(t;x,y)=\Pi_Y(t;x,y)\,\mbf E^{x,y}_t\Big[\mr e^{-\langle L_t(Y),V\rangle+\al \mf L^0_t(Y)-\infty\cdot \mf L^b_t(Y)}\Big].\]
As argued in \cite[Pages 62 and 63]{Papanicolaou}, it can be shown that
\begin{enumerate}
\item $ K(t)$ is a strongly continuous semigroup on $L^2$; and
\item if, for every $n\in\mbb N$, we define
\[K_n(t;x,y)=\Pi_Y(t;x,y)\,\mbf E^{x,y}_t\Big[\mr e^{-\langle L_t(Y),V\rangle+\al \mf L^0_t(Y)-n \mf L^b_t(Y)}\Big],\]
then for every $t>0$, $\|K_n(t)- K(t)\|_{\mr{op}}\to0$ as $n\to\infty$.
\end{enumerate}
Item (1) above implies that $ K(t)$ has a generator, so it only remains to prove that this generator is in fact $ H$.
By Lemma \ref{Lemma: Deterministic Compactness} in the case $p=1$, we know that the $K_n(t)$ and $ K(t)$
are compact. Therefore, if we let $ H_n$ be the operator $-\tfrac12\De+V$ on $(0,b)$ with Robin boundary
\[f'(0)+\al f(0)=-f'(b)-nf(0)=0,\]
then by repeating the argument in Section \ref{Section: Main Outline Part 3}, we need only prove that
$ H_n\to  H$ in the sense of convergence of eigenvalues and $L^2$-convergence of
eigenfunctions.

If we define the matrices
\[A:=\left[\begin{array}{cc}1&\al\\0&0\end{array}\right]
\qquad\text{and}\qquad
B:=\left[\begin{array}{cc}0&0\\0&1\end{array}\right]\]
and the vector function $F(x):=[f'(x),f(x)]^\top$,
then we can represent $ H$'s boundary conditions in matrix form as
$AF(0)+B F(b)=0$.
Similarly, if we let
\[C_n:=\left[\begin{array}{cc}0&0\\1/n&1\end{array}\right],\]
then $ H_n$'s boundary conditions are represented as $AF(0)+C_nF(b)=0$.
Given that $\|B-C_n\|\to0$ as $n\to\infty$, it follows from \cite[Theorems 3.5.1 and 3.5.2]{Zettl}
that for every $k\in\mbb N$, $\la_k( H_n)\to\la_k( H)$ and $\psi_k( H_n)\to\psi_k( H)$ uniformly on compacts.
Since $(0,b)$ is bounded, this implies $L^2$-convergence of the eigenfunctions, concluding the proof.

\subsubsection{Case 2-R}

Let us now assume that $ H$ acts on $(0,\infty)$ with Robin boundary at the origin and that
\[ K(t;x,y)=\Pi_X(t;x,y)\,\mbf E^{x,y}_t\Big[\mr e^{-\langle L_t(X),V\rangle+\al \mf L^0_t(X)}\Big].\]
The same arguments used in \cite[Theorem 3.4 (b)]{Papanicolaou} imply that this semigroup
is strongly continuous on $L^2$, and we know it is compact by Lemma \ref{Lemma: Deterministic Compactness}.

For every $n\in\mbb N$, let $ H_n=-\tfrac12\De+V$, acting on $(0,n)$ with mixed boundary conditions
\[f(0)+\al f'(0)=f(n)=0.\]
By the previous section, the semigroup generated by this operator is given by
\[K_n(t;x,y)=\Pi_{Y_n}(t;x,y)\,\mbf E^{x,y}_t\Big[\mr e^{-\langle L_t(Y_n),V\rangle+\al \mf L^0_t(Y_n)-\infty\cdot\mf L^n_t(Y_n)}\Big],\]
where $Y_n$ is a reflected Brownian motion on $(0,n)$. Arguing as in the previous section, it suffices to prove that
$K_n(t)\to  K(t)$ in operator norm and $ H_n\to  H$ in the sense of eigenvalues and eigenfunctions.

We begin with the semigroup convergence. We first note that $\|K_n(t)- K(t)\|_{\mr{op}}$
is ambiguous, since $K_n(t)$ and $ K(t)$ do not act on the same space.
However, by using an argument similar to \eqref{Equation: Reflected Motion 1},
we can extend the kernel $K_n(t)$ to $(0,\infty)^2$ by defining
\[\tilde K_n(t;x,y)=\Pi_{X}(t;x,y)\,\mbf E^{x,y}_t\Big[\mbf 1_{\{\tau_{[n,\infty)}(X)>t\}}\mr e^{-\langle L_t(X),V\rangle+\al \mf L^0_t(X)}\Big],\]
where $\tau_{[n,\infty)}$ is the first hitting time of $[n,\infty)$.
This transformation does not affect the eigenvalues, and the eigenfunctions are similarly extended from functions on $(0,n)$
vanishing on the boundary to functions on $(0,\infty)$ that are supported on $(0,n)$. One has
\[\|\tilde K_n(t)- K(t)\|_2^2=\int_0^\infty \tilde K_n(2t;x,x)-2\tilde K_{n,0}(2t;x,x)+ K(2t;x,x)\d x,\]
where
\[\tilde K_{n,0}(2t;x,x)=\Pi_{X}(2t;x,x)\,\mbf E^{x,x}_{2t}\Big[\mbf 1_{\{\tau_{[n,\infty)}(X)>t\}}\mr e^{-\langle L_{2t}(X),V\rangle+\al\mf L^0_{2t}(X)}\Big].\]
Thus it suffices to prove that
\[\lim_{n\to\infty}\int_0^\infty\tilde K_{n,0}(2t;x,x)\d x,\lim_{n\to\infty}\int_0^\infty\tilde K_n(2t;x,x)\d x=\int_0^\infty  K(2t;x,x)\d x.\]
Since $X^{x,x}_{2t}$ is almost surely continuous, hence bounded, the result is a straightforward application of monotone convergence
(both with $\mbf E^{x,x}$ and the $\dd x$ integral).

We now prove convergence of eigenvalues and eigenvectors.
Let $\mc E$ denote the form of $ H$ and $D(\mc E)$ its domain, as defined in Definition \ref{Definition: Forms}
for Case 2-R. We note that we can think of $ H_n$ as the operator with the same form
$\mc E$ but acting on the smaller domain
\[D_n:=\big\{f\in \mr H^1_V\big((0,\infty)\big):f(x)=0\text{ for every }x\geq n\big\}\subset D(\mc E).\]
These domains are increasing, in that $D_1\subset D_2\subset \cdots\subset D(\mc E)$.
A straightforward modification of the convergence argument presented in Section \ref{Section: Proof of Operator Convergence}
gives the desired result (at least through a subsequence).

%%%%%%%%%%%%%%%%%%%%%%%%%%%%%%%%%%%%%%%%%%%%%%%%%%%%%%%%%%%%%%%%%%%
%%                                                               %%
%% Use the two commands below for producing your bibliography    %%
%% with bibtex, then comment again the commands and include the  %%
%% content of the .bbl file in this file below the commands.     %%
%%                                                               %%
%%%%%%%%%%%%%%%%%%%%%%%%%%%%%%%%%%%%%%%%%%%%%%%%%%%%%%%%%%%%%%%%%%%

%\bibliographystyle{amsplain}
%\bibliography{bibliography}

% add below the content of your .bbl file produced by bibtex.

%%%%%%%%%%%%%%%%%%%%%%%%%%%%%%%%%%%%%%%%%%%%%%%%%%%%%%%%%%%%%%%%%%%
%%                                                               %%
%% You may add acknowledgments (optional).                       %%
%%                                                               %%
%%%%%%%%%%%%%%%%%%%%%%%%%%%%%%%%%%%%%%%%%%%%%%%%%%%%%%%%%%%%%%%%%%%

\ACKNO{The author thanks Laure Dumaz for an insightful discussion on the one-dimensional Anderson Hamiltonian
and the parabolic Anderson model, which served as a chief
motivation for the writing of this paper.
The author thanks Michael Aizenman for helpful pointers in the literature regarding random Schr\"odinger
operators.
The author gratefully acknowledges Mykhaylo Shkolnikov for
his continuous guidance and support and for his help regarding a few technical obstacles in the proofs of this paper,
as well as Vadim Gorin and Mykhaylo Shkolnikov for discussions concerning the resolution of an error that
appeared in a previous version of the paper.

The author thanks anonymous referees for carefully reading several previous versions of this paper,
as well as a number of insightful comments that helped significantly improve
the presentation of the present version.}

%%%%%%%%%%%%%%%%%%%%%%%%%%%%%%%%%%%%%%%%%%%%%%%%%%%%%%%%%%%%%%%%%%%
%%                                                               %%
%% You have reached the end of your document.                    %%
%%                                                               %%
%%%%%%%%%%%%%%%%%%%%%%%%%%%%%%%%%%%%%%%%%%%%%%%%%%%%%%%%%%%%%%%%%%%

\end{document}